\documentclass[11pt,twoside,english,headsepline,abstracton,openright]{scrreprt}
\usepackage[T1]{fontenc}
\usepackage[latin9]{inputenc}
\usepackage[a4paper]{geometry}
\usepackage{babel}
\usepackage{refstyle}
\usepackage{wrapfig}
\usepackage{amsthm}
\usepackage{amsmath}
\usepackage{amssymb}
\usepackage{esint}
\usepackage[authoryear]{natbib}
\usepackage{lmodern}
\usepackage{graphics}
\usepackage[kerning=true]{microtype}
\usepackage{longtable}
\usepackage{booktabs}
\usepackage{makeidx}
\usepackage[totoc]{idxlayout}
\usepackage[unicode=true,pdfusetitle,
 bookmarks=true,bookmarksnumbered=true,bookmarksopen=false,
 breaklinks=false,pdfborder={0 0 0},backref=page,colorlinks=true]
 {hyperref}

\makeindex

\makeatletter

\theoremstyle{plain}
\newtheorem{thm}{Theorem}[chapter]
\newtheorem{lem}[thm]{Lemma}
\newtheorem{prop}[thm]{Proposition}
\newtheorem{conjecture}[thm]{Conjecture}
\newtheorem{cor}[thm]{Corollary}

\theoremstyle{definition}
\newtheorem{defn}[thm]{Definition}
\newtheorem{example}[thm]{Example}

\theoremstyle{remark}
\newtheorem{rem}[thm]{Remark}

\newcommand{\includefigure}[1]{\centering\includegraphics{figures/#1}}

\setcapindent{0pt}

\renewcommand*{\backrefalt}[4]{%
\ifcase #1 %
[no citations]%
\or
[see p.~#2]%
\else
[see pp.~#2]%
\fi
}

\global\long\def\i{\mathrm{i}}
\global\long\def\e{\mathrm{e}}
\global\long\def\rd{\mathrm{d}}

\global\long\def\sech{\operatorname{sech}}

\global\long\def\dist{\operatorname{dist}}
\global\long\def\bu{\boldsymbol{u}}
\global\long\def\bv{\boldsymbol{v}}

\global\long\def\bR{\boldsymbol{R}}

\global\long\def\bx{\boldsymbol{x}}
\global\long\def\by{\boldsymbol{y}}

\global\long\def\be{\boldsymbol{e}}
\global\long\def\bff{\boldsymbol{f}}
\global\long\def\bk{\boldsymbol{k}}
\global\long\def\bF{\boldsymbol{F}}
\global\long\def\bS{\boldsymbol{S}}
\global\long\def\bN{\boldsymbol{N}}
\global\long\def\bR{\boldsymbol{R}}
\global\long\def\bC{\boldsymbol{C}}
\global\long\def\bn{\boldsymbol{n}}
\global\long\def\ba{\boldsymbol{a}}

\global\long\def\bphi{\boldsymbol{\varphi}}

\global\long\def\bUF{{\boldsymbol{U}\!}_{\bF}}
\global\long\def\bUmF{{\boldsymbol{U}\!}_{-\bF}}
\global\long\def\bnabla{\boldsymbol{\nabla}}
\global\long\def\bcdot{\boldsymbol{\cdot}}
\global\long\def\bwedge{\boldsymbol{\wedge}}
\global\long\def\bzero{\boldsymbol{0}}

\global\long\def\bA{\boldsymbol{A}}
\global\long\def\bsigma{\boldsymbol{\Sigma}}
\global\long\def\argmax{\operatorname*{arg\, max}}
\global\long\def\widebar#1{\mybar{#1}}

\usepackage{mathtools}

\makeatletter

\let\save@mathaccent\mathaccent
\newcommand*\if@single[3]{%
  \setbox0\hbox{${\mathaccent"0362{#1}}^H$}%
  \setbox2\hbox{${\mathaccent"0362{\kern0pt#1}}^H$}%
  \ifdim\ht0=\ht2 #3\else #2\fi
  }
\newcommand*\rel@kern[1]{\kern#1\dimexpr\macc@kerna}
\newcommand*\mybar[1]{\@ifnextchar^{{\wide@bar{#1}{0}}}{\wide@bar{#1}{1}}}
\newcommand*\wide@bar[2]{\if@single{#1}{\wide@bar@{#1}{#2}{1}}{\wide@bar@{#1}{#2}{2}}}
\newcommand*\wide@bar@[3]{%
  \begingroup
  \def\mathaccent##1##2{%
    \let\mathaccent\save@mathaccent
    \if#32 \let\macc@nucleus\first@char \fi
    \setbox\z@\hbox{$\macc@style{\macc@nucleus}_{}$}%
    \setbox\tw@\hbox{$\macc@style{\macc@nucleus}{}_{}$}%
    \dimen@\wd\tw@
    \advance\dimen@-\wd\z@
    \divide\dimen@ 3
    \@tempdima\wd\tw@
    \advance\@tempdima-\scriptspace
    \divide\@tempdima 10
    \advance\dimen@-\@tempdima
    \ifdim\dimen@>\z@ \dimen@0pt\fi
    \rel@kern{0.6}\kern-\dimen@
    \if#31
      \overbracket[0.65pt][0pt]{\rel@kern{-0.6}\kern\dimen@\macc@nucleus\rel@kern{0.4}\kern\dimen@}%
      \advance\dimen@0.1\dimexpr\macc@kerna
      \let\final@kern#2%
      \ifdim\dimen@<\z@ \let\final@kern1\fi
      \if\final@kern1 \kern-\dimen@\fi
    \else
       \overbracket[0.65pt][0pt]{\rel@kern{-0.6}\kern\dimen@#1}%
    \fi
  }%
  \macc@depth\@ne
  \let\math@bgroup\@empty \let\math@egroup\macc@set@skewchar
  \mathsurround\z@ \frozen@everymath{\mathgroup\macc@group\relax}%
  \macc@set@skewchar\relax
  \let\mathaccentV\macc@nested@a
  \if#31
    \macc@nested@a\relax111{#1}%
  \else
    \def\gobble@till@marker##1\endmarker{}%
    \futurelet\first@char\gobble@till@marker#1\endmarker
    \ifcat\noexpand\first@char A\else
      \def\first@char{}%
    \fi
    \macc@nested@a\relax111{\first@char}%
  \fi
  \endgroup
}

\makeatother

\begin{document}

\title{Steady solutions of the\\
Navier-Stokes equations\\
in the plane}

\date{November 12, 2015}

\author{\href{mailto:guill093@umn.edu}{Julien Guillod}}

\publishers{School of Mathematics\\
University of Minnesota}
\maketitle
\begin{abstract}
This study is devoted to the incompressible and stationary Navier-Stokes
equations in two-dimensional unbounded domains. First, the main results
on the construction of the weak solutions and on their asymptotic
behavior are reviewed and structured so that all the cases can be
treated in one concise way. Most of the open problems are linked with
the case of a vanishing velocity field at infinity and this will be
the main subject of the remainder of this study. The linearization
of the Navier-Stokes around the zero solution leads to the Stokes
equations which are ill-posed in two dimensions. It is the well-known
Stokes paradox which states that if the net force is nonzero, the
solution of the Stokes equations will grow at infinity. By studying
the link between the Stokes and Navier-Stokes equations, it is proven
that even if the net force vanishes, the velocity and pressure fields
of the Navier-Stokes equations cannot be asymptotic to those of the
Stokes equations. However, the velocity field can be in some cases
asymptotic to two exact solutions of the Stokes equations which also
solve the Navier-Stokes equations. Finally, a formal asymptotic expansion
at infinity for the solutions of the two-dimensional Navier-Stokes
equations having a nonzero net force is established based physical
arguments. The leading term of the velocity field in this expansion
decays like $\left|\bx\right|^{-1/3}$ and exhibits a wake behavior.
Numerical simulations are performed to validate this asymptotic expansion
when is net force is nonzero and to analyze the asymptotic behavior
in the case where the net force is vanishing. This indicates that
the Navier-Stokes equations admit solutions whose velocity field goes
to zero at infinity in contrast to the Stokes linearization and moreover
this shows that the set of possible asymptotes is very rich.\bigskip{}


\textit{MSC class:} 35Q30, 35J57, 76D05, 76D07, 76D03, 76D25, 76M10
\end{abstract}

\tableofcontents{}

\chapter{Introduction}

We consider a viscous fluid of constant viscosity $\mu$ and constant
density $\rho$ moving in a region $\Omega$ of the two or three-dimensional
space. The motion of the fluid is characterized by the velocity field\index{Velocity field}
$\bu(\bx,t)$ and the pressure field\index{Pressure field} $p(\bx,t)$,
where $\bx\in\Omega$ is the position and $t>0$ the time. In an inertial
frame, the equations of motion are given by
\begin{align}
\rho\left(\frac{\partial\bu}{\partial t}+\bu\bcdot\bnabla\bu\right) & =\mu\Delta\bu-\bnabla p-\rho\bff\,, & \bnabla\bcdot\bu & =\bzero\,,\label{eq:intro-ns-time}
\end{align}
where $\bff$ is minus the external force per unit mass acting on
the fluid. These equations were first described by \citet[p.~414]{Navier-Memoiresurles1827},
but their adequate physical justification was given only later on
in the work of \citet{Stokes-TheoriesInternalfriction1845}. Nowadays,
these equations are referred to as the Navier-Stokes equations. The
resolution of the Navier-Stokes equations consists of finding fields
$\bu$ and $p$ satisfying \eqref{intro-ns-time} together with some
prescribed boundary conditions or initial conditions. The beginning
of mathematical fluid dynamics started with the pioneering work of
\citet{Leray-Etudedediverses1933} who developed a general method
for solving the Navier-Stokes equations essentially without any restriction
on the size of the data. With the usage of computers, the Navier-Stokes
equations can now be solved numerically with good precision in many
cases, which is crucial for applications. However, up to this date,
the Navier-Stokes equations are far from being completely understood
mathematically. One major question is the one stated by the Clay Mathematical
Institute as one of the seven most important open mathematical problems:
do the time-dependent Navier-Stokes equations in an unbounded or periodic
domain of the three-dimensional space admit a solution for large data?
\citet{Ladyzhenskaya-MathematicalTheory1963} answers the same question
affirmatively in two dimensions. A second major question concerns
the steady solutions in two-dimensional unbounded domains, which is
the main subject of this research. For time-independent domains, steady
motions are described by $\partial_{t}\bu=\partial_{t}\bff=\bzero$,
which leads to the following stationary Navier-Stokes equations\index{Navier-Stokes equations},
\begin{align}
\mu\Delta\bu-\bnabla p & =\rho\left(\bu\bcdot\bnabla\bu+\bff\right)\,, & \bnabla\bcdot\bu & =\bzero\,.\label{eq:intro-ns-steady}
\end{align}
Various aspects of these equations have been studied: the monograph
of \citet{Galdi-IntroductiontoMathematical2011} presents them in
great detail. By the change of variables\index{Scaling!Navier-Stokes equations}
\begin{align*}
\bu & \mapsto\frac{\mu}{\rho}\bu\,, & p\mapsto & \frac{\mu^{2}}{\rho}p\,, & \bff & \mapsto\frac{\mu^{2}}{\rho^{2}}\bff\,,
\end{align*}
the parameters $\mu$ and $\rho$ can be set to one,\begin{subequations}
\begin{align}
\Delta\bu-\bnabla p & =\bu\bcdot\bnabla\bu+\bff\,, & \bnabla\bcdot\bu & =0\,,\label{eq:intro-ns-eq}
\end{align}
as we will do from now on. In case the domain $\Omega$ has a boundary
$\partial\Omega$, we complete \eqref{intro-ns-eq} with a condition
that describes how the fluid interacts with the boundary,\index{Boundary condition}
\begin{equation}
\left.\bu\right|_{\partial\Omega}=\bu^{*}\,,\label{eq:intro-ns-bc}
\end{equation}
and if the domain $\Omega$ is unbounded, we add a boundary condition
at infinity,
\begin{equation}
\lim_{\left|\bx\right|\to\infty}\bu(\bx)=\bu_{\infty}\,,\label{eq:intro-ns-limit}
\end{equation}
\label{eq:intro-ns}\end{subequations}where $\bu_{\infty}\in\mathbb{R}^{n}$
is a constant vector. So for a domain $\Omega\subset\mathbb{R}^{n}$,
the stationary Navier-Stokes problem consists of finding $\bu$ and
$p$ satisfying \eqref{intro-ns} for given $\bff$, $\bu^{*}$ and
$\bu_{\infty}$, which are called the data. This research focuses
on the analysis of the existence, uniqueness and asymptotic behavior
of the solutions of this problem in two-dimensional unbounded domains.
The analysis of this problem depends highly on the domain and on the
data.

First, at the end of the introduction, we make some general remarks
on the symmetries and invariant quantities of the Navier-Stokes equations
that will be later on routinely used. Concerning the symmetries, we
show that there are no further infinitesimal symmetries of the stationary
Navier-Stokes equations in $\mathbb{R}^{n}$ beside the Euclidean
group, the scaling symmetry and a trivial shift of the pressure. This
is useful to ensure that there is no hidden symmetries in the stationary
solutions that could have been used otherwise. In the last part of
the introduction, we introduce a concept of invariant quantity and
show that the net flux, the net force, and the net torque are the
only invariant quantities on the Navier-Stokes equations. By definition,
an invariant quantity can be expressed by integration over a closed
curve or surface in $\Omega$ and is independent for any homotopic
change of the curve. In unbounded domains, the invariant quantities
play an important role, because the closed curve can be enlarged to
infinity, and therefore are linked to the asymptotic behavior at infinity
of the solutions. As it will become clear later on, the asymptotic
behavior of the solutions is fundamentally intertwined with the existence
of solutions.\\

The mathematical tools needed to discuss the equations dependent a
lot on the type of the domain $\Omega$, and we distinguish four cases
as shown in \figref{domains}:\index{Domain!type}{\renewcommand\theenumi{(\alph{enumi})} \renewcommand\labelenumi{\theenumi} 
\begin{enumerate}
\item $\Omega$ is bounded;
\item $\Omega$ is unbounded and its boundary $\partial\Omega$ is bounded,
\emph{i.e.} $\Omega$ is an exterior domain;
\item $\Omega$ is unbounded and has no boundary, \emph{i.e.} $\Omega=\mathbb{R}^{n}$;
\item $\Omega$ and $\partial\Omega$ are both unbounded.
\end{enumerate}
}
\begin{figure}[h]
\includefigure{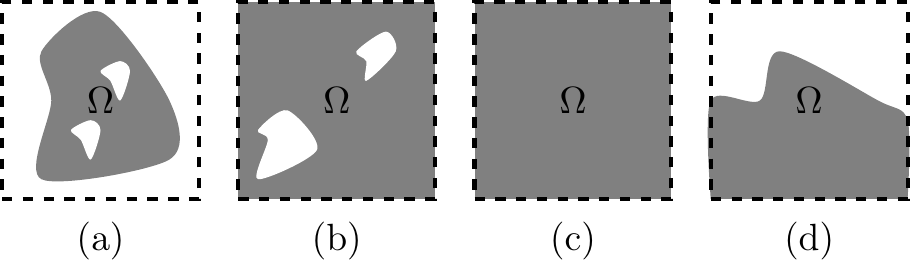}\vspace{-4mm}
\caption{\label{fig:domains}Different families of domains $\Omega$.}
\end{figure}
As already said, the mathematical study of the Navier-Stokes equations
essentially started with the work of \citet{Leray-Etudedediverses1933},
whose method consists of three steps. First the boundary conditions
$\bu^{*}$ and $\bu_{\infty}$ have to be lifted by an extension $\ba$
which satisfies the so-called extension condition. The second step
is to show the existence of weak solutions in bounded domain. Finally
if the domain is unbounded, the third step is to define a sequence
of invading bounded domains that coincide in the limit with the unbounded
domain and show that the induced sequence of solutions converges in
some suitable space. With this strategy, \citet{Leray-Etudedediverses1933}
was able to construct weak solutions in domains with a compact boundary,
\emph{i.e.} cases (a) \& (b), if the flux through each connected component
of the boundary is zero. If $\Omega$ is bounded and in view of the
incompressibility of the fluid, the divergence theorem requires that
the total flux through the boundary $\partial\Omega$ is zero, but
not that the flux through each connected component of the boundary
is zero. If theses fluxes are small enough, the existence of weak
solutions was proved by \citet{Galdi-existencesteadymotions1991}
in bounded domains and respectively in two and three dimensions by
\citet[Theorem 2.6]{Finn-steady-statesolutions1961} and \citet{Russo-NoteExteriorTwo-Dimensional2009}
for the unbounded case (b). Without restriction on the magnitude of
the fluxes, \citet{Korobkov-existence2Dsymmetric2014,Korobkov-existence3Dsymmetric2014}
treated the case of unbounded symmetric exterior domains in both two
and three dimensions and recently, \citet{Korobkov-LerayProblem2015}
proved the existence of weak solutions under no symmetry and smallness
assumptions for two-dimensional bounded domains. In the first chapter,
we review the above results for small fluxes by proposing a method
that includes all the cases in a concise way. In case (c) where $\Omega=\mathbb{R}^{n}$,
the method of Leray work without any differences if $n=3$ but cannot
be used if $n=2$ to construct weak solutions, whose existence is
still an open problem. For the case (d), see \citet{Guillod-NavierStokesHalfplane2015}
and references therein.

If the data are regular enough, \citet{Ladyzhenskaya-InvestigationofNavierStokes1959}
showed by elliptic regularity that the weak solutions satisfy \eqref{intro-ns-eq}
and \eqref{intro-ns-bc} in the classical way, which solves the problem
\eqref{intro-ns} if $\Omega$ is bounded. However, if $\Omega$ is
unbounded, the validity of the boundary condition at infinity \eqref{intro-ns-limit}
depends drastically on the dimension. In three dimensions, the function
space used by Leray, allowed him to show that \eqref{intro-ns-limit}
is satisfied in a weak sense and the existence of uniform pointwise
limit was shown later by \citet{Finn-steadystatesolutions1959}. However,
in two dimensions, the function space used by Leray for the construction
of weak solutions does not even ensure that $\bu$ is bounded at large
distances, so that apparently no information on the behavior at infinity
$\bu_{\infty}$ is retained in the limit where the domain becomes
infinitely large. The validity of \eqref{intro-ns-limit} for two-dimensional
exterior domains remained completely open until \citet{Gilbarg.Weinberger-AsymptoticPropertiesof1974,Gilbarg.Weinberger-Asymptoticpropertiesof1978}
partially answered it by showing that either there exists $\bu_{0}\in\mathbb{R}^{2}$
such that
\[
\lim_{\left|\bx\right|\to\infty}\int_{S^{1}}\left|\bu-\bu_{0}\right|^{2}=0\,,\qquad\text{or}\qquad\lim_{\left|\bx\right|\to\infty}\int_{S^{1}}\left|\bu\right|^{2}=\infty\,.
\]
Nevertheless, the question if the second case of the alternative can
be ruled out and if $\bu_{0}$ coincides with $\bu_{\infty}$ remains
open in general. Later on \citet{Amick-Leraysproblemsteady1988} showed
that if $\bu^{*}=\bff=\bzero$, then the first alternative happens,
so $\bu$ is bounded and
\[
\lim_{\left|\bx\right|\to\infty}\bu=\bu_{0}\,.
\]
In two dimensions, the only results with $\bu_{\infty}=\bzero$ without
assuming small data are obtained by assuming suitable symmetries.
\citet[\S 3.3]{Galdi-StationaryNavier-Stokesproblem2004} showed that
if an exterior domain and the data are symmetric with respect to two
orthogonal axes, then there exists a solution satisfying the boundary
condition at infinity in the following sense:
\[
\lim_{|\bx|\to\infty}\int_{S^{1}}\left|\bu\right|^{2}=0\,.
\]
This result was improved by \citet[Theorem~7]{Russo-existenceofDsolutions2011}
by only requiring the domain and the data to be invariant under the
central symmetry $\bx\mapsto-\bx$, and by \citet{Pileckas-existencevanishing2012}
by allowing a flux through the boundary. However, all these results
rely only on the properties of the subset of symmetric functions in
the function space in which weak solutions are constructed, and therefore
the decay of the velocity at infinity remains unknown.

\Chapref{weak} is a review of the construction of weak solutions
in two- and three-dimensional Lipschitz domains for arbitrary large
data $\bu^{*}$ and $\bff$, provided that the flux of $\bu^{*}$
through each connected component of $\partial\Omega$ is small. The
proofs are based on standard techniques and structured so that all
the cases can be treated in one concise way. For unbounded domains,
the behavior at infinity of the weak solutions is also reviewed.\\

In cases (b) \& (c), more detailed results can be obtained by linearizing
\eqref{intro-ns-eq} around $\bu=\bu_{\infty}$,
\begin{align}
\Delta\bu-\bnabla p-\bu_{\infty}\bcdot\bnabla\bu & =\bff\,, & \bnabla\bcdot\bu & =\bzero\,,\label{eq:intro-ns-lin}
\end{align}
which is called the Stokes equations\index{Stokes equations} if $\bu_{\infty}=\bzero$
and the Oseen equations\index{Oseen equations} if $\bu_{\infty}\neq\bzero$.
The fundamental solution of the Stokes equations behaves like $\left|\bx\right|^{-1}$
in three dimensions and grows like $\log\left|\bx\right|$ in two
dimensions. However, the fundamental solution of the Oseen equations
exhibits a parabolic wake directed in the direction of $\bu_{\infty}$
in which the decay of the velocity is slower than in the other region.
Explicitly in three dimensions the velocity decays like $\left|\bx\right|^{-1}$
inside the wake and like $\left|\bx\right|^{-2}$ outside and in two
dimensions the decays are $\left|\bx\right|^{-1/2}$ and $\left|\bx\right|^{-1}$
respectively inside and outside the wake. In view of these different
behaviors of the fundamental solution at infinity, we distinguish
the two cases $\bu_{\infty}\neq\bzero$ and $\bu_{\infty}=\bzero$.

For $\bu_{\infty}\neq\bzero$, the estimates of the Oseen equations
show that the inversion of the Oseen operator on the nonlinearity
leads to a well-posed problem, so a fixed point argument shows the
existence of solutions behaving at infinity like the Oseen fundamental
solution for small data. This was done by \citet[\S 4]{Finn-exteriorstationaryproblem1965}
in three dimensions and by \citet{Finn-stationarysolutionsNavier1967}
in two dimensions. Moreover, in three dimensions, by using results
of \citet{Finn-exteriorstationaryproblem1965}, \citet{Babenko-stationarysolutionsof1973}
showed that the solution of \eqref{intro-ns} found by the method
of Leray behaves at infinity like the fundamental solution of the
Oseen equations \eqref{intro-ns-lin}, so in particular $\bu-\bu_{\infty}=O(\left|\bx\right|^{-1})$
at infinity. In two dimensions, by the results of \citet[\S 4]{Smith-EstimatesatInfinity1965}
and \citet[Theorem XII.8.1]{Galdi-IntroductiontoMathematical2011},
one has that if $\bu$ is a solution of \eqref{intro-ns}, then $\bu$
is asymptotic to the Oseen fundamental solution, so $\bu-\bu_{\infty}=O(\left|\bx\right|^{-1/2})$.
However, it is still not known if the solutions constructed by the
method of \citet{Leray-Etudedediverses1933} satisfy \eqref{intro-ns-limit}
in two dimensions and therefore if they coincide with the solutions
found by \citet{Finn-stationarysolutionsNavier1967}. These results
on the asymptotic behavior of weak solutions will be reviewed at the
end of \chapref{weak}.

From now one, we consider the case where $\bu_{\infty}=\bzero$. As
already said, in three dimensions, the function spaces imply the validity
of \eqref{intro-ns-limit} even if $\bu_{\infty}=\bzero$, whereas
in two dimensions, all the available results are obtained by assuming
suitable symmetries \citep{Galdi-StationaryNavier-Stokesproblem2004,Yamazaki-stationaryNavier-Stokesequation2009,Yamazaki-Uniqueexistence2011,Pileckas-existencevanishing2012}
or specific boundary conditions \citep{Hillairet-mu2013}. \citet{Yamazaki-Uniqueexistence2011}
showed the existence and uniqueness of solutions for small data in
an exterior domain provided the domain and the data are invariant
under four axes of symmetries with an angle of $\pi/4$ between them.
In the exterior of a disk, \citet{Hillairet-mu2013} proved the existence
of solutions that decay like $\left|\bx\right|^{-1}$ at infinity
provided that the boundary condition on the disk is close to $\mu\be_{r}$
for $\left|\mu\right|>\sqrt{48}$. To our knowledge, these last two
results together with the exact solutions found by \citet{Hamel-SpiralfoermigeBewegungen1917,Guillod-Generalizedscaleinvariant2015}
are the only ones showing the existence of solutions in two-dimensional
exterior domains satisfying \eqref{intro-ns-limit} with $\bu_{\infty}=\bzero$
and a known decay rate at infinity.\\

We now analyze the implications of the decay of the velocity on the
linear and nonlinear terms and on the net force. For simplicity, we
consider in this paragraph the domain $\Omega=\mathbb{R}^{n}$ and
a source force $\bff$ with compact support, but the following considerations
can be extended to the case where $\Omega$ has a compact boundary
and $\bff$ decays fast enough. A fundamental quantity is the net
force $\bF$ which has a simple expression due to the previous hypothesis,
\[
\bF=\int_{\mathbb{R}^{n}}\bff\,.
\]
If the net force is nonzero, the solution of the Stokes equations
has a velocity field that decays like $\left|\bx\right|^{-1}$ for
$n=3$ and that grows like $\log\left|\bx\right|$ for $n=2$. This
is the well-known Stokes paradox.\index{Stokes paradox}\index{Paradox of Stokes}
By power counting\index{Power counting}, if the velocity decays like
$\left|\bx\right|^{-\alpha}$, we have
\begin{align}
\bu & \sim\left|\bx\right|^{-\alpha}\,, & \bnabla\bu & \sim\left|\bx\right|^{-\alpha-1}\,, & \Delta\bu & \sim\left|\bx\right|^{-\alpha-2}\,, & \bu\bcdot\bnabla\bu & \sim\left|\bx\right|^{-2\alpha-1}\,,\label{eq:decay}
\end{align}
and therefore the Navier-Stokes equations \eqref{intro-ns-eq} are
essentially linear (subcritial) for $\alpha>1$, are critical for
$\alpha=1$, and highly nonlinear (supercritical) for $\alpha<1$.\index{Criticality of the Navier-Stokes equations}\index{Navier-Stokes equations!criticality}
However, since the net force is a conserved quantity, we have for
$\bff$ with compact support and $R$ big enough:
\[
\bF=\int_{\mathbb{R}^{n}}\bff=\int_{\partial B(\bzero,R)}\mathbf{T}\bn\,,
\]
where $\mathbf{T}$ is the stress tensor\index{Stress tensor} including
the convective part, $\mathbf{T}=\bnabla\bu+\left(\bnabla\bu\right)^{T}-p\,\boldsymbol{1}-\bu\otimes\bu$
and $B(\bzero,R)$ the open ball of radius $R$ centered at the origin.
Again by power counting, if $\bu$ satisfies \eqref{decay}, we obtain
that $\mathbf{T}\sim\left|\bx\right|^{-\min(\alpha+1,2\alpha)}$,
so if $2\alpha>n-1$, the limit\index{Net force!implications on the asymptotic behavior}
$R\to\infty$ vanishes and $\bF=\bzero$.\label{intro-criticality}
Consequently, in three dimensions, $\alpha=1$ is the critical case
for the equations as well as for the net force, whereas in two dimensions,
the equations have to be supercritical if we want to generate a nonzero
net force. If the net force vanishes, the solution of the Stokes equations
decays like $\left|\bx\right|^{-2}$ in three dimensions, so the problem
is subcritical and like $\left|\bx\right|^{-1}$ in two dimensions,
which is the critical regime. The different regimes are described
in \tabref{summary-decays}. Therefore, the problem is critical in
three dimensions if $\bF\ne\bzero$ and in two dimensions if $\bF=\bzero$.
In both of these cases, inverting the Stokes operator on the nonlinearity,
which by power counting decays like $\left|\bx\right|^{-3}$, leads
to a solution decaying like $\left|\bx\right|^{-1}\log\left|\bx\right|$.
Therefore, the Stokes system is ill-posed in this critical setting
and the leading term at infinity cannot be the Stokes fundamental
solution. In three dimensions this was proven by \citet[Theorem 3.1]{Deuring.Galdi-AsymptoticBehaviorof2000}
and in two dimensions this is proven in \chapref{link-Stokes-NS}.\\

We now discuss the critical cases in more details. In three dimensions,
by using an idea of \citet[Theorem 3.2]{Nazarov-steady2000}, \citet{Korolev.Sverak-largedistanceasymptotics2011}
proved by a fixed point argument that for small data the asymptotic
behavior is given by a class of exact solutions found by \citet{Landau-newexactsolution1944}.
The Landau solutions\index{Landau solutions}\index{Exact solutions!Landau}
are a family of exact and explicit solutions $\bUF$ of \eqref{intro-ns}
in $\mathbb{R}^{3}\setminus\left\{ \bzero\right\} $ parameterized
by $\bF\in\mathbb{R}^{3}$ and corresponding, in the sense of distributions,
to $\bff(\bx)=\bF\delta^{3}(\bx)$, so having a net force $\bF$.
Moreover, these are the only solutions that are invariant under the
scaling symmetry\index{Scaling!symmetry}\index{Symmetry!scaling}\index{Scale-invariant solutions},
\emph{i.e.} such that $\lambda\bu(\lambda\bx)=\bu(\bx)$ for all $\lambda>0$
\citep{Sverak-LandausSolutionsNavier2011}. Given this candidate for
the asymptotic expansion of the solution up to the critical decay,
the second step is to define $\bu=\bUF+\bv$, so that the Navier-Stokes
equations \eqref{intro-ns} become
\begin{align*}
\Delta\bv-\bnabla q & =\bUF\bcdot\bnabla\bv+\bv\bcdot\bnabla\bUF+\bv\bcdot\bnabla\bv+\boldsymbol{g}\,, & \bnabla\bcdot\bu & =0\,, & \lim_{|\bx|\to\infty}\bu & =\bzero\,,
\end{align*}
where the resulting source term $\boldsymbol{g}$ has zero mean, which
lifts the compatibility condition of the Stokes problem related to
the net force. Since $\bUF$ is bounded by $\left|\bx\right|^{-1}$,
the cross term $\bUF\bcdot\bnabla\bv+\bv\bcdot\bnabla\bUF$ is a critical
perturbation of the Stokes operator. Therefore this term can be put
together with the nonlinearity in order to perform a fixed point argument
on a space where $\bv$ is bounded by $\left|\bx\right|^{-2+\varepsilon}$
for some $\varepsilon>0$. This argument leads to the existence of
solutions satisfying
\[
\bu=\bUF+O(\left|\bx\right|^{-2+\varepsilon})\,,
\]
provided $\bff$ is small enough. Therefore, the key idea of this
method is to find the asymptotic term that lifts the compatibility
condition corresponding to the net force $\bF$. If net force is zero,
the solution of the Stokes equations in three dimensions decays like
$\left|\bx\right|^{-2}$, so we are in the subcritical regime and
everything is governed by the linear part of the equation, \emph{i.e.}
the Stokes equations.

In two dimensions and if $\bF=\bzero$, the solution of the Stokes
equations again decays like $\left|\bx\right|^{-1}$, and therefore
we are also in the critical case. In \chapref{strong-compatibility}
we determine the three additional compatibility conditions on the
data needed so that the solution of the Stokes equations decay faster
than $\left|\bx\right|^{-1}$. Once this is known, we can use a fixed
point argument in order to obtain the existence of solutions decaying
faster than $\left|\bx\right|^{-1}$ for small data satisfying three
compatibility conditions. Moreover, these compatibility conditions
can be automatically fulfilled by assuming suitable discrete symmetries,
which will improve the results of \citet{Yamazaki-Uniqueexistence2011}.
In \chapref{strong-compatibility}, we also show how to lift the compatibility
condition corresponding to the net torque $M$ with the solution $M\left|\bx\right|^{-2}\bx^{\perp}$,
however two compatibility conditions not related to invariant quantities
remain.

In \chapref{link-Stokes-NS}, we prove that the two solutions of the
Stokes equations decaying like $\left|\bx\right|^{-1}$ and which
are given by the two remaining compatibility conditions cannot be
the asymptote of any solutions of the Navier-Stokes equations in two-dimensions.
By analogy with the three-dimensional case where the asymptote is
given by the Landau solution which is scale-invariant, we can look
for a scale-invariant solution to describe the asymptotic behavior
also in two dimensions. As proved by \citet{Sverak-LandausSolutionsNavier2011},
the scale-invariant solutions of the Navier-Stokes equations are given
by the exact solutions found by \citet[\S 6]{Hamel-SpiralfoermigeBewegungen1917}.
These solutions are parameterized by the flux $\Phi\in\mathbb{R}$,
an angle $\theta_{0}$, and a discrete parameter $n$. As explained
by \citet[\S 5]{Sverak-LandausSolutionsNavier2011}, they are far
from the Stokes solutions decaying like $\left|\bx\right|^{-1}$,
so cannot be used to lift the compatibility conditions of the Stokes
equations. In an attempt to obtain the correct asymptotic behavior,
\citet{Guillod-Generalizedscaleinvariant2015} defined the notion
of a scale-invariant solution up to a rotation\index{Exact solutions!scale-invariant solution up to a rotation},
\emph{i.e.} a solution that satisfies 
\[
\boldsymbol{u}(\bx)=\e^{\lambda}\mathbf{R}_{\kappa\lambda}\boldsymbol{u}(\e^{\lambda}\mathbf{R}_{-\kappa\lambda}\bx)\,,
\]
for some $\kappa\in\mathbb{R}$, where $\mathbf{R}_{\vartheta}$ is
the rotation matrix of angle $\vartheta$. This is a combination of
the scaling and rotational symmetries. The scale-invariant solutions
up to a rotation of the two-dimensional Navier-Stokes equations in
$\mathbb{R}^{2}\setminus\left\{ \bzero\right\} $ are parameterized
by the flux $\Phi\in\mathbb{R}$, a parameter $\kappa\in\mathbb{R}$,
an angle $\theta_{0}$, and a discrete parameter $n$. These solutions
generalize the solutions found by \citet[\S 6]{Hamel-SpiralfoermigeBewegungen1917}
and exhibit a spiral behavior as shown in \figref{intro-hamel-new}.
However, at zero-flux, these new exact solutions have only two free
parameters, and are therefore not sufficient to lift the three compatibility
conditions of the Stokes equations required for a decay of the velocity
strictly faster than the critical decay $\left|\bx\right|^{-1}$.
Nevertheless, these exact solutions show that the asymptotic behavior
of the solutions in the case where $\bF=\bzero$ are highly nontrivial,
since by choosing a suitable boundary condition $\bu^{*}$ for an
exterior domain or source force $\bff$ if $\Omega=\mathbb{R}^{2}$,
it is easy to construction a solution that is equal to any of these
exact solutions, at least at large distances. Therefore the determination
of the general asymptotic behavior of the two-dimensional Navier-Stokes
equations with zero net force is still open and the numerical simulations
presented in \chapref{wake-sim} seem to indicate that the asymptotic
behavior is quite complicated.
\begin{figure}[h]
\includefigure{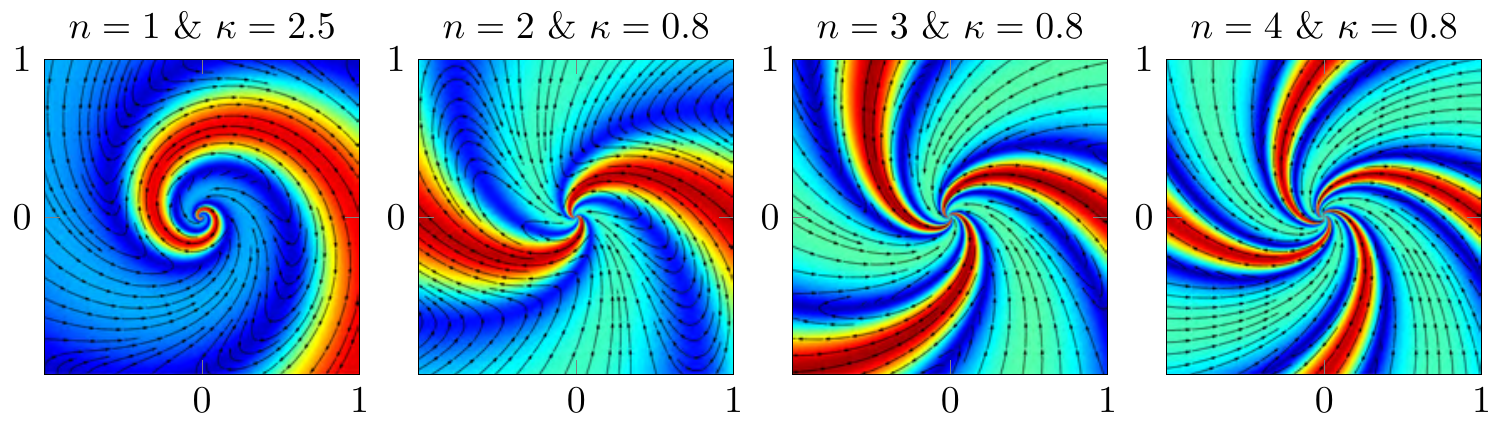}\caption{\label{fig:intro-hamel-new}The exact solutions found by \citet{Guillod-Generalizedscaleinvariant2015}
with zero flux are parametrized by a discrete parameter $n$ and a
real parameter $\kappa$.}
\end{figure}
\\

Finally, we discuss the supercritical case, that is to say the two-dimensional
Navier-Stokes equations for a nonzero net force $\bF\neq\bzero$.
By assuming that the decay of the solution is homogeneous, the previous
power counting argument shows that the solution cannot decay faster
than $\left|\bx\right|^{-1/2}$. By assuming that the velocity field
has an homogeneous decay like $\left|\bx\right|^{-1/2}$, we obtain
that this leading term has to be a solution of the Euler equations.
Such a solution of the Euler equations generating a nonzero net force
$\bF$ exists. However this cannot be the asymptotic behavior of the
Navier-Stokes equations at least for small data, because the solution
will have a big flux $\Phi\leq-3\pi$. This analysis is shown in \secref{asy-euler}.

The idea to determine the correct asymptotic behavior is to make an
ansatz such that at large distances, parts of the linear and nonlinear
terms of the equation remain both dominant unlike for the previous
attempt where only the nonlinear part had dominant terms. More precisely,
\citet{Guillod-Asymptoticbehaviour2013} consider an inhomogeneous
ansatz, whose decay and inhomogeneity are fixed by the requirement
that parts of the linearity and nonlinearity remain at large distances
and that net force is nonzero. The analysis in \citet{Guillod-Asymptoticbehaviour2013}
was done in Cartesian coordinates which are not very adapted to this
problem. In \secref{asy-wake}, we use a conformal change of coordinates
to introduce the inhomogeneity which makes the analysis much simpler
and intuitive. This leads to a solution $\left(\bUF,P_{\bF}\right)$
of the Navier-Stokes equations in $\mathbb{R}^{2}$ with some $\bff=\left(O(\left|\bx\right|^{-7/3}),O(\left|\bx\right|^{-8/3})\right)$
at infinity. This solution generates a net force $\bF$ and is a candidate
for the general asymptotic behavior in the case $\bF\neq\bzero$.
In polar coordinates, the velocity field has the following decay at
infinity, 
\begin{equation}
\bUF=\frac{2a^{2}}{3r^{1/3}}\sech^{2}\left(a\sin\left(\frac{\theta-\theta_{0}}{3}\right)r^{1/3}\right)\frac{\bF}{\left|\bF\right|}+O(r^{-2/3})\,,\label{eq:intro-exact-wake}
\end{equation}
where
\begin{align*}
\theta_{0} & =\arg(-F_{1}-\i F_{2})\,, & a & =\left(\frac{9\left|\bF\right|}{16}\right)^{1/3}\,.
\end{align*}
This solution is represented in \figref{intro-ansatz} and has a wake
behavior: inside the wake characterized by $\left|\theta-\theta_{0}\right|r^{1/3}\leq1$,
the velocity decays like $\left|\bx\right|^{-1/3}$ and outside the
wake like $\left|\bx\right|^{-2/3}$. This time, the asymptotic expansion
does not have a flux, and moreover numerical simulations (see \figref{intro-wake})
indicate that this is most probably the correct asymptotic behavior
if $\bF\neq\bzero$. In the last part of \chapref{wake-sim}, we will
perform systematic numerical simulations based on the analysis of
the Stokes equations and the results of \chapref{strong-compatibility,link-Stokes-NS}.
More precisely, when the net force is nonzero, the asymptotic behavior
is given by $\bUF$, however when the net force is vanishing the asymptotic
behavior seems to be much less universal. In some regime, the asymptote
is given by a double wake $\bUF+\bUmF$ so that the net force is effectively
zero (see \figref{intro-wake-double}), in some other regime by the
harmonic solution $\mu\be_{\theta}/r$ , and finally can also be the
exact scale -invariant solution up to a rotation discussed in \citet{Guillod-Generalizedscaleinvariant2015}.
The presence of the double wake is surprising, because intuitively
on would expect that the solution should behave like the Stokes solution,\emph{
i.e.} like $\left|\bx\right|^{-1}$ and not like $\left|\bx\right|^{-1/3}$,
since we are in the critical case as in three dimensions where the
asymptote is the \citet{Landau-newexactsolution1944} solution. Finally,
in \secref{wake-multiple} we also show numerically that three or
more wakes can be produced, but only for large data. The decays of
the Stokes and Navier-Stokes equations as well as their asymptotes
are summarized in \tabref{summary-decays}.

\subsubsection*{Acknowledgment}

The author would like to thank Peter Wittwer for many useful discussions
and comments on the subject of this work as well as Matthieu Hillairet
for fruitful discussions and for having pointed out the exact solution
of the Euler equations given in \eqref{euler-sol}.

\begin{figure}[p]
\includefigure{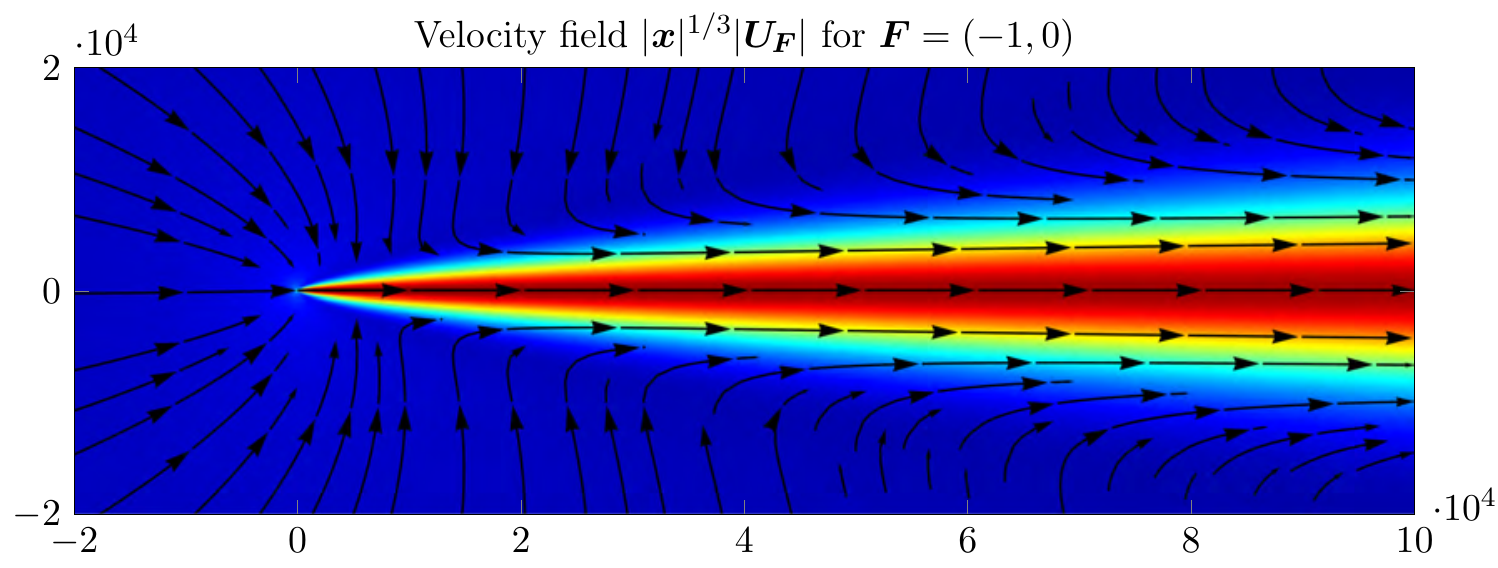}\caption{\label{fig:intro-ansatz}The solution $\protect\bUF$ is multiplied
by $\left|\protect\bx\right|^{1/3}$ in order to highlight its decay
properties. Inside a wake characterized by $\left|\theta\right|r^{1/3}\leq1$,
$\protect\bUF$ decays like $\left|\protect\bx\right|^{-1/3}$ inside
the wake, whereas it decays like $\left|\protect\bx\right|^{-2/3}$
outside the wake region.}
\end{figure}
\begin{figure}[p]
\includefigure{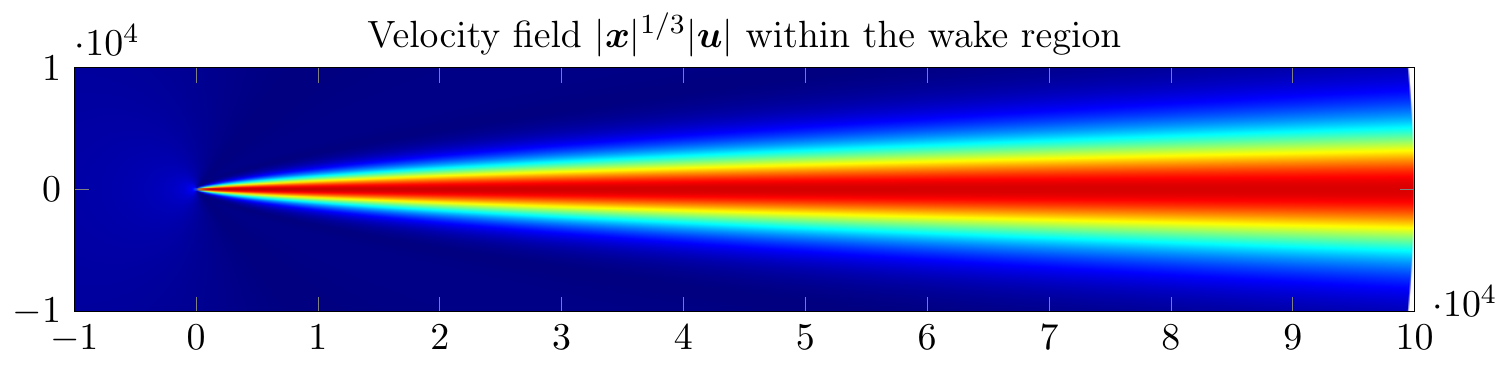}\caption{\label{fig:intro-wake}Numerical simulation of the Navier-Stokes equations
with $\protect\bF\protect\neq\protect\bzero$. The velocity field
is asymptotic to $\protect\bUF$ defined by \eqref{intro-exact-wake}
with very high precision.}
\end{figure}
\begin{figure}[p]
\includefigure{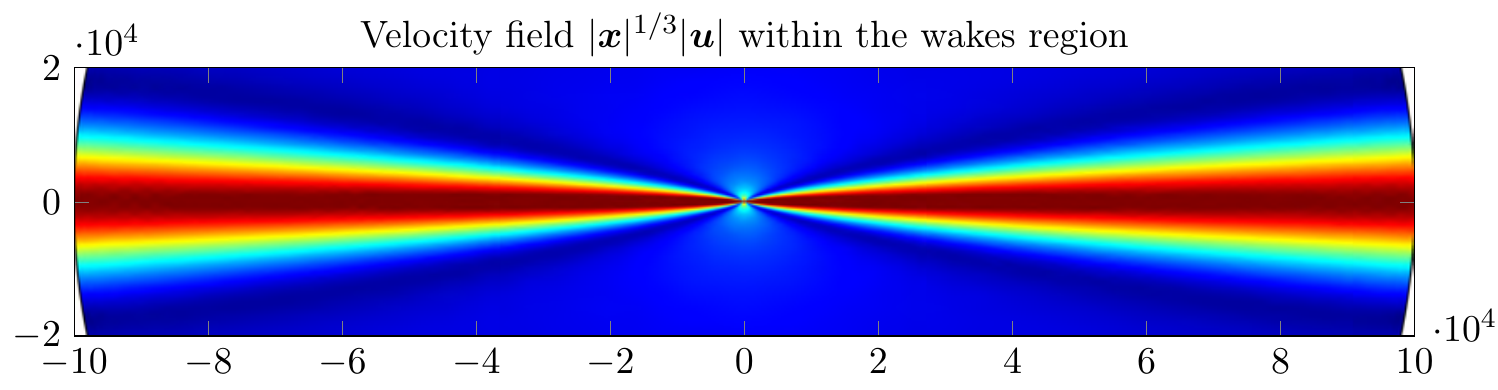}\caption{\label{fig:intro-wake-double}Numerical simulation of the Navier-Stokes
equations with $\protect\bF=\protect\bzero$ for a specific choice
of the boundary conditions. The velocity field is only bounded by
$\left|\protect\bx\right|^{-1/3}$.}
\end{figure}
\begin{table}[h]
\includefigure{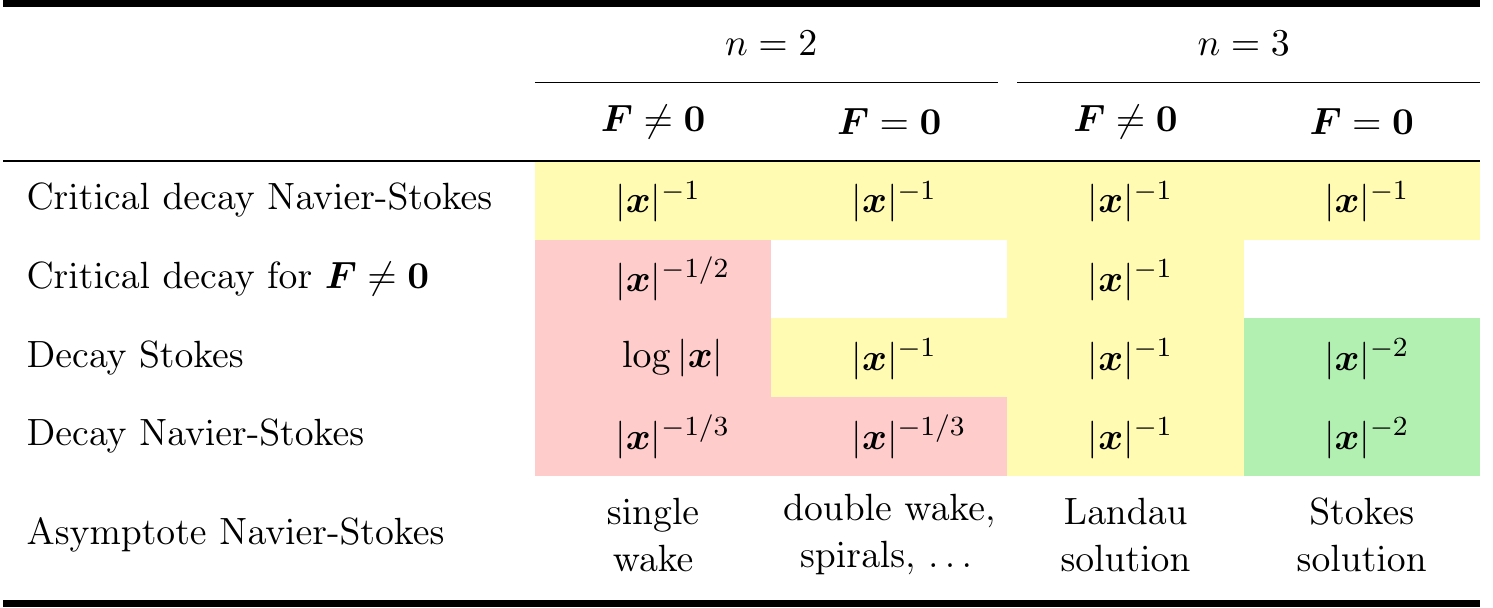}

\caption{\label{tab:summary-decays}Summary of the different properties of
the Stokes and Navier-Stokes equations in $\mathbb{R}^{n}$. In every
dimension, the critical decay of the Navier-Stokes equations is given
by $\left|\protect\bx\right|^{-1}$ and is drawn in yellow. The decays
that make the equations subcritical are drawn in green and the ones
that are supercritical are shown in red. As shown on page \pageref{intro-criticality},
the critical decay for having a nonzero net force is $\left|\protect\bx\right|^{-1/2}$
in two dimensions and $\left|\protect\bx\right|^{-1}$ in three dimensions.
The results of the two-dimensional cases are based on \citet{Guillod-Asymptoticbehaviour2013,Guillod-Generalizedscaleinvariant2015}
and on the results of \chapref{wake-sim}. In three dimensions, the
results were proven by \citet{Korolev.Sverak-largedistanceasymptotics2011}.}
\end{table}

\clearpage{}

\section{Notations}

For the reader's convenience, we collect here the most frequently
used symbols:\index{Notations}%
\begin{longtable}[l]{rl}
$\lesssim$ & less than up to a constant: $a\lesssim b$ means $a\leq Cb$ for some
$C>0$\tabularnewline
$n$ & dimension of the underlying space\tabularnewline
$\bx$ & position: $\bx=(x_{1},\dots,x_{n})$\tabularnewline
$\be_{i}$ & unit vector in the direction $i$\tabularnewline
$r$ & radial polar coordinate: $r=\vert\bx\vert$\tabularnewline
$\theta$ & angular polar coordinate: $\theta=\arg(x_{1}+\i x_{2})\in(-\pi;\pi]$\\[8pt]
$\ensuremath{B(\bx,R)}$ & open ball of radius $R$ centered at $\bx$\tabularnewline
$\Omega$ & region of flow\tabularnewline
$\partial\Omega$ & boundary of the domain $\Omega$\tabularnewline
$\bn$ & normal outgoing unit vector to the boundary $\partial\Omega$\\[8pt]
$\bv$ & vector: $\bv=(v_{1},\dots,v_{n})$\tabularnewline
$\ensuremath{\vert\bv\vert}$ & Euclidean norm of the vector $\bv$: $\vert\bv\vert^{2}=\sum_{i=1}^{n}v_{i}^{2}$\tabularnewline
$\ensuremath{\bv^{\perp}}$ & orthogonal of the two-dimensional vector $\bv=(v_{1},v_{2})$: $\bv^{\perp}=(-v_{2},v_{1})$\tabularnewline
$\ensuremath{\bv_{1}\bcdot\bv_{2}}$ & scalar product between $\bv_{1}$ and $\bv_{2}$\tabularnewline
$\ensuremath{\bv_{1}\bwedge\bv_{2}}$ & cross product between the three-dimensional vectors $\bv_{1}$ and
$\bv_{2}$\tabularnewline
$\mathbf{A}$ & second-order tensor field: $\bA=(A_{ij})_{i,j=1,\dots,n}$\tabularnewline
$\mathbf{A}:\mathbf{B}$ & contraction of the tensors $\mathbf{A}$ and $\mathbf{B}$: $\mathbf{A}:\mathbf{B}=\sum_{i,j=1}^{n}A_{ij}B_{ij}$\\[8pt]
$\varphi$ & scalar field: $\varphi(x)$\tabularnewline
$\bphi$ & vector field: $\left(\varphi_{1}(\bx),\dots,\varphi_{n}(\bx)\right)$\tabularnewline
$\ensuremath{\bnabla\varphi}$ & gradient of the scalar field $\varphi$: $\bnabla\varphi=(\partial_{1}\varphi,\dots,\partial_{n}\varphi)$\tabularnewline
$\ensuremath{\bnabla\bcdot\bphi}$ & divergence of the vector field $\bphi$: $\bnabla\bcdot\bphi=\sum_{i=1}^{n}\partial_{i}\varphi_{i}$\tabularnewline
$\ensuremath{\bnabla\bwedge\bphi}$ & curl of the three-dimensional vector field $\bphi$\tabularnewline
$\ensuremath{\bnabla\wedge\varphi}$ & curl of the scalar field $\varphi$: $\bnabla\wedge\varphi=\bnabla^{\perp}\varphi=(-\partial_{2}\varphi,\partial_{1}\varphi)$\\[8pt]
$\bu$ & velocity field\tabularnewline
$p$ & pressure field\tabularnewline
$\omega$ & vorticity field: $\omega=\bnabla\bwedge\bu$\tabularnewline
$\psi$ & stream function: $\bu=\bnabla\wedge\psi$\tabularnewline
\end{longtable}

\section{Symmetries of the Navier-Stokes equations}

The aim is to determine all the infinitesimal symmetries\index{Symmetry!Navier-Stokes equations}
that leave the homogeneous Navier-Stokes equations in $\mathbb{R}^{n}$
invariant. The symmetries of the time-dependent Navier-Stokes equations
were determined by \citet{Lloyd-infinitesimalgroup1981}. It is not
completely obvious that the symmetries of the stationary case are
given by the time-independent symmetries of the time-dependent case
only. The following proposition establishes that this is actually
the case:
\begin{prop}
\label{prop:symmetries}For $n=2,3$, the only infinitesimal symmetries\index{Infinitesimal symmetry}
of the type
\begin{align}
\bx\mapsto & \bx+\varepsilon\boldsymbol{\xi}(\bu,p,\bx)\,, & \left(\bu,p\right) & \mapsto\left(\bu,p\right)+\boldsymbol{\eta}(\bu,p,\bx)\,,\label{eq:symmetries-epsilon}
\end{align}
\emph{i.e.} generated by 
\[
X=\boldsymbol{\xi}\bcdot\bnabla_{\bx}+\boldsymbol{\eta}\bcdot\bnabla_{(\bu,p)}\,,
\]
which leave the homogeneous Navier-Stokes equations in $\mathbb{R}^{n}$
invariant are:
\begin{enumerate}
\item The translations\index{Symmetry!translations}\index{Translational symmetry}
\[
\bx\mapsto\bx+\boldsymbol{\delta}\,,
\]
where $\boldsymbol{\delta}\in\mathbb{R}^{n}$, whose generator is
given by
\[
X=\frac{\boldsymbol{\delta}}{\left|\boldsymbol{\delta}\right|}\bcdot\bnabla_{\bx}\,.
\]

\item The rotations\index{Symmetry!rotations}\index{Rotational symmetry}
\begin{align*}
\bu(\bx) & \mapsto\mathbf{R}^{-1}\bu(\mathbf{R}\bx)\,, & p(\bx) & \mapsto p(\mathbf{R}\bx)\,,
\end{align*}
for $\mathbf{R}\in\text{SO}(n)$, where the $n(n-1)/2$ generators
are given in terms of the lie algebra $\mathfrak{so}(n)$. For example
for $n=2$,
\[
X=\bx^{\perp}\bcdot\left(\bnabla_{\bx}+\bnabla_{\bu}\right)\,.
\]

\item The scaling symmetry,\index{Symmetry!scaling}\index{Scaling!symmetry}
\begin{align*}
\bu(\bx) & \mapsto\e^{\lambda}\bu(\e^{\lambda}\bx)\,, & p(\bx) & =\e^{2\lambda}p(\e^{\lambda}\bx)\,,
\end{align*}
for $\lambda\in\mathbb{R}$, which corresponds to
\[
X=\bx\bcdot\bnabla_{\bx}-\bu\bcdot\bnabla_{\bu}-2p\partial_{p}\,.
\]

\item The addition of a constant $c$ to the pressure,
\[
p\mapsto p+c\,,
\]
for $c\in\mathbb{R}$, which corresponds to
\[
X=\partial_{p}\,.
\]

\end{enumerate}
\end{prop}
\begin{proof}
We use the same method as \citet{Lloyd-infinitesimalgroup1981}, which
is explained in details by \citet{Eisenhart-ContinuousGroupsTransformations1933}.
First of all we write the Navier-Stokes equations as $\boldsymbol{L}=\bzero$,
where
\[
\boldsymbol{L}=\begin{pmatrix}\Delta\bu-\bnabla p-\bu\bcdot\bnabla\bu\\
\bnabla\bcdot\bu
\end{pmatrix}\,,
\]
and define $\bv=\left(\bu,p\right)$. Since $\boldsymbol{L}$ is a
second order differential operator, we have to compute the transformations
of the first and second derivatives. We have
\[
\partial_{i}\mapsto\partial_{i}-\varepsilon\frac{\rd\boldsymbol{\xi}}{\rd x_{i}}\bcdot\bnabla\,,
\]
so that
\[
D^{\alpha}\bv\mapsto D^{\alpha}\bv+\varepsilon\boldsymbol{\eta}^{\alpha}\,,
\]
where $\boldsymbol{\eta}^{\alpha}$ is defined by recursion through
\[
\boldsymbol{\eta}^{(\alpha,\beta)}=\frac{\rd\boldsymbol{\eta}^{\beta}}{\rd x_{\alpha}}-\frac{\rd\boldsymbol{\xi}}{\rd x_{\alpha}}\bcdot\bnabla D^{\beta}\bv\,,
\]
where $\alpha$ and $\beta$ are multi-indices with $\left|\alpha\right|=1$.
We consider the second extension of $X$,
\[
X_{2}=\boldsymbol{\xi}\bcdot\bnabla_{\bx}+\sum_{\left|\alpha\right|\leq2}\boldsymbol{\eta}^{\alpha}\bcdot\bnabla_{D^{\alpha}\bv}\,.
\]
Then the Navier-Stokes system admits the symmetry \eqref{symmetries-epsilon}
if and only if $X_{2}\boldsymbol{L}=\bzero$ whenever $\boldsymbol{L}=\bzero$.
The idea of the proof is the following: we solve $\boldsymbol{L}=\bzero$
for $\bnabla p$ and $\partial_{1}u_{1}$, and substitute this into
$X_{2}\boldsymbol{L}=\bzero$. By grouping similar terms involving
$\bv$ and its derivatives, we can obtain a list of linear partial
differential equations for $\boldsymbol{\xi}$ and $\boldsymbol{\eta}$.
By using a computer algebra system, we obtain the explicit list of
partial differential equations for $\boldsymbol{\xi}$ and $\boldsymbol{\eta}$.
For $n=2$, the general solution is given by
\begin{align*}
\boldsymbol{\xi} & =\boldsymbol{\delta}+\lambda\bx+r\bx^{\perp}\,,\\
\left(\eta_{1},\eta_{2}\right) & =-\lambda\bu+r\bx^{\perp}\\
\eta_{3} & =-2p\lambda+c\,,
\end{align*}
where $\boldsymbol{\delta}\in\mathbb{R}^{2}$ and $\lambda,r,c\in\mathbb{R}$.
For $n=3$, we have similar results, except that there are three different
rotations. 
\end{proof}
In additions to the four infinitesimal symmetries listed in \propref{symmetries},
the Navier-Stokes equations are also invariant under discrete symmetries.
They are invariant under the central symmetry
\begin{align}
\bx & \mapsto-\bx\,, & \bu & \mapsto-\bu\,,\label{eq:central-symmetry}
\end{align}
and under the reflections with respect to an axis or a plane. For
example, the reflection with respect to the first coordinate $x_{1}$
is given by
\begin{align}
\bx=\left(x_{1},\tilde{\bx}\right) & \mapsto\left(-x_{1},\tilde{\bx}\right)\,, & \bu=\left(u_{1},\tilde{\bu}\right) & \mapsto\left(-u_{1},\tilde{\bu}\right)\,.\label{eq:axial-symmetry}
\end{align}
This corresponds to the reflection with respect to the $x_{2}$-axis
for $n=2$ and with respect to the $x_{2}x_{3}$-plane for $n=3$.

\section{Invariant quantities of the Navier-Stokes equations}

We consider the stationary Navier-Stokes equations \eqref{intro-ns-eq}
in a sufficiently smooth bounded domain $\Omega\subset\mathbb{R}^{n}$,
$n=2,3$. For clarity, we add a source-term $g$ in the divergence
equation, so we consider
\begin{align}
\Delta\bu-\bnabla p & =\bnabla\bcdot\left(\bu\otimes\bu\right)+\bff\,, & \bnabla\bcdot\bu & =g\,,\label{eq:invariant-ns}
\end{align}
which is equal to \eqref{intro-ns-eq} if $g=0$. The aim is to show
that the only invariant quantities in a sense defined below, are the
flux, the net force, and the net torque.
\begin{defn}[invariant quantity]
For two functions $\boldsymbol{\Lambda}\in C^{\infty}(\Omega,\mathbb{R}^{n+1})$
and $\Lambda\in C^{\infty}(\Omega,\mathbb{R})$ we consider the functional
\[
I[\bff,g]=\int_{\Omega}\left(\boldsymbol{\Lambda}\bcdot\bff+\Lambda g\right)\,.
\]
The functional $I[\bff,g]$ is an invariant quantity\index{Invariant quantity}
if it can be expressed in terms of an integral on $\partial\Omega$,
\emph{i.e.} such that there exists a function $\boldsymbol{\lambda}\in C^{\infty}(\mathbb{R}^{n+1},\mathbb{R}^{n})$
with 
\[
I[\bff,g]=\int_{\partial\Omega}\boldsymbol{\lambda}[\bu,p]\bcdot\bn\,,
\]
for any smooth $\bu$, $p$, $\bff$ and $g$ satisfying \eqref{invariant-ns}.\end{defn}
\begin{rem}
The name invariant comes from the fact that if for example $\bu,p,\bff,g$
satisfy \eqref{invariant-ns} in $\mathbb{R}^{n}$, with $\bff,g$
having support in a bounded set $B$, then the quantity $I[\bff,g]$
does not depend on the domain of integration $\Omega$ as soon as
$B\subset\Omega$, and in particular $\int_{\partial\Omega}\boldsymbol{\lambda}[\bu,p]\bcdot\bn$
is independent of the choice of any smooth closed curve or surface
$\partial\Omega$ that encircles $B$.\end{rem}
\begin{prop}
\label{prop:invariants}The only invariant quantities (that are not
linearly related) are the flux $\Phi\in\mathbb{R}$, the net force
$\bF\in\mathbb{R}^{n}$, and the net torque $M\in\mathbb{R}$ if $n=2$
and $\boldsymbol{M}\in\mathbb{R}^{3}$ if $n=3$, which are given
by\index{Invariant quantity!net flux}\index{Invariant quantity!net force}\index{Invariant quantity!net torque}\index{Net flux}\index{Flux}\index{Net force}\index{Net torque}\index{Torque}
\begin{align*}
\Phi & =\int_{\Omega}g=\int_{\partial\Omega}\bu\bcdot\bn\,, & \bF & =\int_{\Omega}\bff=\int_{\partial\Omega}\mathbf{T}\bn\,, & M\text{ or }\boldsymbol{M} & =\int_{\Omega}\bx\bwedge\bff=\int_{\partial\Omega}\bx\bwedge\mathbf{T}\bn\,,
\end{align*}
where $\mathbf{T}$ is the stress tensor including the convective
part,
\begin{equation}
\mathbf{T}=\bnabla\bu+\left(\bnabla\bu\right)^{T}-p\,\mathbf{1}-\bu\otimes\bu\,.\label{eq:stress-tensor}
\end{equation}
\end{prop}
\begin{proof}
The Navier-Stokes equation \eqref{invariant-ns} can be written as
\begin{align*}
\bnabla\bcdot\mathbf{T} & =\bff\,, & \bnabla\bcdot\bu & =g\,.
\end{align*}
For two general functions $\boldsymbol{\Lambda}$ and $\Lambda$,
and a solution of the previous equation, we have
\begin{align*}
I[\bff,g] & =\int_{\Omega}\left(\boldsymbol{\Lambda}\bcdot\bff+\Lambda g\right)=\int_{\Omega}\boldsymbol{\Lambda}\bcdot\bnabla\bcdot\mathbf{T}+\int_{\Omega}\Lambda\bnabla\bcdot\bu\\
 & =\int_{\Omega}\bnabla\bcdot\left(\mathbf{T}\boldsymbol{\Lambda}\right)-\int_{\Omega}\bnabla\boldsymbol{\Lambda}:\mathbf{T}+\int_{\Omega}\bnabla\bcdot\left(\Lambda\bu\right)-\int_{\Omega}\bnabla\Lambda\bcdot\bu\\
 & =\int_{\partial\Omega}\left(\mathbf{T}\boldsymbol{\Lambda}+\Lambda\bu\right)\bcdot\bn-\int_{\Omega}\left(\bnabla\boldsymbol{\Lambda}:\mathbf{T}+\bnabla\Lambda\bcdot\bu\right)\,.
\end{align*}
Now we determine in which cases the integral over $\Omega$ vanishes,
\[
\int_{\Omega}\left(\bnabla\boldsymbol{\Lambda}:\mathbf{T}+\bnabla\Lambda\bcdot\bu\right)=0
\]
for all $\bu,p,\bff,g$ satisfying \eqref{invariant-ns}. Since this
integral does not depend on $\bff$ and $g$, we can choose $\bu\in C^{\infty}(\Omega,\mathbb{R}^{n})$
and $p\in C^{\infty}(\Omega,\mathbb{R})$ arbitrarily, and therefore
the tensor $\mathbf{T}$ is an arbitrary symmetric tensor. Consequently,
we obtain the conditions
\begin{align*}
\int_{\Omega}\bnabla\boldsymbol{\Lambda}:\mathbf{T} & =0\,, & \int_{\Omega}\bnabla\Lambda\bcdot\bu & =0\,,
\end{align*}
for all $\bu\in C^{\infty}(\Omega,\mathbb{R}^{n})$ and all symmetric
tensors $\mathbf{T}\in C^{\infty}(\Omega,\mathbb{R}^{n}\otimes\mathbb{R}^{n})$.
For $n=2$, this implies the equations
\begin{align*}
\partial_{1}\Lambda_{1} & =0\,, & \partial_{2}\Lambda_{2} & =0\,, & \partial_{1}\Lambda_{2}+\partial_{2}\Lambda_{1} & =0\,,
\end{align*}
and
\begin{align*}
\partial_{1}\Lambda & =0\,, & \partial_{2}\Lambda & =0\,.
\end{align*}
The general solution of the system is given by
\begin{align*}
\boldsymbol{\Lambda}(\bx) & =\boldsymbol{A}+B\bx^{\perp}\,, & \Lambda(\bx) & =C\,,
\end{align*}
where $\boldsymbol{A}\in\mathbb{R}^{n}$, $B,C\in\mathbb{R}$, and
therefore the only invariant quantities linearly independent are the
net force $\bF$ and the net torque $M$. For $n=3$, the equations
are similar and lead to the same result, except that the net torque
has three parameters.
\end{proof}

\chapter{\label{chap:weak}Existence of weak solutions}

In order to prove existence of weak solutions to \eqref{intro-ns},
one has to face two kinds of difficulties: the local behavior and
the behavior at large distances. The local behavior corresponds to
the differentiability properties of the solutions, which can be deduced
from the case where $\Omega$ is bounded. The behavior at large distances
is much more complicated but information on it is required to prove
that the solutions satisfy \eqref{intro-ns-limit}. In three dimensions,
the function spaces used in the definition of weak solutions are sufficient
to prove the limiting behavior at large distances, but in two dimensions
this is not the case. The behavior of the two-dimensional weak solutions
of the Navier-Stokes equations is one of the most important open problem
in stationary fluid mechanics. In this chapter, we review the construction
of weak solutions in Lipschitz domain in two and three dimensions
and analyze their asymptotic behavior.

We denote by $C_{0,\sigma}^{\infty}(\Omega)$ the space of smooth
solenoidal functions\index{Space!C@$C_{0,\sigma}^{\infty}$} compactly
supported in $\Omega$,
\[
C_{0,\sigma}^{\infty}(\Omega)=\left\{ \bphi\in C_{0}^{\infty}(\Omega):\,\bnabla\bcdot\bphi=0\right\} .
\]
By multiplying \eqref{intro-ns-eq} by $\bphi\in C_{0,\sigma}^{\infty}(\Omega)$
and integrating over $\Omega$, we have
\[
\int_{\Omega}\Delta\bu\bcdot\bphi-\int_{\Omega}\bnabla p\bcdot\bphi=\int_{\Omega}\bu\bcdot\bnabla\bu\bcdot\bphi+\int_{\Omega}\bff\bcdot\bphi\,,
\]
and if we integrate by parts, we obtain
\begin{equation}
\int_{\Omega}\bnabla\bu:\bnabla\bphi+\int_{\Omega}\bu\bcdot\bnabla\bu\bcdot\bphi+\int_{\Omega}\bff\bcdot\bphi=0\,.\label{eq:weak-solution}
\end{equation}
This implies that every regular solution of \eqref{intro-ns-eq} satisfies
\eqref{weak-solution} for all $\bphi\in C_{0,\sigma}^{\infty}(\Omega)$.
However, the converse is true only if $\bu$ is sufficiently regular.
This is the reason why a function $\bu$ satisfying \eqref{weak-solution}
for all $\bphi\in C_{0,\sigma}^{\infty}(\Omega)$ is called a weak
solution. We review the construction of weak solutions by the method
of \citet{Leray-Etudedediverses1933} and analyze the asymptotic behavior
of the velocity in the case where the domain is unbounded.

\section{Function spaces}

We now introduce the function spaces required for the proof of the
existence of weak solutions.
\begin{defn}[Lipschitz domain]
\index{Lipschitz domain}A Lipschitz domain $\Omega$ is a locally
Lipschitz domain whose boundary $\partial\Omega$ is compact. In particular
a Lipschitz domain is either:
\begin{enumerate}
\item a bounded domain;
\item an exterior domain\index{Exterior domain}\index{Domain!exterior},
\emph{i.e.} the complement in $\mathbb{R}^{n}$ of a compact set $B$
having a nonempty interior;
\item the whole space $\mathbb{R}^{n}$.
\end{enumerate}
If $\Omega$ is a bounded domain, respectively an exterior domain,
we can assume without loss of generality that $\bzero\in\Omega$,
respectively $\bzero\notin\Omega$.
\end{defn}

\begin{defn}[spaces $W^{1,2}$ and $D^{1,2}$]
The Sobolev space\index{Sobolev space} $W^{1,2}(\Omega)$ is the
Banach space\index{Space!W@$W^{1,2}$}
\[
W^{1,2}(\Omega)=\left\{ \bu\in L^{2}(\Omega):\,\bnabla\bu\in L^{2}(\Omega)\right\} ,
\]
with the norm
\[
\left\Vert \bu\right\Vert _{1,2}=\left\Vert \bu\right\Vert _{2}+\left\Vert \bnabla\bu\right\Vert _{2}\,.
\]
The homogeneous Sobolev space $D^{1,2}(\Omega)$ is defined as the
linear space\index{Space!D@$D^{1,2}$}
\[
D^{1,2}(\Omega)=\left\{ \bu\in L_{loc}^{1}(\Omega):\,\bnabla\bu\in L^{2}(\Omega)\right\} ,
\]
with the associated semi-norm
\[
\left|\bu\right|_{1,2}=\left\Vert \bnabla\bu\right\Vert _{2}\,.
\]
This semi-norm on $D^{1,2}(\Omega)$ defines the following equivalent
classes on $D^{1,2}(\Omega)$,
\[
\left[\bu\right]_{1}=\left\{ \bu+\boldsymbol{c},\:\boldsymbol{c}\in\mathbb{R}^{n}\right\} ,
\]
so that
\[
\left\{ \left[\bu\right]_{1},\:\bu\in D^{1,2}(\Omega)\right\} ,
\]
is a Hilbert space with the scalar product
\[
\left[\left[\bu\right]_{1},\left[\bv\right]_{1}\right]=\left(\bnabla\bu,\bnabla\bv\right)\,.
\]

\end{defn}
We now define the completion of $C_{0}^{\infty}(\Omega)$ in the previous
norms:
\begin{defn}[spaces $W_{0}^{1,2}$ and $D_{0}^{1,2}$]
The Banach space $W_{0}^{1,2}(\Omega)$ is defined as the completion
of $C_{0}^{\infty}(\Omega)$ with respect to the norm $\left\Vert \cdot\right\Vert _{1,2}$.
The semi-norm $\left|\cdot\right|_{1,2}$ defines a norm on $C_{0}^{\infty}(\Omega)$,
so we introduced the Banach space $D_{0}^{1,2}(\Omega)$ as the completion
of $C_{0}^{\infty}(\Omega)$ in the norm $\left|\cdot\right|_{1,2}$.\index{Space!WA@$W_{0}^{1,2}$}\index{Space!D0@$D_{0}^{1,2}$}
\end{defn}
The following lemmas (see for example \citealp[Theorems II.6.1 \& II.7.6]{Galdi-IntroductiontoMathematical2011}
or \citealp[Lemma III.1.2.1]{Sohr-Navier-Stokesequations.elementary2001})
prove that $D_{0}^{1,2}(\Omega)$ can be viewed as a space of locally
defined functions in case $\widebar{\Omega}\neq\mathbb{R}^{2}$:
\begin{lem}
Let $n\geq3$ and $\Omega\subset\mathbb{R}^{n}$ be any domain\index{Poincaré inequality}\index{Inequality!Poincaré}.
Then for all $\bu\in D_{0}^{1,2}(\Omega)$,
\[
\left\Vert \bu\right\Vert _{2n/(n-2)}\leq C\left\Vert \bnabla\bu\right\Vert _{2}\,,
\]
where $C=C(n)$. Moreover, for any $R>0$ big enough and $1<p\leq2n/(n-2)$,
\[
\left\Vert \bu;L^{p}(\Omega\cap B(\bzero,R))\right\Vert \leq C\left\Vert \bnabla\bu\right\Vert _{2}\,,
\]
for all $\bu\in D_{0}^{1,2}(\Omega)$, where $C=C(n,R,p)$.\end{lem}
\begin{proof}
The first inequality is a classical Sobolev embedding \citep[Theorem 9.9]{Brezis-FunctionalAnalysisSobolev2011},
since $p^{*}=\frac{2n}{n-2}$. Then for any $p<p^{*}$ and $R>0$
big enough, by Hölder inequality,
\[
\bigl\Vert\bu;L^{p}(\Omega\cap B(\bzero,R))\bigr\Vert\leq C(R,p)\bigl\Vert\bu;L^{p^{*}}(\Omega\cap B(\bzero,R))\bigr\Vert\,,
\]
and the second inequality follows by applying the first one.
\end{proof}

\begin{lem}
\label{lem:local-D0}Let $\Omega\subset\mathbb{R}^{2}$ be any domain
such that $\widebar{\Omega}\neq\mathbb{R}^{2}$. Then for any $R>0$
big enough and $p>1$, 
\[
\left\Vert \bu;L^{p}(\Omega\cap B(\bzero,R))\right\Vert \leq C\left\Vert \bnabla\bu\right\Vert _{2}\,,
\]
for all $\bu\in D_{0}^{1,2}(\Omega)$, where $C=C(\Omega,R,p)$. In
particular if $\Omega$ is bounded, then $D_{0}^{1,2}(\Omega)$ is
isomorphic to $W_{0}^{1,2}(\Omega)$.\end{lem}
\begin{proof}
It suffices to prove the inequality for all $\bu\in C_{0}^{\infty}(\Omega)$.
By the Sobolev embedding \citep[Corollary 9.11]{Brezis-FunctionalAnalysisSobolev2011},
for $p>2$,
\[
\bigl\Vert\bu;L^{p}(\Omega\cap B(\bzero,R))\bigr\Vert\leq C(R,q)\left(\bigl\Vert\bu;L^{2}(\Omega\cap B(\bzero,R))\bigr\Vert+\bigl\Vert\bnabla\bu;L^{2}(\Omega\cap B(\bzero,R))\bigr\Vert\right)\,.
\]
By the Hölder inequality, for $p<2$,
\[
\bigl\Vert\bu;L^{p}(\Omega\cap B(\bzero,R))\bigr\Vert\leq C(R,q)\bigl\Vert\bu;L^{2}(\Omega\cap B(\bzero,R))\bigr\Vert\,.
\]
Therefore it remains to prove the inequality for $p=2$. Since $\widebar{\Omega}\neq\mathbb{R}^{2}$,
there exists $\varepsilon>0$ and $\bx_{0}\in\mathbb{R}^{2}$ such
that $B(\bx_{0},\varepsilon)\cap\Omega=\emptyset$. By extending each
function $\bu\in C_{0}^{\infty}(\Omega)$ by zero from $\Omega\cap B(\bzero,R)$
to $B(\bzero,R)\setminus B(\bx_{0},\varepsilon)$, the Poincaré inequality
(\citealp[Theorem 1.5.]{Necas-DirectMethods2012} or \citealp[Corollary 9.19]{Brezis-FunctionalAnalysisSobolev2011})
implies that
\[
\bigl\Vert\bu;L^{2}(\Omega\cap B(\bzero,R))\bigr\Vert\leq C(R)\bigl\Vert\bnabla\bu\bigr\Vert_{2}\,.
\]

\end{proof}
The following example \citep[Remarque 4.1]{Deny-Lesespaces1954} shows
that the elements of $D_{0}^{1,2}(\mathbb{R}^{2})$ are equivalence
classes and cannot be viewed as functions.
\begin{example}
There exists a sequence $\left(u_{n}\right)\subset C_{0}^{\infty}(\mathbb{R}^{2})$
which converges to $u\in D_{0}^{1,2}(\mathbb{R}^{2})$ in the norm
$\left|\cdot\right|_{1,2}$ and a sequence $\left(c_{n}\right)_{n\in\mathbb{N}}\subset\mathbb{R}$
such that for any bounded domain $B$,
\[
\bigl\Vert u_{n}-u;L^{4}(B)\bigr\Vert\to\infty\qquad\text{and}\qquad\bigl\Vert u_{n}-c_{n}-u;L^{4}(B)\bigr\Vert\to0
\]
as $n\to\infty$.\end{example}
\begin{proof}
Let $a\in C^{\infty}(\mathbb{R},[0,1])$ such that $a(r)=0$ if $r\leq5/2$,
$a(r)=1$ if $r\geq3$. For $n\in\mathbb{N}$, let $a_{n}\in C_{0}^{\infty}(\mathbb{R},[0,1])$
such that $a_{n}(r)=a(r)$ if $r\leq n$ and $a_{n}(x)=0$ if $r\geq n+1$.
Then we consider the function $u_{n}\in C_{0}^{\infty}(\mathbb{R}^{2})$
defined by
\[
u_{n}(\bx)=-\int_{\left|\bx\right|}^{\infty}\frac{1}{r\log r}a_{n}(r)\,\rd r\,.
\]
The function $u_{n}$ is constant on $B(\bzero,2)$, and has support
inside $B(\bzero,n+1)$. We have
\[
\bnabla u_{n}(\bx)=\frac{1}{\left|\bx\right|\log\left|\bx\right|}a_{n}(\left|\bx\right|)\, e_{r}\,,
\]
and
\[
\left\Vert \bnabla u_{n}\right\Vert _{2}^{2}=2\pi\int_{2}^{n+1}\left(\frac{1}{r\log r}a_{n}(r)\right)^{2}r\,\rd r\leq2\pi\int_{2}^{n+1}\frac{1}{r\log^{2}r}\,\rd r\leq\frac{2\pi}{\log2}\,.
\]
so the sequence $\left(u_{n}\right)_{n\in\mathbb{N}}$ is bounded
in $D_{0}^{1,2}(\mathbb{R}^{2})$. Explicitly, we have
\[
\lim_{n\to\infty}\left\Vert \bnabla u_{n}-\bnabla u\right\Vert _{2}=0\,,
\]
where $u$ is determined by
\[
u(\bx)=\int_{0}^{\left|\bx\right|}\frac{1}{r\log r}a(r)\,\mathrm{d}r\,.
\]
We have
\[
u_{n}-u=c_{n}+\int_{0}^{\left|\bx\right|}\frac{1}{r\log r}\left[a_{n}(r)-a(r)\right]\mathrm{d}r\,,
\]
where
\[
c_{n}=-\int_{0}^{\infty}\frac{1}{r\log r}a_{n}(r)\,\rd r\,.
\]
Therefore, $u_{n}-c_{n}-u$ vanishes on $B(\bzero,n)$, the sequence
$\left(u_{n}-c_{n}\right)_{n\in\mathbb{N}}$ converges to $u$ in
$L^{4}(B)$ for all bounded domain $B$, but $\left(u_{n}\right)_{n\in\mathbb{N}}$
doesn't converge in $L^{4}(B)$.\end{proof}
\begin{defn}[spaces of divergence-free vector fields]
We denote by $D_{\sigma}^{1,2}(\Omega)$ the subspace of divergence-free
vector fields of $D^{1,2}(\Omega)$,
\[
D_{\sigma}^{1,2}(\Omega)=\left\{ \bu\in D^{1,2}(\Omega):\,\bnabla\bcdot\bu=0\right\} .
\]
We denote by $D_{0,\sigma}^{1,2}(\Omega)$ the subspace of $D_{0}^{1,2}(\Omega)$
defined as the completion of $C_{0,\sigma}^{\infty}(\Omega)$ in the
semi-norm $\left|\cdot\right|_{1,2}$.\index{Space!D0s@$D_{0,\sigma}^{1,2}$}
\end{defn}
Finally, we recall the following standard compactness result of \citet{Rellich-EinSatzuber1930}\textendash{}\citet{Kondrachov-certainpropertiesfunctions1945}:
\begin{lem}[{\citealp[Theorem 9.16]{Brezis-FunctionalAnalysisSobolev2011}}]
\label{lem:compact-embedding}If $\Omega$ is a bounded Lipschitz
domain, the embedding $W^{1,2}(\Omega)\subset L^{p}(\Omega)$ is compact
for $p\geq1$ if $n=2$ and for $1\leq p<6$ if $n=3$.
\end{lem}

\section{Existence of an extension}

This section is devoted to the construction of an extension\index{Extension!existence}\index{Existence!extension}
$\ba$ of the boundary condition $\bu^{*}\in W^{1/2,2}(\partial\Omega)$
that satisfies the so called extension condition\index{Extension!condition},
\emph{i.e.} such that
\[
\left|\left(\bv\bcdot\bnabla\ba,\bv\right)\right|\leq\varepsilon\left\Vert \bnabla\bv\right\Vert _{2}^{2}\,,
\]
for some $\varepsilon>0$ small enough. The proofs of the following
two lemmas are inspired by \citet[Lemma III.6.2, Lemma IX.4.1, Lemma IX.4.2, Lemma X.4.1,]{Galdi-IntroductiontoMathematical2011}
and by \citet{Russo-NoteExteriorTwo-Dimensional2009} for the two-dimensional
unbounded case. We first define admissible domains and boundary conditions
which will be required for the existence of an extension satisfying
the extension condition.
\begin{defn}[admissible domain]
\index{Admissible domain}\index{Domain!admissible}An admissible
domain is a Lipschitz domain $\Omega\subset\mathbb{R}^{n}$, $n=2,3$
such that $\mathbb{R}^{n}\setminus\Omega$ is composed of a finite
number $k\in\mathbb{N}$ of bounded simply connected components (single
points are not allowed), denoted by $B_{i}$, $i=0\dots k$ and possibly
one unbounded component. The main possibilities are drawn in \figref{admissible-domains}.
\begin{figure}[h]
\includefigure{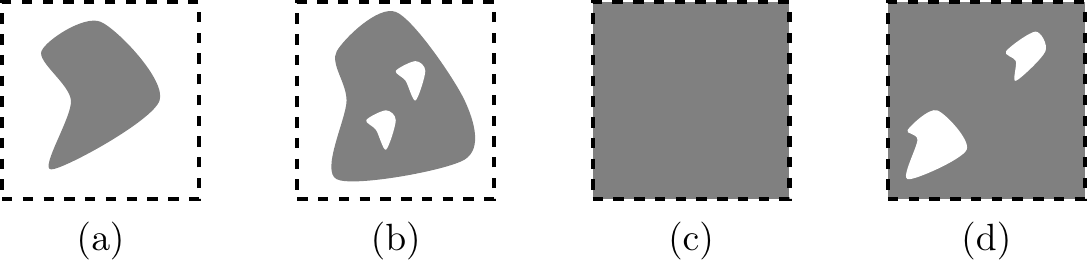}

\caption{\label{fig:admissible-domains}Admissible domains for the existence
of weak solutions: (a) a simply connected bounded domain; (b) a bounded
domain with holes; (c) the whole plane; (d) an exterior domain. }
\end{figure}

\end{defn}

\begin{defn}[admissible boundary condition]
\index{Admissible boundary condition}\index{Boundary condition!admissible}If
$\Omega$ is an admissible domain, an admissible boundary condition
is a field $\bu^{*}\in W^{1/2,2}(\partial\Omega)$, defined on the
boundary such that if $\Omega$ is bounded, the total flux is zero,
\[
\int_{\partial\Omega}\bu^{*}\bcdot\bn=0\,.
\]
We define the flux through each bounded component $B_{i}$ by
\[
\Phi_{i}=\int_{\partial B_{i}}\bu^{*}\bcdot\bn\,,
\]
and we denote the sum of the magnitude of the fluxes by $\Phi$,
\[
\Phi=\sum_{i=1}^{k}\left|\Phi_{i}\right|\,.
\]
\end{defn}
\begin{lem}
\label{lem:extention}If $\Omega$ is an admissible domain and $\bu^{*}\in W^{1/2,2}(\partial\Omega)$
an admissible boundary condition, then there exists an extension $\ba\in D_{\sigma}^{1,2}(\Omega)\cap L^{4}(\Omega)$
such that $\bu^{*}=\ba$ in the trace sense on $\partial\Omega$,
and moreover, there exists a constant $C>0$ depending on the domain
and on $\bu^{*}$ such that
\[
\bigl|\bigl(\bv\bcdot\bnabla\ba,\bv\bigr)\bigr|\leq\left(\frac{1}{4}+C\Phi\right)\left\Vert \bnabla\bv\right\Vert _{2}^{2}
\]
for all $\bv\in D_{0,\sigma}^{1,2}(\Omega)$.\end{lem}
\begin{proof}
For each $i=1,\dots,k$, there exists $\bx_{i}\in B_{i}$. We consider
the field $\ba_{\Phi}\in C_{\sigma}^{\infty}(\widebar{\Omega})$ defined
by
\[
\ba_{\Phi}(\bx)=\frac{1}{2\pi\left(n-1\right)}\sum_{i=1}^{k}\Phi_{i}\frac{\bx-\bx_{i}}{\left|\bx-\bx_{i}\right|^{n}}\,.
\]
By construction, the boundary field $\bu^{*}-\ba_{\Phi}$ has zero
flux through each connected component of $\partial\Omega$. Since
the connected components of the boundary $\partial\Omega$ are separated,
by using \lemref{extension-basic}, there exists $\delta>0$ and an
extension $\ba_{\delta}\in W_{\sigma}^{1,2}(\Omega)\cap L^{4}(\Omega)$
of $\bu^{*}-\ba_{\Phi}$ such that
\[
\bigl|\bigl(\bv\bcdot\bnabla\ba_{\delta},\bv\bigr)\bigr|\leq\frac{1}{4}\left\Vert \bnabla\bv\right\Vert _{2}^{2}\,.
\]
By integrating by parts and using that $\bv$ is divergence free,
we have
\[
\bigl(\bv\bcdot\bnabla\ba_{\Phi},\bv\bigr)=-\bigl(\bv\bcdot\bnabla\bv,\ba_{\Phi}\bigr)=\int_{\Omega}\bnabla\bcdot\left(\bv\bcdot\bnabla\bv\right)A_{\Phi}\,,
\]
where $A_{\Phi}$ is the potential of $\ba_{\Phi}$, \emph{i.e.} $\ba_{\Phi}=\bnabla A_{\Phi}$.
We note that in case $\Omega$ is bounded, we could easily conclude
the proof now, but not in the unbounded case. For $n=3$, we have
\[
\left|\int_{\Omega}\bnabla\bcdot\left(\bv\bcdot\bnabla\bv\right)A_{\Phi}\right|\leq\int_{\Omega}\left|\bnabla\bv\right|^{2}\left|A_{\Phi}\right|\leq\frac{1}{4\pi}\sum_{i=1}^{k}\left|\Phi_{i}\right|\sup_{x\in\Omega}\frac{1}{\left|\bx-\bx_{i}\right|}\left\Vert \bnabla\bv\right\Vert _{2}^{2}\leq C\sum_{i=1}^{k}\left|\Phi_{i}\right|\left\Vert \bnabla\bv\right\Vert _{2}^{2}\,,
\]
where $C>0$ is a constant depending on the domain. For $n=2$, by
using \citet[Theorem II.1]{Coifman-Compensatedcompactness1993}, $\bnabla\bcdot\left(\bv\bcdot\bnabla\bv\right)=\bnabla\bv:\left(\bnabla\bv\right)^{T}$
is in the Hardy space $\mathcal{H}^{1}$ and by using \citet[Proposition 12.11]{Taylor-PartialDifferentialEquations2011},
we obtain that the form $\left(\bnabla\bcdot\left(\bv\bcdot\bnabla\bv\right),A_{\Phi}\right)$
is bilinear and continuous for $\bv\in D_{\sigma}^{1,2}$, so there
exists a constant $C>0$ depending on the domain such that
\[
\left|\int_{\Omega}\bnabla\bcdot\left(\bv\bcdot\bnabla\bv\right)A_{\Phi}\right|\leq C\sum_{i=1}^{k}\left|\Phi_{i}\right|\left\Vert \bnabla\bv\right\Vert _{2}^{2}\,.
\]
Therefore, by choosing $\delta$ small enough, $\ba=\ba_{\Phi}+\ba_{\delta}$
satisfies the statement of the lemma.\end{proof}
\begin{lem}
\label{lem:extension-basic}Let $B$ be a bounded and simply connected
domain with smooth boundary. Let $\Omega$ be either $B$ or its complement
$\mathbb{R}^{2}\setminus\widebar B$. If $\bu^{*}\in W^{1/2,2}(\partial\Omega)$
is an admissible boundary condition with $\Phi=0$, then for all $\delta>0$,
there exists an extension $\ba_{\delta}\in W_{\sigma}^{1,2}(\Omega)\cap L^{4}(\Omega)$
of $\bu^{*}$ having support in a tube of weight $\delta$ around
the boundary, \emph{i.e.} in $\left\{ \bx\in\Omega\,:\,\dist(\bx,\partial\Omega)\leq2\delta\right\} $,
and such that
\[
\left|\left(\bv\bcdot\bnabla\ba_{\delta},\bv\right)\right|\leq\frac{C}{\left|\log\delta\right|}\left\Vert \bnabla\bv\right\Vert _{2}^{2}
\]
for all $\bv\in D_{0,\sigma}^{1,2}(\Omega)$ where $C>0$ is a constant
depending on the domain and on $\bu^{*}$.\end{lem}
\begin{proof}
We will construct an extension having support near the boundary of
$\Omega$. If $\Omega$ is unbounded, we can truncate the domain to
some large enough ball and therefore, without lost of generality,
we consider that $\Omega$ is bounded. Since $\bu^{*}\in W^{1/2,2}(\partial\Omega)$
has zero flux, there exists $\boldsymbol{\psi}\in W^{2,2}(\Omega)$
such that $\bu^{*}=\bnabla\bwedge\boldsymbol{\psi}$ on $\partial\Omega$
in the trace sense \citep{Galdi-existencesteadymotions1991}. By \citet[Chapter VI, Theorem 2]{Stein-Singularintegralsand1970},
there exists a function $\rho\in C^{\infty}(\Omega)$ and $\kappa>0$,
such that
\begin{align*}
\frac{1}{\kappa}\rho(\bx) & \leq\dist(\bx,\partial\Omega)\leq\rho(\bx)\,, & \left|\bnabla\rho(\bx)\right| & \leq\kappa\,.
\end{align*}
We define
\[
\Omega_{\delta}=\left\{ \bx\in\Omega\,:\,\dist(\bx,\partial\Omega)\leq\delta\right\} .
\]
Let $\chi_{\delta}\in C_{0}^{\infty}(\mathbb{R},[0,1])$ be a smooth
function such that $\chi_{\delta}(r)=1$ if $r\leq\delta^{2}/2$ and
$\chi_{\delta}(r)=0$ if $r\geq2\delta$, and moreover $\left|\chi_{\delta}^{\prime}(r)\right|\leq r^{-1}\left|\log(\delta)\right|^{-1}$.
We define $\xi_{\delta}=\chi_{\delta}\circ\rho$ so that $\xi_{\delta}\in C_{0}^{\infty}(\widebar{\Omega})$,
$\psi_{\delta}(\bx)=1$ $\dist(\bx,\partial\Omega)\leq\frac{\delta^{2}}{2\kappa}$,
and $\xi_{\delta}(\bx)=0$ if $\dist(\bx,\partial\Omega)\geq2\delta$.
Moreover,
\[
\left|\bnabla\xi_{\delta}(\bx)\right|=\left|\chi_{\delta}^{\prime}(\rho(\bx))\left|\bnabla\rho(\bx)\right|\right|\leq\frac{\kappa}{\rho(\bx)\left|\log(\delta)\right|}\,.
\]
By setting $\boldsymbol{\psi}_{\delta}=\xi_{\delta}\boldsymbol{\psi}$,
$\ba_{\delta}=\bnabla\bwedge\boldsymbol{\psi}_{\delta}$ is an extension
of $\bu^{*}$, which has support in $\Omega_{\delta}=\left\{ \bx\in\Omega\,:\,\dist(\bx,\partial\Omega)\leq2\delta\right\} $.

Since
\[
\ba_{\delta}=\xi_{\delta}\bnabla\bwedge\boldsymbol{\psi}+\bnabla\xi_{\delta}\bwedge\boldsymbol{\psi}\,,
\]
we have
\begin{align*}
\left\Vert \left|\bv\right|\left|\ba_{\delta}\right|\right\Vert _{2} & \leq\left\Vert \left|\bv\right|\left|\xi_{\delta}\right|\left|\bnabla\boldsymbol{\psi}\right|\right\Vert _{2}+\left\Vert \left|\bv\right|\left|\bnabla\xi_{\delta}\right|\left|\boldsymbol{\psi}\right|\right\Vert _{2}\\
 & \leq\bigl\Vert\bv;L^{4}(\Omega_{\delta})\bigr\Vert\left\Vert \bnabla\boldsymbol{\psi}\right\Vert _{4}+\frac{\kappa}{\left|\log\delta\right|}\left\Vert \left|\bv/\rho\right|\left|\boldsymbol{\psi}\right|\right\Vert _{2}\,.
\end{align*}
By using the Hölder inequality and Sobolev embeddings, we have
\begin{align*}
\bigl\Vert\bv;L^{4}(\Omega_{\delta})\bigr\Vert\left\Vert \bnabla\boldsymbol{\psi}\right\Vert _{4} & \leq C_{1}\bigl\Vert1;L^{12}(\Omega_{\delta})\bigr\Vert\bigl\Vert\bv;L^{6}(\Omega_{\delta})\bigr\Vert\bigl\Vert\boldsymbol{\psi};W^{2,2}(\Omega)\bigr\Vert\\
 & \leq\frac{C_{2}}{\left|\log\delta\right|}\bigl\Vert\bnabla\bv\bigr\Vert_{2}\bigl\Vert\boldsymbol{\psi};W^{2,2}(\Omega)\bigr\Vert\,,
\end{align*}
and by Hardy inequality,
\[
\left\Vert \left|\bv/\rho\right|\left|\boldsymbol{\psi}\right|\right\Vert _{2}\leq C_{3}\left\Vert \bnabla\bv\right\Vert _{2}\left\Vert \boldsymbol{\psi}\right\Vert _{\infty}\leq C_{4}\left\Vert \bnabla\bv\right\Vert _{2}\bigl\Vert\boldsymbol{\psi};W^{2,2}(\Omega)\bigr\Vert\,,
\]
where $C_{i}$ are constants depending only on the domain $\Omega$.
Therefore, there exists a constant $C>0$ depending on the domain
$\Omega$ such that
\[
\left\Vert \left|\bv\right||\ba_{\delta}|\right\Vert _{2}\leq\frac{C}{\left|\log\delta\right|}\left\Vert \bnabla\bv\right\Vert _{2}\bigl\Vert\boldsymbol{\psi};W^{2,2}(\Omega)\bigr\Vert\,,
\]
and finally by integrating by parts, we obtain the claimed bound
\[
\bigl|\bigl(\bv\bcdot\bnabla\ba_{\delta},\bv\bigr)\bigr|\leq\bigl|\bigl(\bv\bcdot\bnabla\bv,\ba_{\delta}\bigr)\bigr|\leq\left\Vert \bnabla\bv\right\Vert _{2}\left\Vert \left|\bv\right||\ba_{\delta}|\right\Vert _{2}\leq\frac{C\bigl\Vert\boldsymbol{\psi};W^{2,2}(\Omega)\bigr\Vert}{\left|\log\delta\right|}\left\Vert \bnabla\bv\right\Vert _{2}^{2}\,.
\]

\end{proof}

\section{\label{sec:existence-weak}Existence of weak solutions}
\begin{defn}
\label{def:weak-solution}A vector field $\bu:\Omega\to\mathbb{R}^{n}$
is called a weak solution\index{Weak solutions!definition}\index{Navier-Stokes equations!weak solutions}
to \eqref{intro-ns} if
\begin{enumerate}
\item $\bu\in D_{\sigma}^{1,2}(\Omega)$;
\item $\left.\bu\right|_{\partial\Omega}=\bu^{*}$ in the trace sense;
\item $\bu$ satisfies
\begin{equation}
\bigl(\bnabla\bu,\bnabla\bphi\bigr)+\bigl(\bu\bcdot\bnabla\bu,\bphi\bigr)+\bigl(\bff,\bphi\bigr)=0\label{eq:weak-solution-short}
\end{equation}
for all $\bphi\in C_{0,\sigma}^{\infty}(\Omega)$.
\end{enumerate}
\end{defn}
\begin{rem}
We note that in this definition, there is no mention of the limit
of $\bu$ at infinity in case $\Omega$ is unbounded. The limit of
$\bu$ at infinity will be discussed in \secref{weal-solution-limit}.
\end{rem}
If the total $\Phi$ is small enough, there exists a weak solution
as stated by:
\begin{thm}
\label{thm:existence-weak-solutions}\index{Weak solutions!existence}\index{Existence!weak solutions}\index{Navier-Stokes equations!weak solutions}If
$\Omega\neq\mathbb{R}^{2}$ is an admissible domain, $\bu^{*}\in W^{1/2,2}(\partial\Omega)$
an admissible boundary condition with $\Phi$ small enough, there
exists a weak solution to \eqref{intro-ns}, provided $\left(\bff,\bphi\right)$
defines a linear functional on $\bphi\in D_{0,\sigma}^{1,2}(\Omega)$.\end{thm}
\begin{rem}
In symmetric unbounded domains, \citet{Korobkov-existence2Dsymmetric2014,Korobkov-existence3Dsymmetric2014}
showed the existence of a weak solution for arbitrary large $\Phi$.
This was recently improved by \citet{Korobkov-LerayProblem2015} that
showed the existence of weak solutions in two-dimensional bounded
domains without any symmetry and smallness assumptions.
\end{rem}

\begin{rem}
\citet[pp. 36--37]{Ladyzhenskaya-MathematicalTheory1963} listed some
conditions on $\bff$, so that $\left(\bff,\bphi\right)$ defines
a linear functional on $\bphi\in D_{0,\sigma}^{1,2}(\Omega)$.\end{rem}
\begin{proof}
We treat the case where $\Omega$ is bounded and unbounded in parallel.
In case $\Omega$ is bounded, we set $\bu_{\infty}=\bzero$ by convenience
in what follows. By using Riesz' theorem, there exists $\bF\in D_{0,\sigma}^{1,2}(\Omega)$,
such that
\[
\bigl[\bF,\bphi\bigr]=\bigl(\bff,\bphi\bigr)\,,
\]
where $\left[\cdot,\cdot\right]$ denotes the scalar product in $D_{0,\sigma}^{1,2}(\Omega)$.
We look for a solution of the form $\bu=\bu_{\infty}+\ba+\bv$, where
$\ba$ is the extension of $\bu^{*}-\bu_{\infty}$ given by \lemref{extention},
so that $\bv$ vanishes at the boundary and with the hope that $\bv$
will converges to zero for large $\bx$ in case $\Omega$ is unbounded.
\begin{enumerate}
\item We first treat the case where $\Omega$ is bounded, so that $D_{0,\sigma}^{1,2}(\Omega)=W_{0,\sigma}^{1,2}(\Omega)$,
and $D_{0,\sigma}^{1,2}(\Omega)$ is compactly embedded in $L^{4}(\Omega)$.
First of all, by integrating by parts we have since $\bu$ is divergence-free,
\[
\bigl(\bu\bcdot\bnabla\bu,\bphi\bigr)+\bigl(\bu\bcdot\bnabla\bphi,\bu\bigr)=\int_{\Omega}\bnabla\bcdot\left(\bu\bcdot\bphi\,\bu\right)=\int_{\partial\Omega}\bu\bcdot\bphi\,\bu\bcdot\bn=0\,.
\]
By the Riesz' theorem there exists $\boldsymbol{B}\in D_{0,\sigma}^{1,2}(\Omega)$
such that
\[
\bigl[\boldsymbol{B},\bphi\bigr]=\bigl(\bnabla\ba,\bnabla\bphi\bigr)+\bigl(\ba\bcdot\bnabla\ba,\bphi\bigr)\,,
\]
because
\[
\bigl|\bigl[\boldsymbol{B},\bphi\bigr]\bigr|\leq\bigl|\bigl(\bnabla\ba,\bnabla\bphi\bigr)\bigr|+\bigl|\bigl(\ba\bcdot\bnabla\bphi,\ba\bigr)\bigr|\leq\left\Vert \bnabla\ba\right\Vert _{2}\left\Vert \bnabla\bphi\right\Vert _{2}+\left\Vert \ba\right\Vert _{4}^{2}\left\Vert \bnabla\bphi\right\Vert _{2}\,.
\]
In the same way, there exists a map $A:D_{0,\sigma}^{1,2}(\Omega)\to D_{0,\sigma}^{1,2}(\Omega)$
such that
\[
\bigl[A\bv,\bphi\bigr]=\bigl(\ba\bcdot\bnabla\bv,\bphi\bigr)+\bigl(\bv\bcdot\bnabla\ba,\bphi\bigr)+\bigl(\bv\bcdot\bnabla\bv,\bphi\bigr)\,,
\]
because 
\[
\bigl|\bigl[A\bv,\bphi\bigr]\bigr|\leq\bigl|\bigl(\ba\bcdot\bnabla\bphi,\bv\bigr)\bigr|+\bigl|\bigl(\bv\bcdot\bnabla\bphi,\ba\bigr)\bigr|+\bigl|\bigl(\bv\bcdot\bnabla\bphi,\bv\bigr)\bigr|\leq\left(2\left\Vert \ba\right\Vert _{4}+\left\Vert \bv\right\Vert _{4}\right)\left\Vert \bnabla\bphi\right\Vert _{2}\left\Vert \bv\right\Vert _{4}
\]
Since $D_{0,\sigma}^{1,2}(\Omega)$ is compactly embedded in $L^{4}(\Omega)$,
the map $A$ is continuous on $D_{0,\sigma}^{1,2}(\Omega)$ when equipped
with the $L^{4}$-norm and therefore is completely continuous on $D_{0,\sigma}^{1,2}(\Omega)$
when equipped with its underlying norm.

The condition \eqref{weak-solution-short} is equivalent to
\[
\bigl[\bv+A\bv+\boldsymbol{B}+\bF,\bphi\bigr]=0\,,
\]
which corresponds to solving the nonlinear equation
\begin{equation}
\bv+A\bv+\boldsymbol{B}+\bF=0\label{eq:ns-functional}
\end{equation}
in $D_{0,\sigma}^{1,2}(\Omega)$. From the Leray-Schauder fixed point
theorem \citep[see for example][Theorem 11.6]{Gilbarg.Trudinger-EllipticPartialDifferential1988}
to prove the existence of a solution to \eqref{ns-functional} it
is sufficient to prove that the set of solutions $\bv$ of the equation
\begin{equation}
\bv+\lambda\left(A\bv+\boldsymbol{B}+\bF\right)=0\label{eq:ns-functional-lambda}
\end{equation}
is uniformly bounded in $\lambda\in\left[0,1\right]$. To this end,
we take the scalar product of \eqref{ns-functional-lambda} with $\bv$,
\[
\bigl(\bnabla\bv,\bnabla\bv\bigr)+\lambda\bigl(\bv\bcdot\bnabla\ba,\bv\bigr)+\lambda\bigl(\bu\bcdot\bnabla\bv,\bv\bigr)+\lambda\bigl(\bnabla\ba,\bnabla\bv\bigr)+\lambda\bigl(\ba\bcdot\bnabla\ba,\bv\bigr)+\lambda\bigl(\bff,\bv\bigr)=0\,.
\]
where $\bu=\ba+\bv$. We have
\[
\bigl(\bu\bcdot\bnabla\bv,\bv\bigr)=\frac{1}{2}\int_{\Omega}\bnabla\bcdot\left(\bv\bcdot\bv\,\bu\right)=\frac{1}{2}\int_{\partial\Omega}\left(\bv\bcdot\bv\,\bu\bcdot\bn\right)=0\,,
\]
and by \lemref{extention}, if $\Phi$ is small enough,
\[
\bigl|\bigl(\bv\bcdot\bnabla\ba,\bv\bigr)\bigr|\leq\frac{1}{2}\left\Vert \bnabla\bv\right\Vert _{2}^{2}\,,
\]
so by Hölder inequality, we obtain
\begin{align*}
\left\Vert \bnabla\bv\right\Vert _{2}^{2} & \leq\frac{1}{2}\left\Vert \bnabla\bv\right\Vert _{2}^{2}+\left\Vert \bnabla\ba\right\Vert _{2}\left\Vert \bnabla\bv\right\Vert _{2}+\left\Vert \ba\right\Vert _{4}^{2}\left\Vert \bnabla\bv\right\Vert _{2}+\left\Vert \bnabla\bF\right\Vert _{2}\left\Vert \bnabla\bv\right\Vert _{2}\\
 & \leq\frac{1}{2}\left\Vert \bnabla\bv\right\Vert _{2}^{2}+\left(\left\Vert \bnabla\ba\right\Vert _{2}+\left\Vert \ba\right\Vert _{4}^{2}\right)\left\Vert \bnabla\bv\right\Vert _{2}+\left\Vert \bnabla\bF\right\Vert _{2}\left\Vert \bnabla\bv\right\Vert _{2}\,.
\end{align*}
Consequently, we have
\[
\left\Vert \bnabla\bv\right\Vert _{2}\leq\frac{\left\Vert \bnabla\ba\right\Vert _{2}+\left\Vert \ba\right\Vert _{4}^{2}+\left\Vert \bnabla\bF\right\Vert _{2}}{2}\,.
\]

\item We now consider the case where $\Omega$ is unbounded. There exists
$R>0$ such that $\mathbb{R}^{n}\setminus\Omega$ is contained in
$B(\bzero,R)$. For $n\in\mathbb{N}$, we consider the domains $\Omega_{n}=\Omega\cap B(\bzero,R+n)$.
By the existence result for the bounded case, there exists for each
$n\in\mathbb{N}$ a weak solution $\bu_{n}=\bu_{\infty}+\ba+\bv_{n}$,
where $\bv_{n}\in D_{0,\sigma}^{1,2}(\Omega_{n})$ to \eqref{intro-ns-eq}
in $\Omega_{n}$, with $\bu^{*}=\bu_{\infty}+\ba$ on $\partial\Omega_{n}$.
By extending $\bv_{n}$ to $\Omega$ by setting $\bv_{n}=\bzero$
on $\Omega\setminus\Omega_{n}$, then $\bv_{n}\in D_{0,\sigma}^{1,2}(\Omega)$
and the sequence $\left(\bv_{n}\right)_{n\in\mathbb{N}}$ is bounded
in $D_{0,\sigma}^{1,2}(\Omega)$. Therefore, there exists a subsequence,
denoted also by $\left(\bv_{n}\right)_{n\in\mathbb{N}}$, which converges
weakly to some $\bv$ in $D_{0,\sigma}^{1,2}(\Omega)$. We now show
that $\bu=\bu_{\infty}+\ba+\bv$ is a weak solution to \eqref{intro-ns}
in $\Omega$. Given $\bphi\in C_{0,\sigma}^{\infty}(\Omega)$, there
exists $m\in\mathbb{N}$ such that the support of $\bphi$ is contained
in $\Omega_{m}$. Therefore, for any $n\geq m$, we have
\[
\bigl(\bnabla\bu_{n},\bnabla\bphi\bigr)+\bigl(\bu_{n}\bcdot\bnabla\bu_{n},\bphi\bigr)+\bigl(\bff,\varphi\bigr)=0\,,
\]
and it only remains to show that the equation is valid in the limit
$n\to\infty$. By definition of the weak convergence, 
\[
\lim_{n\to\infty}\bigl(\bnabla\bv_{n},\bnabla\bphi\bigr)=\bigl(\bnabla\bv,\bnabla\bphi\bigr)\,,
\]
and since $\bphi$ has compact support in $\Omega_{m}$, 
\begin{align*}
\bigl|\bigl(\bu_{n}\bcdot\bnabla\bu_{n}-\bu\bcdot\bnabla\bu,\bphi\bigr)\bigr| & \leq\bigl|\left(\left(\bu_{n}-\bu\right)\bcdot\bnabla\bu_{n},\bphi\right)\bigr|+\bigl|\left(\bu\bcdot\left(\bnabla\bu_{n}-\bnabla\bu\right),\bphi\right)\bigr|\\
 & \leq\bigl|\left(\left(\bu_{n}-\bu\right)\bcdot\bnabla\bu_{n},\bphi\right)\bigr|+\bigl|\left(\bu\bcdot\bnabla\bphi,\bv_{n}-\bv\right)\bigr|\\
 & \leq\bigl|\left(\left(\bv_{n}-\bv\right)\bcdot\bnabla\bu_{n},\bphi\right)\bigr|+\bigl|\left(\bu\bcdot\bnabla\bphi,\bv_{n}-\bv\right)\bigr|\\
 & \leq\left(\left\Vert \bnabla\bu_{n}\right\Vert _{2}\left\Vert \bphi\right\Vert _{4}+\bigl\Vert\bu;L^{4}(\Omega_{m})\bigr\Vert\left\Vert \bnabla\bphi\right\Vert _{2}\right)\bigl\Vert\bv_{n}-\bv;L^{4}(\Omega_{m})\bigr\Vert\,.
\end{align*}
By \lemref{local-D0}, the sequence $\left(\bv_{n}\right)_{n\in\mathbb{N}}$
is bounded in $W^{1,2}(\Omega_{m})$ and by  \lemref{compact-embedding},
there exists a subsequence also denoted by $\left(\bv_{n}\right)_{n\in\mathbb{N}}$
which converges strongly to $\bv$ in $L^{4}(\Omega_{m})$. Therefore,
\[
\lim_{n\to\infty}\bigl(\bu_{n}\bcdot\bnabla\bu_{n},\bphi\bigr)=\bigl(\bu\bcdot\bnabla\bu,\bphi\bigr)\,.
\]
and $\bu$ satisfies \eqref{weak-solution-short}.
\end{enumerate}
\end{proof}

\section{Regularity of weak solutions}

A weak solution\index{Weak solutions!regularity}\index{Regularity of weak solutions}\index{Navier-Stokes equations!regularity}
is a vector field $\bu\in D_{\sigma}^{1,2}(\Omega)$ that satisfies
the Navier-Stokes equations in a variational way and therefore a weak
solution is defined even for low regularity on the data $\bu^{*}$
and $\bff$ and does not necessarily satisfies the equations in a
classical way. By assuming more regularity on the data, any weak solution
becomes more regular and satisfies the Navier-Stokes equations in
the classical way. The following theorem states this fact:
\begin{thm}[{\citealp[Theorems IX.5.1, IX.5.2 and X.1.1]{Galdi-IntroductiontoMathematical2011}}]
Let $\bu$ be a weak solution according to \defref{weak-solution}.
The following properties hold:
\begin{enumerate}
\item For $m\ge1$ if $\bff\in W_{loc}^{m,2}(\Omega)$, then $\bu\in W_{loc}^{m+2,2}(\Omega)$
and $p\in W_{loc}^{m+1,2}(\Omega)$.
\item If $\Omega$ is a smooth domain, $\bu^{*}\in C^{\infty}(\partial\Omega)$
and $\bff\in C^{\infty}(\widebar{\Omega})$, then $\bu,p\in C^{\infty}(\widebar{\Omega})$.
\end{enumerate}
\end{thm}

\section{\label{sec:weal-solution-limit}Limit of the velocity at large distances}

We start with two lemmas \citep[\S 1.4]{Ladyzhenskaya-MathematicalTheory1963}
on the behavior at infinity of functions in $D_{0}^{1,2}(\Omega)$,
with $\Omega$ unbounded. Due to the presence of a logarithm if $n=2$,
the discussion of the validity of
\begin{equation}
\lim_{\left|\bx\right|\to\infty}\bu=\bu_{\infty}\,,\label{eq:weak-limit}
\end{equation}
for a weak solution depends drastically on the dimension.
\begin{lem}[{\citealp[Theorem II.6.1]{Galdi-IntroductiontoMathematical2011}}]
\label{lem:hardy-3d}For $n\geq3$, if $\Omega\subset\mathbb{R}^{n}$
is an unbounded Lipschitz domain, then for all $\bu\in D_{0}^{1,2}(\Omega)$,\index{Hardy inequality!in three dimensions}\index{Inequality!Hardy!in three dimensions}
\[
\left\Vert \frac{\bu}{\left|\bx\right|}\right\Vert _{2}\leq\frac{2}{n-2}\left\Vert \bnabla\bu\right\Vert _{2}\,.
\]
\end{lem}
\begin{proof}
It suffices to prove the inequality for a scalar field $u\in C_{0}^{\infty}(\mathbb{R}^{n})$.
Since
\[
\bnabla\bcdot\left(\frac{\bx}{\left|\bx\right|^{2}}\right)=\frac{n-2}{\left|\bx\right|^{2}}\,,
\]
we have by integrating by parts,
\[
\int_{\mathbb{R}^{n}}\frac{u^{2}}{\left|\bx\right|^{2}}=-\frac{1}{n-2}\int_{\mathbb{R}^{n}}\frac{\bx}{\left|\bx\right|^{2}}\bcdot\bnabla\left(u^{2}\right)=-\frac{2}{n-2}\int_{\mathbb{R}^{n}}\frac{\bx}{\left|\bx\right|^{2}}\bcdot\bnabla u\, u\,.
\]
Then by Schwarz inequality, we obtain
\[
\left\Vert \frac{u}{\left|\bx\right|}\right\Vert _{2}^{2}\leq\frac{2}{n-2}\left\Vert \frac{\bx}{\left|\bx\right|^{2}}u\right\Vert _{2}\left\Vert \bnabla u\right\Vert _{2}\leq\frac{2}{n-2}\left\Vert \frac{u}{\left|\bx\right|}\right\Vert _{2}\left\Vert \bnabla u\right\Vert _{2}\,,
\]
and the inequality is proved.\end{proof}
\begin{lem}[{\citealp[Theorem II.6.1]{Galdi-IntroductiontoMathematical2011}}]
\label{lem:hardy-2d-log}If $\Omega\subset\mathbb{R}^{2}$ is an
exterior Lipschitz domain such that $B(\bzero,\varepsilon)\subset\mathbb{R}^{2}\setminus\Omega$
for some $\varepsilon>0$, then for all $\bu\in D_{0}^{1,2}(\Omega)$,\index{Hardy inequality!in two dimensions}\index{Inequality!Hardy!in two dimensions}
\[
\left\Vert \frac{\bu}{\left|\bx\right|\log(\left|\bx\right|/\varepsilon)}\right\Vert _{2}\leq2\left\Vert \bnabla\bu\right\Vert _{2}\,.
\]
\end{lem}
\begin{proof}
Again, it is sufficient to prove the inequality for the scalar field
$u\in C_{0}^{\infty}(\mathbb{R}^{2}\setminus\widebar B(\bzero,\varepsilon))$.
Since
\[
\bnabla\bcdot\left(\frac{\bx}{\left|\bx\right|^{2}\log(\left|\bx\right|/\varepsilon)}\right)=-\frac{1}{\left|\bx\right|^{2}\log^{2}(\left|\bx\right|/\varepsilon)}\,,
\]
by integrating by parts,
\[
\int_{\mathbb{R}^{2}}\frac{u^{2}}{\left|\bx\right|^{2}\log^{2}(\left|\bx\right|/\varepsilon)}=\int_{\mathbb{R}^{2}}\frac{\bx}{\left|\bx\right|^{2}\log(\left|\bx\right|/\varepsilon)}\bcdot\bnabla\left(u^{2}\right)=\int_{\mathbb{R}^{2}}\frac{\bx}{\left|\bx\right|^{2}\log(\left|\bx\right|/\varepsilon)}\bcdot\bnabla u\, u\,.
\]
Then the lemma is proven by using the Schwartz inequality,
\[
\left\Vert \frac{u}{\left|\bx\right|\log(\left|\bx\right|/\varepsilon)}\right\Vert _{2}^{2}\leq2\left\Vert \frac{\bx}{\left|\bx\right|^{2}\log(\left|\bx\right|/\varepsilon)}u\right\Vert _{2}\left\Vert \bnabla u\right\Vert _{2}\leq2\left\Vert \frac{u}{\left|\bx\right|\log(\left|\bx\right|/\varepsilon)}\right\Vert _{2}\left\Vert \bnabla u\right\Vert _{2}\,.
\]

\end{proof}

\subsection{Three dimensions}

By using \lemref{hardy-3d}, we can now prove that a function in $D_{0}^{1,2}(\Omega)$
tends to zero at infinity. In what follows, we set $B_{r}=B(\bzero,r)$.\index{Weak solutions!limit of the velocity}\index{Limit of the velocity}\index{Navier-Stokes equations!limit of the velocity}
\begin{lem}
\label{lem:limit-3d}For $n=3$, if $\bu\in D_{0}^{1,2}(\Omega)$,
then
\[
\int_{S^{2}}\left|\bu\right|^{2}=O(\left|\bx\right|^{-1})\,,
\]
where $S^{2}\subset\mathbb{R}^{3}$ is the sphere of unit radius,
or more precisely
\[
\frac{1}{\left|\partial B_{r}\right|}\int_{\partial B_{r}}\left|\bu\right|^{2}=O(r^{-1})\,.
\]
\end{lem}
\begin{proof}
There exists $R>0$ such that $\mathbb{R}^{3}\setminus\Omega\subset B_{R}$.
For $r\geq1$. By the trace theorem in $B_{R}$, there exists $C>0$
such that for all $\bu\in W^{1,2}(B_{R})$,
\[
\bigl\Vert\bu;L^{2}(\partial B_{R})\bigr\Vert^{2}\leq C\left(\bigl\Vert\bu;L^{2}(B_{R})\bigr\Vert^{2}+\bigl\Vert\bnabla\bu;L^{2}(B_{R})\bigr\Vert^{2}\right)\,.
\]
By a scaling argument, we have, for all $r\geq R$, 
\begin{align*}
\frac{R^{2}}{r^{2}}\bigl\Vert\bu;L^{2}(\partial B_{r})\bigr\Vert^{2} & \leq\frac{CR^{3}}{r^{3}}\bigl\Vert\bu;L^{2}(B_{r})\bigr\Vert^{2}+\frac{CR}{r}\bigl\Vert\bnabla\bu;L^{2}(B_{r})\bigr\Vert^{2}\\
 & \leq\frac{CR\left(1+R^{2}\right)}{r}\left[\bigl\Vert\bu/\left|\bx\right|;L^{2}(B_{r})\bigr\Vert^{2}+\bigl\Vert\bnabla\bu;L^{2}(B_{r})\bigr\Vert^{2}\right]\,.
\end{align*}
By using \lemref{hardy-3d}, we have for some $C>0$ independent of
$r$,
\[
\frac{1}{r^{2}}\bigl\Vert\bu;L^{2}(\partial B_{r})\bigr\Vert^{2}\leq\frac{C}{r}\bigl\Vert\bnabla\bu;L^{2}(\Omega)\bigr\Vert^{2}\,.
\]
Since $\left|\partial B_{r}\right|=4\pi r^{2}$, this completes the
proof.
\end{proof}
By applying this lemma to the weak solution constructed in \secref{existence-weak},
we obtain its behavior at infinity:
\begin{prop}
Let the hypothesis of \thmref{existence-weak-solutions} be satisfied,
so that there exists a weak solution $\bu\in D_{0,\sigma}^{1,2}$.
In case $\Omega\subset\mathbb{R}^{3}$ is unbounded, we have \eqref{weak-limit}
in the following sense
\[
\int_{S^{2}}\left|\bu-\bu_{\infty}\right|^{2}=O(\left|\bx\right|^{-1})\,.
\]
\end{prop}
\begin{proof}
The weak solution has the form $\bu-\bu_{\infty}=\ba+\bv$. By construction,
$\ba$ has one part of compact support, and one part carrying the
fluxes decaying like $\left|\bx\right|^{-2}$, so $\ba=O(\left|\bx\right|^{-2})$.
By applying \lemref{limit-3d} to $\bv\in D_{\sigma,0}^{1,2}(\Omega)$,
we obtain the claimed result.
\end{proof}

\subsection{Two dimensions}

In two dimensions,\index{Weak solutions!limit of the velocity}\index{Limit of the velocity}\index{Navier-Stokes equations!limit of the velocity}
the information contained in the space $D_{0}^{1,2}(\Omega)$ is not
sufficient to determine the limit of the velocity at infinity, mainly
due to the failure of \lemref{hardy-3d} for $n=2$. In fact a function
in $D_{\sigma,0}^{1,2}(\Omega)$ can even grow at infinity, as shown
by the following example. Therefore, the choice of $\bu_{\infty}$
is apparently completely lost during the construction of weak solutions.
\begin{example}
\label{expl:counter-example-hardy}Let $\Omega\subset\mathbb{R}^{2}$
be an unbounded Lipschitz domain. For $R>0$ such that $\mathbb{R}^{2}\setminus\Omega\subset B(\bzero,R)$,
let $\chi$ be a cut-off function such that $\chi(\bx)=0$ for $\left|\bx\right|\leq R$,
and $\chi(\bx)=1$ for $\left|\bx\right|\geq2R$. For $\nu\in\bigl[\frac{-1}{2};\frac{1}{2}\bigr)$,
the function $\bu=\bnabla\bwedge\left(\chi\psi\right)$, where
\[
\psi=-x_{2}\,\log^{\nu}\left|\bx\right|\,,
\]
satisfies $\bu\in D_{0,\sigma}^{1,2}(\Omega)$, $\bu/\left|\bx\right|\notin L^{2}(\Omega)$
and $\bu=O(\log^{\nu}\left|\bx\right|)$ at infinity.\end{example}
\begin{proof}
By construction, $\bu(\bx)=\bzero$ for $\left|\bx\right|\leq R$,
so in particular on $\partial\Omega$. For $\left|\bx\right|\geq2R$,
we have
\[
\bu=\log^{\nu}\left|\bx\right|\left(1+\nu\frac{x_{2}^{2}}{\left|\bx\right|^{2}}\frac{1}{\log\left|\bx\right|},-\nu\frac{x_{1}x_{2}}{\left|\bx\right|^{2}}\frac{1}{\log\left|\bx\right|}\right)\,,
\]
and
\[
\bnabla\bu=O\left(\frac{\log^{\nu}\left|\bx\right|}{\left|\bx\right|\log\left|\bx\right|}\right)\,.
\]
Since
\[
\frac{1}{s\log^{\alpha}s}\in L^{1}([2,\infty))\quad\Longleftrightarrow\quad\alpha>1\,,
\]
we obtain that
\begin{align*}
\bu/\left|\bx\right|\in L^{2}(\Omega)\quad & \Longleftrightarrow\quad\nu<\frac{-1}{2}\,,\\
\bnabla\bu\in L^{2}(\Omega)\quad & \Longleftrightarrow\quad\nu<\frac{1}{2}\,,
\end{align*}
and therefore, we obtained the desired behavior for $\nu\in\bigl[\frac{-1}{2};\frac{1}{2}\bigr)$.
\end{proof}
In two dimensions, the best known result concerning the behavior at
infinity is due to \citet{Gilbarg.Weinberger-AsymptoticPropertiesof1974,Gilbarg.Weinberger-Asymptoticpropertiesof1978}:
\begin{thm}[{\citealp[Theorem 3.3]{Galdi-StationaryNavier-Stokesproblem2004}}]
Let $(\bu,p)$ be a weak solution in an exterior domain $\Omega$
that contains an open ball. Let $L\in[0,\infty]$ be defined by

\[
L=\lim_{r\to\infty}\sup_{\theta\in[0,2\pi]}\left|\bu(r,\theta)\right|\,.
\]
If $L<\infty$, there exists $\boldsymbol{\xi}\in\mathbb{R}^{2}$
such that $\lim_{\left|\bx\right|\to\infty}\bu=\boldsymbol{\xi}$
in the following sense,
\[
\lim_{r\to\infty}\int_{0}^{2\pi}\left|\bu(r,\theta)-\boldsymbol{\xi}\right|^{2}\rd\theta=0\,,
\]
and if $L=\infty$, then
\[
\lim_{r\to\infty}\int_{0}^{2\pi}\left|\bu(r,\theta)\right|^{2}\rd\theta=\infty\,.
\]
Moreover, if $\bu^{*}=\bff=\bzero$, then $L<\infty$.
\end{thm}
However, the question of the finiteness of $L$ and of the coincidence
of $\boldsymbol{\xi}$ with the prescribed value $\bu_{\infty}$ is
still open. Unfortunately, the proof of the pointwise limit of $\bu$
obtained in \citet[Theorem 3.4]{Galdi-StationaryNavier-Stokesproblem2004}
is not correct due to a gap in the proof between (3.54) and (3.55)
when integrating over $\theta$.

In case the data are invariant under the central symmetry \eqref{central-symmetry},
we can prove that the velocity satisfies \eqref{weak-limit} with
$\bu_{\infty}=\bzero$. We first improve \lemref{hardy-2d-log} by
removing the logarithm. The following lemma improves the results of
\citet[Lemma 3.2]{Galdi-StationaryNavier-Stokesproblem2004} which
requires, in addition to the central symmetry, a reflection symmetry,
\emph{i.e.} the symmetry \eqref{symmetry-galdi}.
\begin{lem}
\label{lem:hardy-inequality-central-symmetry}Let $\Omega\subset\mathbb{R}^{2}$
by an exterior Lipschitz domain that is centrally symmetric and such
that there exists $\varepsilon>0$ with $B(\bzero,\varepsilon)\subset B$.
Then for any $\bu\in D^{1,2}(\Omega)$ that is centrally symmetric
\eqref{central-symmetry}, we have
\[
\left\Vert \frac{\bu}{\left|\bx\right|}\right\Vert _{2}\leq\frac{C}{\varepsilon}\left\Vert \bnabla\bu\right\Vert _{2}\,,
\]
where $C=C(\Omega)$.\end{lem}
\begin{proof}
First of all, since $\bu$ is centrally symmetric, we have
\[
\int_{\gamma}\bu=0\,,
\]
for $\gamma$ any centrally symmetric smooth curve and the average
of $\bu$ vanishes on any centrally symmetric bounded domain. Let
$B=\mathbb{R}^{2}\setminus\Omega$, so there exists $R>0$ such that
$B\subset B(\bzero,R)$. We denote by $B_{n}$ the ball $B_{n}=B(\bzero,nR)$
and by $S_{n}$ the shell
\begin{align*}
S_{0} & =B_{1}\setminus B\,, & S_{n} & =B_{2n}\setminus B_{n}\,,\quad\text{for}\quad n\geq1\,.
\end{align*}
By Poincaré inequality in $S_{n}$, there exists a constant $C_{n}>0$
such that
\[
\bigl\Vert\bu;L^{2}(S_{n})\bigr\Vert\leq C_{n}\bigl\Vert\bnabla\bu;L^{2}(S_{n})\bigr\Vert\,,
\]
for all $\bu\in W^{1,2}(\Omega)$ that are centrally symmetric, because
$\bar{\bu}_{S_{n}}=\bzero$. Since $\left|\bx\right|\geq\varepsilon$
by hypothesis, we obtain
\[
\bigl\Vert\bu/\left|\bx\right|;L^{2}(S_{n})\bigr\Vert\leq\frac{C_{n}}{\varepsilon}\bigl\Vert\bnabla\bu;L^{2}(S_{n})\bigr\Vert\,.
\]
But the domains $S_{n}$ are scaled versions of $S_{1}$, \emph{i.e.}
$S_{n}=nS_{1}$ for $n\geq1$ and therefore, since the two norms in
the previous inequality are scale invariant, we obtain that $C_{n}=C_{1}$,
for $n\geq1$. Now we have for $N\geq1$,
\begin{align*}
\bigl\Vert\bu/\left|\bx\right|;L^{2}(B_{2N}\setminus B)\bigr\Vert & =\sum_{n=0}^{N}\bigl\Vert\bu/\left|\bx\right|;L^{2}(S_{n})\bigr\Vert\leq\frac{1}{\varepsilon}\sum_{n=0}^{N}C_{n}\bigl\Vert\bnabla\bu;L^{2}(S_{n})\bigr\Vert\\
 & \leq\frac{C_{0}+C_{1}}{\varepsilon}\sum_{n=0}^{N}\bigl\Vert\bnabla\bu;L^{2}(S_{n})\bigr\Vert\leq\frac{C_{0}+C_{1}}{\varepsilon}\bigl\Vert\bnabla\bu;L^{2}(B_{2N}\setminus B)\bigr\Vert\,.
\end{align*}
Finally, by taking the limit $N\to\infty$, we have
\[
\bigl\Vert\bu/\left|\bx\right|\bigr\Vert_{2}\leq\frac{C}{\varepsilon}\bigl\Vert\bnabla\bu\bigr\Vert_{2}\,,
\]
for all $\bu\in D^{1,2}(\Omega)$ where $C=C_{0}+C_{1}$ depends on
$R$ only.
\end{proof}
Now, we can obtain the limit of a function $\bu\in D^{1,2}(\Omega)$
under the central symmetry.
\begin{lem}
\label{lem:limit-central-symmetry}If the hypothesis of \lemref{hardy-inequality-central-symmetry}
are satisfied, we have $\lim_{\left|\bx\right|\to\infty}\bu=\bzero$
in the following sense
\[
\lim_{r\to\infty}\int_{0}^{2\pi}\left|\bu(r,\theta)\right|^{2}\rd\theta=0\,,
\]
for all $\bu\in D^{1,2}(\Omega)$ that are invariant under the central
symmetry.\end{lem}
\begin{proof}
For $r>0$, we denote by $B_{r}$ the ball $B(\bzero,r)$ and by $S_{r}$
the shell $B_{2r}\setminus B_{r}$. Again, we define $R>0$ such that
$\mathbb{R}^{2}\setminus\Omega\subset B_{R}$. By the trace theorem
in $S_{R}$, there exists a constant $C>0$ such that
\[
\bigl\Vert\bu;L^{2}(\partial B_{R})\bigr\Vert^{2}\leq\bigl\Vert\bu;L^{2}(\partial S_{R})\bigr\Vert^{2}\leq C\bigl\Vert\bu;L^{2}(S_{R})\bigr\Vert^{2}+C\bigl\Vert\bnabla\bu;L^{2}(S_{R})\bigr\Vert^{2},
\]
for any $\bu\in W^{1,2}(S_{R})$. By a rescaling argument, we obtain
that for $r\ge R$,
\begin{align*}
\frac{R}{r}\bigl\Vert\bu;L^{2}(\partial B_{r})\bigr\Vert^{2} & \leq\frac{CR^{2}}{r^{2}}\bigl\Vert\bu;L^{2}(S_{r})\bigr\Vert^{2}+C\bigl\Vert\bnabla\bu;L^{2}(S_{r})\bigr\Vert^{2}\\
 & \leq4CR^{2}\bigl\Vert\bu/\bx;L^{2}(S_{r})\bigr\Vert^{2}+C\bigl\Vert\bnabla\bu;L^{2}(S_{r})\bigr\Vert^{2},
\end{align*}
for any $\bu\in W^{1,2}(S_{r})$. Now if $\bu$ is in addition centrally
symmetric, by applying \lemref{hardy-inequality-central-symmetry},
we obtain that there exists $C>0$ depending on $\Omega$ such that
\[
\frac{1}{r}\bigl\Vert\bu;L^{2}(\partial B_{r})\bigr\Vert\leq C\bigl\Vert\bnabla\bu;L^{2}(S_{r})\bigr\Vert,
\]
for all centrally symmetric $\bu\in W^{1,2}(S_{r})$. For $\bu\in D^{1,2}(\Omega)$,
we have
\[
\bigl\Vert\bnabla\bu;L^{2}(S_{r})\bigr\Vert^{2}=\bigl\Vert\bnabla\bu;L^{2}(B_{2r}\setminus B)\bigr\Vert^{2}-\bigl\Vert\bnabla\bu;L^{2}(B_{r}\setminus B)\bigr\Vert^{2},
\]
and since
\[
\lim_{r\to\infty}\bigl\Vert\bnabla\bu;L^{2}(B_{2r}\setminus B)\bigr\Vert=\bigl\Vert\bnabla\bu;L^{2}(\Omega)\bigr\Vert,
\]
we obtain
\[
\lim_{r\to\infty}\frac{1}{r}\bigl\Vert\bu;L^{2}(\partial B_{r})\bigr\Vert=0\,,
\]
which proves the claimed result.
\end{proof}
This result shows that a centrally symmetric weak solution goes to
zero at infinity. A stronger result showing the uniformly pointwise
limit was announced by \citet[Theorem~7]{Russo-existenceofDsolutions2011},
 but the correctness of the uniform limit is questionable, since it
implicitly relies on Lemma~3.10 of \citet{Galdi-StationaryNavier-Stokesproblem2004},
whose proof contains a gap.
\begin{thm}
Let the hypothesis of \thmref{existence-weak-solutions} be satisfied.
If $\Omega$, $\bu^{*}$ and $\bff$ are invariant under the central
symmetry \eqref{central-symmetry}, there exists a weak solution $\bu$
such that $\lim_{\left|\bx\right|\to0}\bu=\bzero$ in the following
sense
\[
\lim_{r\to\infty}\int_{0}^{2\pi}\left|\bu(r,\theta)\right|^{2}\rd\theta=0\,.
\]
\end{thm}
\begin{proof}
Since the Navier-Stokes equations are invariant under the central
symmetry \eqref{central-symmetry}, by applying \thmref{existence-weak-solutions},
we can construct of weak solution $\bu$ that is centrally symmetric.
Then the result follows by applying \lemref{limit-central-symmetry}.
\end{proof}

\section{Asymptotic behavior of the velocity}

The linearization\index{Weak solutions!asymptotic behavior}\index{Asymptotic behavior!Navier-Stokes solutions!weak solutions},
of the Navier-Stokes equations around $\bu=\bu_{\infty}$, leads to
the system
\begin{align}
\Delta\bu-\bnabla p-\bu_{\infty}\bcdot\bnabla\bu & =-\bff\,, & \bnabla\bcdot\bu & =0\,,\label{eq:weak-stokes-oseen}
\end{align}
which is the Stokes system for $\bu_{\infty}=\bzero$, and the Oseen
system in case $\bu_{\infty}\neq\bzero$. By bootstrapping the decay
of the velocity and of the nonlinearity, the Oseen system is well-posed
which furnishes the asymptotic behavior in case $\bu_{\infty}\neq\bzero$.
If $\bu_{\infty}=\bzero$, the situation is more complicated because
the Stokes system is ill-posed.

\subsection{In case \texorpdfstring{$\protect\bu_{\infty}\protect\neq\protect\bzero$}{u!=0}}

In three dimensions the following result was first obtained by \citet{Babenko-stationarysolutionsof1973}
by using results of \citet{Finn-exteriorstationaryproblem1965} and
later on by \citet[Theorem 4.1]{Galdi-AsymptoticStructureDsol-1992}.
In two dimensions, \citet[\S 4]{Smith-EstimatesatInfinity1965} showed
that if $\bu$ is a solution the Navier-Stokes equations such that
$\left|\bu-\bu_{\infty}\right|=O(\left|\bx\right|^{-1/4-\varepsilon})$
for some $\varepsilon>0$, the leading term of the asymptotic expansion
of $\bu-\bu_{\infty}$ is given by the Oseen fundamental solution.
This result was further clarified by \citet[Theorem XII.8.1]{Galdi-IntroductiontoMathematical2011}.
\begin{thm}[{\citealt[Theorems X.8.1 \& XII.8.1]{Galdi-IntroductiontoMathematical2011}}]
\label{thm:weak-asy-oseen}Let $\bu$ be a weak solution in an exterior
domain $\Omega\subset\mathbb{R}^{n}$, $n=2,3$ of class $C^{2}$.
If $\bu_{\infty}\neq\bzero$, $\bff\in L^{q}(\Omega)$ has compact
support, and $\bu^{*}\in W^{2-1/q_{0},1/q_{0}}(\partial\Omega)$ for
some $q_{0}>n$ and all $q\in(1,q_{0}]$. In case $n=2$, we assume
moreover that
\begin{equation}
\lim_{\left|\bx\right|\to\infty}\bu=\bu_{\infty}\,.\label{eq:weak-asymptotic-limit}
\end{equation}
Then
\[
\bu-\bu_{\infty}=\mathbf{E}\bcdot\boldsymbol{Z}+O(\left|\bx\right|^{-n/2+\varepsilon})\quad\text{for any}\quad\varepsilon>0
\]
where $\mathbf{E}$ is the Oseen tensor which satisfies as $\left|\bx\right|\to\infty$,
\[
\left|\mathbf{E}\right|=\begin{cases}
O\left(\dfrac{1}{\left|\bx\right|}+\dfrac{\e^{-s}}{\sqrt{\left|\bx\right|}}\right)\,, & n=2\,,\vspace{2mm}\\
O\left(\dfrac{1}{\left|\bx\right|}\dfrac{1-\e^{-s}}{s}\right)\,, & n=3\,,
\end{cases}\quad\text{with}\quad s=\frac{\left|\bu_{\infty}\right|\left|\bx\right|-\bu_{\infty}\bcdot\bx}{2}\,,
\]
and $\boldsymbol{Z}$ is a modification of the net force $\bF$ by
the flux $\Phi$,
\[
\boldsymbol{Z}=\bF+\bu_{\infty}\Phi\,,
\]
where
\[
\bF=\int_{\Omega}\bff+\int_{\partial B}\mathbf{T}\bn\quad\text{with}\quad\mathbf{T}=\bnabla\bu+\left(\bnabla\bu\right)^{T}-p\,\boldsymbol{1}-\bu\otimes\bu\,,
\]
and
\[
\Phi=\int_{\partial B}\bu\bcdot\bn\,.
\]
\end{thm}
\begin{rem}
In two dimensions, the validity of \eqref{weak-asymptotic-limit}
for a weak solution constructed by \thmref{existence-weak-solutions}
is still an open problem.
\end{rem}

\subsection{In case \texorpdfstring{$\protect\bu_{\infty}=\protect\bzero$}{u=0}}

If $\bu_{\infty}=\bzero$, the situation is more complicated and we
have to distinguish the two-dimensional and three-dimensional cases.
For $n=3$, the fundamental solution $\mathbf{U}$ of the Stokes system
\eqref{weak-stokes-oseen} decay like $\left|\bx\right|^{-1}$, which
by power counting implies that the nonlinearity $\bu\bcdot\bnabla\bu$
decays like $\left|\bx\right|^{-3}$. But as shown on \secref{failure-asymptotic}
for the two-dimensional case, the inversion of the Stokes operator
on a source term that decays like $\left|\bx\right|^{-3}$, leads
to a solution that decays like $\log\left|\bx\right|/\left|\bx\right|$.
Therefore, the Stokes system is ill-posed in this setting, and the
leading term at infinity cannot be the Stokes fundamental solution.
This fact was precisely formulated and proved by \citet[Theorem 3.1]{Deuring.Galdi-AsymptoticBehaviorof2000}.
Therefore, the term in $\left|\bx\right|^{-1}$ of the asymptotic
expansion has to be solution of a nonlinear equation. \citet[Theorem 3.2]{Nazarov-steady2000}
have shown that there exists a function $\boldsymbol{V}$ on the sphere
$S^{2}$ such that 
\[
\bu=\frac{\boldsymbol{V}\left(\hat{\bx}\right)}{\left|\bx\right|}+O(\left|\bx\right|^{-2+\varepsilon})\,,
\]
for all $\varepsilon>0$ provided the data are small enough. \citet{Sverak-LandausSolutionsNavier2011}
proved that the only nontrivial scale-invariant solution of the Navier-Stokes
equation in $\mathbb{R}^{3}\setminus\left\{ \bzero\right\} $ is the
\citet{Landau-newexactsolution1944} solution. The proof that the
leading asymptotic term is given by the Landau solution was simplified
by \citet{Korolev.Sverak-largedistanceasymptotics2011}. They proved
the following result:
\begin{thm}[{\citealp[Theorem 1]{Korolev.Sverak-largedistanceasymptotics2011}}]
\label{thm:weak-asy-landau}Let $\left(\bu,p\right)$ be a solution
of the Navier-Stokes equation in $\mathbb{R}^{3}\setminus\widebar B(\bzero,1)$.
For each $\varepsilon>0$, there exists $\nu>0$, such that if
\[
\left|\bu(\bx)\right|\leq\frac{\nu}{1+\left|\bx\right|}\,,
\]
then
\[
\bu=\bUF(\bx)+O(\left|\bx\right|^{-2+\varepsilon})\,,
\]
where $\bUF(\bx)$ is the Landau solution with net force $\bF$.\end{thm}
\begin{rem}
In particular, the asymptotic results of \thmref{weak-asy-oseen,weak-asy-landau}
show that in three dimensions,
\[
\sup_{\left|\bx\right|=r}\left|\bu-\bu_{\infty}\right|=O(r^{-1})\,,
\]
and therefore the limit \eqref{weak-limit} is uniformly pointwise.

In two dimensions, even if we take \eqref{weak-asymptotic-limit}
as an hypothesis, the asymptotic behavior of such a hypothetical solution
is not known. The aim of the following chapters is to determine the
asymptotic behavior of the solutions under compatibility conditions
or under symmetries (\chapref{strong-compatibility}), to study the
link between the asymptotic behavior of the Stokes and Navier-Stokes
equations equations (\chapref{link-Stokes-NS}), to perform a formal
asymptotic expansion in case the net force is non zero and to provide
some ideas of the possible asymptotic behavior that can emerge (\chapref{wake-sim}).
\end{rem}

\chapter{\label{chap:strong-compatibility}Strong solutions with compatibility
conditions}

We construct strong solution to the stationary and incompressible
Navier-Stokes equations in the plane, under compatibility conditions
on the source force. In particular these compatibility conditions
are fulfilled if the source force is invariant under four axes of
symmetry passing through the origin and separated by an angle of $\pi/4$.
Under this symmetry, the existence of a solution that is bounded by
$\left|\bx\right|^{-1}$ was shown by \citet{Yamazaki-Uniqueexistence2011}.
Here we improve this result by showing the existence of a solution
decaying like $\left|\bx\right|^{-3+\varepsilon}$ for all $\varepsilon>0$.
We also discuss how an explicit solution can be used to lift the compatibility
condition and actually lift the compatibility condition corresponding
to the net torque.

\section{Introduction}

The stationary Navier-Stokes equations in two-dimensional unbounded
domains are not mathematically understood in a proper way, especially
the existence of solutions such that the velocity converges to zero
at large distances is an open problem \citep[see][]{Galdi-IntroductiontoMathematical2011,Galdi-StationaryNavier-Stokesproblem2004}.
\citet{Leray-Etudedediverses1933} constructed weak solutions to the
Navier-Stokes equations in exterior domains in two and three dimensions,
with one major restriction: the domain cannot be $\mathbb{R}^{2}$
in his construction. Due to the properties of the function spaces
in two dimensions, \citet{Leray-Etudedediverses1933} was not able
to characterize the behavior at infinity of the weak solutions, \emph{i.e.}
more precisely the validity of
\begin{equation}
\lim_{\left|\bx\right|\to\infty}=\bu_{\infty}\,,\label{eq:compatibility-limit-uinfty}
\end{equation}
where $\bu_{\infty}\in\mathbb{R}^{2}$ is a prescribed vector. This
was remained open until \citet{Gilbarg.Weinberger-AsymptoticPropertiesof1974,Gilbarg.Weinberger-Asymptoticpropertiesof1978}
partially answer this question, by showing that either there exists
$\bu_{0}\in\mathbb{R}^{2}$ such that
\[
\lim_{\left|\bx\right|\to\infty}\int_{S^{1}}\left|\bu-\bu_{0}\right|^{2}=0\,,
\]
or either
\[
\lim_{\left|\bx\right|\to\infty}\int_{S^{1}}\left|\bu\right|^{2}=\infty\,.
\]
However, they cannot show that $\bu_{0}$ can be chosen arbitrarily,
that is to say that $\bu_{0}=\bu_{\infty}$ holds. Under some restriction,
this result was improved by \citet{Amick-Leraysproblemsteady1988}
who shows that $\bu$ is bounded. In case $\bu_{\infty}\neq\bzero$,
the linearization of the Navier-Stokes equations around $\bu=\bu_{\infty}$
is the Oseen equations and by a fixed point argument \citet{Finn-stationarysolutionsNavier1967}
showed the existence of solutions satisfying \eqref{compatibility-limit-uinfty}
provided the data are small enough. However, the existence of solutions
satisfying \eqref{compatibility-limit-uinfty} with $\bu_{\infty}=\bzero$
is still an open problem in its generality, even for small data. Moreover,
if the domain is the whole plane, even the existence of weak solutions
is unknown in general. The only results, which will be described in
details later on, are under suitable symmetries \citep{Galdi-StationaryNavier-Stokesproblem2004,Yamazaki-stationaryNavier-Stokesequation2009,Yamazaki-Uniqueexistence2011,Pileckas-existencevanishing2012}
or specific boundary conditions \citep{Hillairet-mu2013}.

We consider the incompressible Navier-Stokes equations in $\mathbb{R}^{2}$,
\begin{align}
\Delta\bu-\bnabla p & =\bu\bcdot\bnabla\bu+\bff\,, & \bnabla\bcdot\bu & =0\,, & \lim_{|\bx|\to\infty}\bu & =\bzero\,,\label{eq:compatibility-ns-force}
\end{align}
where $\bff$ is the source force. Under compatibility conditions
on the source term $\bff$ or suitable symmetries that fulfill these
compatibility conditions, we will show the existence of solutions
satisfying \eqref{compatibility-ns-force} and provide their asymptotic
expansions.

If an exterior domain and the data are symmetric with respect to two
orthogonal axes, then \citet[\S 3.3]{Galdi-StationaryNavier-Stokesproblem2004}
showed the existence of solutions satisfying the limit at infinity
in the following sense
\[
\lim_{|\bx|\to\infty}\int_{S^{1}}\left|\bu\right|^{2}=0\,.
\]
This result was improved by \citet[Theorem~7]{Russo-existenceofDsolutions2011}
by only requiring that the domain and the data are invariant under
the central symmetry $\bx\mapsto-\bx$, and by \citet{Pileckas-existencevanishing2012}
by allowing a flux through the boundary. However, all these results
rely only on the properties on the subset of symmetric functions of
the function space in which weak solutions are constructed, and therefore
the decay of the velocity at infinity is unknown. If the force force
$\bff$ is symmetric with respect to four axes with an angle of $\pi/4$
between them, \citet{Yamazaki-stationaryNavier-Stokesequation2009}
proved the existence of solutions in $\mathbb{R}^{2}$ such that the
velocity decays like $\left|\bx\right|^{-1}$ at infinity. Moreover,
\citet{Nakatsuka-uniquenessofsymmetric2013} proved the uniqueness
of the solution in this symmetry class. Later on, \citet{Yamazaki-Uniqueexistence2011}
showed the existence and uniqueness of the solutions in an exterior
domain always under the same four symmetries. In fact under these
symmetries, we will show that the solution decays like $\left|\bx\right|^{-3+\varepsilon}$
for all $\varepsilon>0$. In the exterior of a disk, \citet{Hillairet-mu2013}
proved the existence of solutions that also decay like $\left|\bx\right|^{-1}$
at infinity provided that the boundary condition on the disk is close
to $\mu\be_{r}$ for $\left|\mu\right|>\sqrt{48}$. To our knowledge,
these results are the only ones showing the existence of solutions
in two-dimensional unbounded domains with a known decay rate at infinity.

The linearization of Navier-Stokes equations \eqref{compatibility-ns-force}
around $\bu=\bzero$ is the Stokes system\index{Stokes equations}
\begin{align}
\Delta\bu-\bnabla p & =\bff\,, & \bnabla\bcdot\bu & =0\,.\label{eq:compatibility-stokes}
\end{align}
First of all, we will perform the general asymptotic expansion up
to any order of the solution of the Stokes system (\secref{compatibility-asymptotic})
and then explain the implications of some symmetries on this asymptotic
behavior (\secref{compatibility-symmetries}). By defining the net
force as
\[
\bF=\int_{\mathbb{R}^{2}}\bff\,,
\]
we will in particular recover the Stokes paradox: if $\bF\neq\bzero$
the Stokes equation \eqref{compatibility-stokes} has no solution
satisfying
\[
\lim_{|\bx|\to\infty}\bu=\bzero\,.
\]
Even in the case where $\bF=\bzero$, so that the solution of the
Stokes equation decay like $\left|\bx\right|^{-1}$, one can show
that the inversion of the Stokes operator on the nonlinearity $\bu\bcdot\bnabla\bu$
leads to an ill-defined problem (\secref{failure-asymptotic}). This
ill-possessedness of the Stokes system in a space of function decaying
like $\left|\bx\right|^{-1}$ is also present in three dimensions
\citep[Theorem 3.1]{Deuring.Galdi-AsymptoticBehaviorof2000}.

If one restricts oneself to the case where $\bF=\bzero$, then the
Stokes system has three compatibility conditions in order that its
solution decays better than $\left|\bx\right|^{-1}$ and only one
of them is an invariant quantity: the net torque (see \lemref{lin-stokes-asy-explicit}).
As shown by \thmref{compatibility-strong-mu}, the compatibility condition
corresponding to the net torque $M$ can be lifted by the exact solution
$M\be_{\theta}/r$. We remark that another way of lifting this compatibility
condition might be given by the small exact solutions found by \citet{Guillod-Generalizedscaleinvariant2015}.
The other two compatibility conditions are not invariant quantities
and therefore much more difficult to lift (see also \chapref{wake-sim}).

\section{Stokes fundamental solution}

The fundamental solution of the Stokes\index{Stokes equations!fundamental solution}\index{Fundamental solution!Stokes equations}
equation is given by
\begin{align*}
\mathbf{E} & =\frac{1}{4\pi}\left[\log\left|\bx\right|\,\mathbf{1}-\frac{\bx\otimes\bx}{\left|\bx\right|^{2}}\right]\,, & \be & =\frac{-1}{2\pi}\frac{\bx}{\left|\bx\right|^{2}}\,,
\end{align*}
so that the solution of the Stokes equation
\begin{align*}
\Delta\bu-\bnabla p & =\bff\,, & \bnabla\bcdot\bu & =0\,,
\end{align*}
in $\mathbb{R}^{2}$ is given by
\begin{align*}
\bu & =\mathbf{E}*\bff\,, & p & =\be*\bff\,.
\end{align*}
We can rewrite the Stokes tensor so that it becomes explicitly divergence
free,
\[
\mathbf{E}=\bnabla\bwedge\boldsymbol{\Psi}\,,
\]
where
\[
\boldsymbol{\Psi}=\frac{\bx^{\perp}}{4\pi}\left(\log\left|\bx\right|-1\right)\,.
\]
This notation is to be understood as the $i$th line of $\mathbf{E}$
is the curl (or rotated gradient) of the $i$th element of the vector
field $\boldsymbol{\Psi}$.

\section{\label{sec:compatibility-asymptotic}Asymptotic expansion of the
Stokes solutions}

We first define weighted $L^{\infty}$-spaces:
\begin{defn}[function spaces]
For $q\geq0$, we define the weight
\[
w_{q}(\bx)=\begin{cases}
1+\left|\bx\right|^{q}\,, & q>0\,,\\
\left[\log\left(2+\left|\bx\right|\right)\right]^{-1}\,, & q=0\,,
\end{cases}
\]
and the associated Banach space for $k\in\mathbb{N}$,\index{Space!B@$\mathcal{B}_{k,q}$}
\[
\mathcal{B}_{k,q}=\left\{ f\in C^{k}(\mathbb{R}^{n})\,:\: w_{q+\left|\alpha\right|}D^{\alpha}f\in L^{\infty}(\mathbb{R}^{n})\:\forall\left|\alpha\right|\leq k\right\} \,,
\]
with the norm
\[
\bigl\Vert f;\mathcal{B}_{k,q}\bigr\Vert=\max_{\left|\alpha\right|\leq k}\sup_{\bx\in\mathbb{R}^{n}}w_{q+\left|\sigma\right|}\left|D^{\alpha}f\right|\,.
\]

\end{defn}
The asymptotic expansion of a solution of the Stokes equation is given
by:
\begin{lem}
\label{lem:lin-stokes-asy}For $q>0$ and $q\notin\mathbb{N}$, if
$\bff\in\mathcal{B}_{0,2+q}$, then the solution of\index{Asymptotic behavior!Stokes solutions}\index{Stokes equations!asymptotic behavior}\index{Asymptotic expansion!Stokes equations}
\begin{align*}
\Delta\bu-\bnabla p & =\bff\,, & \bnabla\bcdot\bu & =0\,,
\end{align*}
satisfies
\begin{align*}
\bu & =\sum_{n=0}^{\left\lfloor q\right\rfloor }\bS_{n}+\bR\,, & p & =\sum_{n=0}^{\left\lfloor q\right\rfloor }s_{n}+r\,,
\end{align*}
where $\bS_{n}\in\mathcal{B}_{1,n}$, $\bR\in\mathcal{B}_{1,q}$,
$s_{n}\in\mathcal{B}_{0,n+1}$ and $r\in\mathcal{B}_{0,q+1}$. The
asymptotic terms are given by
\begin{align*}
\bS_{n} & =\bnabla\bwedge\left[\sum_{\left|\alpha\right|=n}\frac{\chi}{\alpha!}\left(\int_{\mathbb{R}^{2}}\left(-\bx\right)^{\alpha}\bff(\bx)\rd^{2}\bx\right)\bcdot\left(D^{\alpha}\boldsymbol{\Psi}\right)\right]\,,\\
s_{n} & =\sum_{\left|\alpha\right|=n}\frac{\chi}{\alpha!}\left(\int_{\mathbb{R}^{2}}\left(-\bx\right)^{\alpha}\bff(\bx)\rd^{2}\bx\right)\bcdot\left(D^{\alpha}\be\right)\,.
\end{align*}
\end{lem}
\begin{proof}
The solution is given by
\begin{align*}
\bu & =\mathbf{E}*\bff\,, & \bnabla\bu & =\bnabla\mathbf{E}*\bff\,, & p & =\be*\bff\,.
\end{align*}
The Stokes tensor $\mathbf{E}$ diverges like $\log\left|\bx\right|$
at the origin and $\bnabla\mathbf{E}$ as well as $\be$ like $\left|\bx\right|^{-1}$,
but the integrals defining $\bu$, $\bnabla\bu$ and $p$ converge
and are continuous \citep[Proposition 8.8]{Folland1999}, so $\bu\in C^{1}(\mathbb{R}^{2})$.
Therefore, it remains only to prove the decay of $\bR$, $\bnabla\bR$
and $r$. By definition, we have the estimate
\[
\left|\bR\right|\leq\int_{\mathbb{R}^{2}}\biggl|\mathbf{E}(\bx-\by)-\bnabla\bwedge\biggl[\chi(\left|\bx\right|)\sum_{\left|\alpha\right|\leq\left\lfloor q\right\rfloor }\frac{\left(-\by\right)^{\alpha}}{\alpha!}D^{\alpha}\boldsymbol{\Psi}(\bx)\biggr]\biggr|\left|\bff(\by)\right|\rd^{2}\by\,.
\]
We first define the cut-off of the Stokes tensor,
\[
\mathbf{E}_{\chi}(\bx)=\chi(\left|\bx\right|)\mathbf{E}(\bx)\,,
\]
and split the bound in three parts,
\[
\left|\bR\right|\lesssim I+J+K\,,
\]
where
\begin{align*}
I & =\int_{\mathbb{R}^{2}}\left|\mathbf{E}(\bx-\by)-\mathbf{E}_{\chi}(\bx-\by)\right|\frac{1}{1+\left|\by\right|^{q+2}}\rd^{2}\by\,,\\
J & =\int_{\mathbb{R}^{2}}\biggl|\mathbf{E}_{\chi}(\bx-\by)-\sum_{\left|\alpha\right|\leq\left\lfloor q\right\rfloor }\frac{\left(-\by\right)^{\alpha}}{\alpha!}D^{\alpha}\mathbf{E}_{\chi}(\bx)\biggr|\frac{1}{1+\left|\by\right|^{q+2}}\rd^{2}\by\,,\\
K & =\int_{\mathbb{R}^{2}}\sum_{\left|\alpha\right|\leq\left\lfloor q\right\rfloor }\left|\frac{\left(-\by\right)^{\alpha}}{\alpha!}\left[D^{\alpha}\mathbf{E}_{\chi}(\bx)-\bnabla\bwedge\left(\chi(\left|\bx\right|)D^{\alpha}\boldsymbol{\Psi}(\bx)\right)\right]\right|\frac{1}{1+\left|\by\right|^{q+2}}\rd^{2}\by\,.
\end{align*}
The first integral is easy to estimate, since it has support only
in the region where $\left|\bx-\by\right|\leq2$,
\[
I\lesssim\int_{\mathbb{R}^{2}}\left(1-\chi\left(\left|\bx-\by\right|\right)\right)\left|\mathbf{E}(\bx-\by)\right|\frac{1}{1+\left|\by\right|^{q+2}}\rd^{2}\by\lesssim\frac{1}{1+\left|\bx\right|^{q+2}}\,.
\]
For the third integral, we have 
\[
K\lesssim\bigl|D^{\alpha}\mathbf{E}_{\chi}(\bx)-\bnabla\bwedge\left(\chi(\left|\bx\right|)D^{\alpha}\boldsymbol{\Psi}(\bx)\right)\bigr|\int_{\mathbb{R}^{2}}\frac{1}{1+\left|\by\right|^{q-\left\lfloor q\right\rfloor +2}}\rd^{2}\by\lesssim\left(1-\chi(\left|\bx\right|)\right)\,,
\]
since the integral vanishes for $\left|\bx\right|\ge2$. We now estimate
the second integral which requires more calculations. Since $\mathbf{E}_{\chi}$
is a smooth function on $\mathbb{R}^{2}$, by using Taylor theorem,
we have
\[
\mathbf{E}_{\chi}(\bx-\by)=\sum_{\left|\alpha\right|\leq k}\frac{\left(-\by\right)^{\alpha}}{\alpha!}D^{\alpha}\mathbf{E}_{\chi}(\bx)+\mathbf{H}_{k}(\bx,\by)\,,
\]
where
\[
\mathbf{H}_{k}(\bx,\by)=\left(k+1\right)\sum_{\left|\alpha\right|=k+1}\frac{\left(-\by\right)^{\alpha}}{\alpha!}\int_{0}^{1}\left(1-\lambda\right)^{k}D^{\alpha}\mathbf{E}_{\chi}(\bx-\lambda\by)\,\rd\lambda\,.
\]
Since $D^{\alpha}\mathbf{E}_{\chi}\in\mathcal{B}_{0,\left|\alpha\right|}$,
we have
\[
\left|\mathbf{H}_{k}(\bx,\by)\right|\lesssim\left|\by\right|^{k+1}\int_{0}^{1}\frac{\left(1-\lambda\right)^{k}}{1+\left|\bx-\lambda\by\right|^{k+1}}\rd\lambda\,.
\]
In order to estimate $J$, we divide the integration into two parts
$J=J_{1}+J_{2}$, with
\begin{align*}
D_{1} & =\left\{ \by\in\mathbb{R}^{2}\,:\:\left|\by\right|\leq\left|\bx\right|/2\right\} \,, & D_{2} & =\left\{ \by\in\mathbb{R}^{2}\,:\:\left|\by\right|\geq\left|\bx\right|/2\right\} \,.
\end{align*}
If $\by\in D_{1}$, we have
\[
\bigl|\mathbf{H}_{\left\lfloor q\right\rfloor }(\bx,\by)\bigr|\lesssim\frac{\left|\by\right|^{\left\lfloor q\right\rfloor +1}}{1+\left|\bx\right|^{\left\lfloor q\right\rfloor +1}}\,,
\]
and therefore
\[
\left|J_{1}\right|=\int_{D_{1}}\bigl|\mathbf{H}_{\left\lfloor q\right\rfloor }(\bx,\by)\bigr|\frac{1}{1+\left|\by\right|^{q+2}}\rd^{2}\by\lesssim\frac{1}{1+\left|\bx\right|^{\left\lfloor q\right\rfloor +1}}\int_{D_{1}}\frac{1}{\left|\by\right|^{q-\left\lfloor q\right\rfloor +1}}\rd^{2}\by\lesssim\frac{1}{1+\left|\bx\right|^{q}}\,.
\]
If $\by\in D_{2}$, we use that
\[
\bigl|\mathbf{H}_{\left\lfloor q\right\rfloor }(\bx,\by)\bigr|\lesssim\bigl|\mathbf{H}_{0}(\bx,\by)\bigr|+\sum_{k=1}^{\left\lfloor q\right\rfloor }\frac{\left|\by\right|^{k}}{1+\left|\bx\right|^{k}}\lesssim\int_{0}^{1}\frac{\left|\by\right|}{1+\left|\bx-\lambda\by\right|}\rd\lambda+\sum_{k=1}^{\left\lfloor q\right\rfloor }\frac{\left|\by\right|^{k}}{1+\left|\bx\right|^{k}}\,,
\]
so
\begin{align*}
\left|J_{2}\right| & =\int_{D_{2}}\bigl|\mathbf{H}_{\left\lfloor q\right\rfloor }(\bx,\by)\bigr|\frac{1}{1+\left|\by\right|^{q+2}}\rd^{2}\by\\
 & \lesssim\int_{D_{2}}\int_{0}^{1}\frac{1}{1+\left|\bx-\lambda\by\right|}\frac{1}{1+\left|\by\right|^{q+1}}\rd\lambda\,\rd^{2}\by+\sum_{k=1}^{\left\lfloor q\right\rfloor }\frac{1}{1+\left|\bx\right|^{k}}\int_{D_{2}}\frac{1}{1+\left|\by\right|^{q-k+2}}\rd^{2}\by\\
 & \lesssim\frac{1}{1+\left|\bx\right|^{\left\lfloor q\right\rfloor }}\left(\int_{D_{2}}\int_{0}^{1}\frac{1}{1+\left|\bx-\lambda\by\right|}\frac{1}{1+\left|\by\right|^{q-\left\lfloor q\right\rfloor +1}}\rd\lambda\,\rd^{2}\by+\int_{D_{2}}\frac{1}{1+\left|\by\right|^{q-\left\lfloor q\right\rfloor +2}}\rd^{2}\by\right)\\
 & \lesssim\frac{1}{1+\left|\bx\right|^{\left\lfloor q\right\rfloor }}\left(\int_{0}^{1}\lambda^{q-\left\lfloor q\right\rfloor -1}\rd\lambda\int_{\mathbb{R}^{2}}\frac{1}{1+\left|\bx-\boldsymbol{z}\right|}\frac{1}{\left|\boldsymbol{z}\right|^{q-\left\lfloor q\right\rfloor +1}}\rd^{2}\boldsymbol{z}+\frac{1}{1+\left|\bx\right|^{q-\left\lfloor q\right\rfloor }}\right)\\
 & \lesssim\frac{1}{1+\left|\bx\right|^{q}}\,.
\end{align*}
Consequently, we have proven that $\bR\in\mathcal{B}_{0,q}$.

We now estimate the pressure remainder, also by splitting the bound
into three parts,
\[
\left|r\right|\leq\int_{\mathbb{R}^{2}}\biggl|\be(\bx-\by)-\sum_{\left|\alpha\right|\leq\left\lfloor q\right\rfloor }\frac{\left(-\by\right)^{\alpha}}{\alpha!}\chi(\left|\bx\right|)D^{\alpha}\be(\bx)\biggr|\left|\bff(\by)\right|\rd^{2}\by\lesssim I+J+K\,,
\]
where
\begin{align*}
I & =\int_{\mathbb{R}^{2}}\biggl|\be(\bx-\by)-\be_{\chi}(\bx-\by)\biggr|\frac{1}{1+\left|\by\right|^{q+2}}\rd^{2}\by\,,\\
J & =\int_{\mathbb{R}^{2}}\biggl|\be_{\chi}(\bx-\by)-\sum_{\left|\alpha\right|\leq\left\lfloor q\right\rfloor }\frac{\left(-\by\right)^{\alpha}}{\alpha!}D^{\alpha}\be_{\chi}(\bx)\biggr|\frac{1}{1+\left|\by\right|^{q+2}}\rd^{2}\by\,,\\
K & =\int_{\mathbb{R}^{2}}\sum_{\left|\alpha\right|\leq\left\lfloor q\right\rfloor }\biggl|\frac{\left(-\by\right)^{\alpha}}{\alpha!}\left[D^{\alpha}\be_{\chi}(\bx)-\chi(\left|\bx\right|)D^{\alpha}\be(\bx)\right]\biggr|\frac{1}{1+\left|\by\right|^{q+2}}\rd^{2}\by\,.,
\end{align*}
and where we also consider the cut-off of the fundamental solution
for the pressure,
\[
\be_{\chi}(\bx)=\chi(\left|\bx\right|)\be(\bx)\,.
\]
The first integral is easy to estimate, since it has support only
in the region where $\left|\bx-\by\right|\leq2$,
\[
I\lesssim\int_{\mathbb{R}^{2}}\left(1-\chi\left(\left|\bx-\by\right|\right)\right)\left|\be(\bx-\by)\right|\frac{1}{1+\left|\by\right|^{q+2}}\rd^{2}\by\lesssim\frac{1}{1+\left|\bx\right|^{q+2}}\,.
\]
The third integral converges and has compact support. For the second
integral, by using Taylor theorem, we have
\[
\be_{\chi}(\bx-\by)=\sum_{\left|\alpha\right|\leq\left\lfloor q\right\rfloor }\frac{\left(-\by\right)^{\alpha}}{\alpha!}D^{\alpha}\be_{\chi}(\bx)+\boldsymbol{h}(\bx,\by)\,,
\]
where
\[
\boldsymbol{h}(\bx,\by)=\left(\left\lfloor q\right\rfloor +1\right)\sum_{\left|\alpha\right|=\left\lfloor q\right\rfloor +1}\frac{\left(-\by\right)^{\alpha}}{\alpha!}\int_{0}^{1}\left(1-\lambda\right)^{\left\lfloor q\right\rfloor }D^{\alpha}\be_{\chi}(\bx-\lambda\by)\,\rd\lambda\,.
\]
Since $D^{\alpha}\be_{\chi}\in\mathcal{B}_{0,\left|\alpha\right|+1}$,
we have
\[
\left|\boldsymbol{h}(\bx,\by)\right|\lesssim\left|\by\right|^{\left\lfloor q\right\rfloor +1}\int_{0}^{1}\frac{\left(1-\lambda\right)^{\left\lfloor q\right\rfloor }}{1+\left|\bx-\lambda\by\right|^{\left\lfloor q\right\rfloor +2}}\rd\lambda\,.
\]
In order to estimate $J$, we divide the integration into two parts
$J=J_{1}+J_{2}$, with
\begin{align*}
D_{1} & =\left\{ \by\in\mathbb{R}^{2}\,:\:\left|\by\right|\leq\left|\bx\right|/2\right\} \,, & D_{2} & =\left\{ \by\in\mathbb{R}^{2}\,:\:\left|\by\right|\geq\left|\bx\right|/2\right\} \,.
\end{align*}
If $\by\in D_{1}$, we have
\[
\left|\boldsymbol{h}(\bx,\by)\right|\lesssim\frac{\left|\by\right|^{\left\lfloor q\right\rfloor +1}}{1+\left|\bx\right|^{\left\lfloor q\right\rfloor +2}}\,,
\]
and therefore
\[
\left|J_{1}\right|=\int_{D_{1}}\left|\boldsymbol{h}(\bx,\by)\right|\frac{1}{1+\left|\by\right|^{q+2}}\rd^{2}\by\lesssim\frac{1}{1+\left|\bx\right|^{\left\lfloor q\right\rfloor +2}}\int_{D_{1}}\frac{1}{\left|\by\right|^{q-\left\lfloor q\right\rfloor +1}}\rd^{2}\by\lesssim\frac{1}{1+\left|\bx\right|^{q+1}}\,.
\]
If $\by\in D_{2}$, we bound each term separately,
\[
\left|\boldsymbol{h}(\bx,\by)\right|\lesssim\frac{1}{1+\left|\bx-\by\right|}+\frac{1}{1+\left|\bx\right|}\sum_{k=0}^{\left\lfloor q\right\rfloor }\frac{\left|\by\right|^{k}}{1+\left|\bx\right|^{k}}\,,
\]
so we have
\begin{align*}
\left|J_{2}\right| & =\int_{D_{2}}\left|\boldsymbol{h}(\bx,\by)\right|\frac{1}{1+\left|\by\right|^{q+2}}\rd^{2}\by\\
 & \lesssim\int_{D_{2}}\frac{1}{1+\left|\bx-\by\right|}\frac{1}{1+\left|\by\right|^{q+2}}\rd^{2}\by+\sum_{k=0}^{\left\lfloor q\right\rfloor }\frac{1}{1+\left|\bx\right|^{k+1}}\int_{D_{2}}\frac{1}{1+\left|\by\right|^{q-k+2}}\rd^{2}\by\\
 & \lesssim\frac{1}{1+\left|\bx\right|^{\left\lfloor q\right\rfloor +1}}\left(\int_{D_{2}}\frac{1}{1+\left|\bx-\by\right|}\frac{1}{1+\left|\by\right|^{q-\left\lfloor q\right\rfloor +1}}\rd^{2}\by+\int_{D_{2}}\frac{1}{1+\left|\by\right|^{q-\left\lfloor q\right\rfloor +2}}\rd^{2}\by\right)\\
 & \lesssim\frac{1}{1+\left|\bx\right|^{q+1}}\,.
\end{align*}
Consequently, we have proven that $r\in\mathcal{B}_{0,q+1}$. The
proof of $\bnabla\bR\in\mathcal{B}_{0,q+1}$ works the same way as
the previous bounds, so only the main differences are pointed out
below. For $\sigma$ a multi-index with $\left|\sigma\right|=1$,
we split the integrals $\left|D^{\sigma}\bR\right|$ into three parts,
$I^{\sigma}$, $J^{\sigma}$ and $K^{\sigma}$. The parts $I^{\sigma}$
and $K^{\sigma}$ are as before. The second part is given by 
\[
\left|J^{\sigma}\right|\leq\int_{\mathbb{R}^{2}}\left|\mathbf{H}^{\sigma}(\bx,\by)\right|\left|\bff(\by)\right|\rd^{2}\by\,,
\]
where $\mathbf{H}^{\sigma}$ is defined through 
\[
D^{\sigma}\mathbf{E}_{\chi}(\bx-\by)=\sum_{\left|\alpha\right|\leq\left\lfloor q\right\rfloor }\frac{\left(-\by\right)^{\alpha}}{\alpha!}D^{\alpha+\sigma}\mathbf{E}_{\chi}(\bx)+\mathbf{H}^{\sigma}(\bx,\by)\,.
\]
In the region $D_{1}$, we use the Taylor theorem to obtain
\[
\left|\mathbf{H}^{\sigma}(\bx,\by)\right|\lesssim\left|\by\right|^{\left\lfloor q\right\rfloor +1}\int_{0}^{1}\frac{\left(1-\lambda\right)^{\left\lfloor q\right\rfloor }}{1+\left|\bx-\lambda\by\right|^{\left\lfloor q\right\rfloor +2}}\rd\lambda\lesssim\frac{\left|\by\right|^{\left\lfloor q\right\rfloor +1}}{1+\left|\bx\right|^{\left\lfloor q\right\rfloor +2}}\,,
\]
and in the region $D_{2}$, we bound each terms separately,
\[
\left|\mathbf{H}^{\sigma}(\bx,\by)\right|\lesssim\frac{1}{1+\left|\bx-\by\right|}+\frac{1}{1+\left|\bx\right|}\sum_{k=0}^{\left\lfloor q\right\rfloor }\left(\frac{\left|\by\right|}{1+\left|\bx\right|}\right)^{k}\,.
\]

\end{proof}
The first three orders of the asymptotic expansion are computed explicitly
in the following lemma:
\begin{lem}
\label{lem:lin-stokes-asy-explicit}Each term of the asymptotic expansion
of the Stokes system can be written as
\begin{align*}
\bS_{i} & =\bC_{i}[\bff]\bcdot\mathbf{E}_{i}\,, & s_{i} & =\bC_{i}[\bff]\bcdot\be_{i}\,,
\end{align*}
where
\[
\mathbf{E}_{i}=\bnabla\bwedge\left(\chi\boldsymbol{\Psi}_{i}\right)\,,
\]
and $\bC_{i}[\bff]$ is a constant\index{Compatibility conditions!Stokes equations}\index{Asymptotic expansion!Stokes equations}
vector whose length depends on $i$. The zeroth order is given by
\begin{align*}
\boldsymbol{\Psi}_{0} & =\frac{r}{4\pi}\left(\log r-1\right)\bigl(-\sin\theta,\:\cos\theta\bigr)\,,\\
\be_{0} & =\frac{-\chi}{2\pi r}\be_{r}\,,\\
\boldsymbol{C}_{0}[\bff] & =\int_{\mathbb{R}^{2}}\bff\,.
\end{align*}
The first order is given by
\begin{align*}
\boldsymbol{\Psi}_{1} & =\frac{1}{8\pi}\bigl(\sin(2\theta),\:-\cos(2\theta),\:1-2\log r\bigr)\,,\\
\be_{1} & =\frac{-\chi}{2\pi r^{2}}\bigl(\cos(2\theta),\:\sin(2\theta),\:0\bigr)\,,\\
\boldsymbol{C}_{1}[\bff] & =\int_{\mathbb{R}^{2}}\bigl(x_{1}f_{1}-x_{2}f_{2},\: x_{1}f_{2}+x_{2}f_{1},\: x_{1}f_{2}-x_{2}f_{1}\bigr)\,,
\end{align*}
and explicitly for $\left|\bx\right|\geq2$,
\[
\mathbf{E}_{1}=\frac{-1}{4\pi r}\bigl(\cos(2\theta)\be_{r},\:\sin(2\theta)\be_{r},\:\be_{\theta}\bigr)\,.
\]
Finally, for the second order, we have
\begin{align*}
\boldsymbol{\Psi}_{2} & =\frac{1}{8\pi r}\bigl(\sin(3\theta),\:\cos(3\theta),\:\sin(\theta),\:\cos(\theta)\bigr)\,,\\
\be_{2} & =\frac{-\chi}{2\pi r^{2}}\bigl(\cos(3\theta),\:\sin(3\theta),0,0\bigr)\,,\\
\boldsymbol{C}_{2}[\bff] & =\int_{\mathbb{R}^{2}}\Big((x_{1}^{2}-x_{2}^{2})f_{1}-2x_{1}x_{2}f_{2},\:(x_{2}^{2}-x_{1}^{2})f_{2}-2x_{1}x_{2}f_{1},\:\\
 & \phantom{=\int_{\mathbb{R}^{2}}\Big(\Big(}2x_{1}x_{2}f_{2}-3x_{2}^{2}f_{1}-x_{1}^{2}f_{1},\:3x_{1}^{2}f_{2}+x_{2}^{2}f_{2}-2x_{1}x_{2}f_{1}\Big)\,,
\end{align*}
and for $\left|\bx\right|\geq2$ we have explicitly
\begin{align*}
\boldsymbol{S}_{2} & =\frac{A_{1}}{4\pi r^{2}}\bigl(\cos(2\theta),\:\sin(2\theta)\bigr)+\frac{A_{2}}{4\pi r^{2}}\bigl(\sin(2\theta),\:-\cos(2\theta)\bigr)\\
 & \phantom{==}+\frac{A_{3}}{4\pi r^{2}}\bigl(\cos(2\theta)+\cos(4\theta),\:\sin(4\theta)\bigr)+\frac{A_{4}}{4\pi r^{2}}\bigl(\sin(2\theta)+\sin(4\theta),\:-\cos(4\theta)\bigr)\,,
\end{align*}
where $\boldsymbol{A}\in\mathbb{R}^{4}$ is related to the second
moments $\boldsymbol{C}_{2}[\bff]$.\end{lem}
\begin{proof}
The zeroth order follows directly by applying \lemref{lin-stokes-asy}.
By definition, the first order is
\begin{align*}
\bS_{1} & =-\bnabla\bwedge\biggl[\chi\left(\int_{\mathbb{R}^{2}}x_{1}\bff(\bx)\,\rd^{2}\bx\right)\bcdot\left(\partial_{1}\boldsymbol{\Psi}\right)+\chi\left(\int_{\mathbb{R}^{2}}x_{2}\bff(\bx)\,\rd^{2}\bx\right)\bcdot\left(\partial_{2}\boldsymbol{\Psi}\right)\biggr]\,,\\
 & =\bnabla\bwedge\biggl[\frac{\chi}{8\pi}\biggl[\sin(2\theta)\left(\int_{\mathbb{R}^{2}}x_{1}f_{1}-x_{2}f_{2}\right)-\cos(2\theta)\left(\int_{\mathbb{R}^{2}}x_{1}f_{2}+x_{2}f_{1}\right)\\
 & \phantom{=\bnabla\bwedge\biggl[\frac{\chi}{8\pi}\biggl[}+\left(1-2\log r\right)\left(\int_{\mathbb{R}^{2}}x_{1}f_{2}-x_{2}f_{1}\right)\biggr]\biggr]\\
 & =\bnabla\bwedge\left[\chi\bC_{1}[\bff]\bcdot\boldsymbol{\Psi}_{1}\right]=\bC_{1}[\bff]\bcdot\mathbf{E}_{1}\,,
\end{align*}
where $\bC_{1}[\bff]$, $\boldsymbol{\Psi}_{1}$ and $\mathbf{E}_{1}$
are defined in the wording of the Lemma. In the same way, we obtain
the pressure,
\begin{align*}
s_{1} & =-\chi\left(\int_{\mathbb{R}^{2}}x_{1}\bff\right)\bcdot\left(\partial_{1}\be\right)-\chi\left(\int_{\mathbb{R}^{2}}x_{1}\bff\right)\bcdot\left(\partial_{2}\be\right)\\
 & =\frac{-\chi}{2\pi r^{2}}\biggl[\cos(2\theta)\left(\int_{\mathbb{R}^{2}}x_{1}f_{1}-x_{2}f_{2}\right)+\sin(2\theta)\left(\int_{\mathbb{R}^{2}}x_{1}f_{2}+x_{2}f_{1}\right)\biggr]\\
 & =\bC_{1}[\bff]\bcdot\be_{1}\,.
\end{align*}
By explicitly taking the curl, for $\left|\bx\right|\geq2$, we have
\[
\mathbf{E}_{1}=\frac{-1}{4\pi r}\bigl(\cos(2\theta)\be_{r},\sin(2\theta)\be_{r},\be_{\theta}\bigr)\,.
\]
For the second order, we have

\begin{align*}
\bS_{2} & =\bnabla\bwedge\biggl[\frac{\chi}{8\pi r}\biggl[\sin(3\theta)\left(\int_{\mathbb{R}^{2}}(x_{1}^{2}-x_{2}^{2})f_{1}-2x_{1}x_{2}f_{2}\right)+\cos(3\theta)\left(\int_{\mathbb{R}^{2}}(x_{2}^{2}-x_{1}^{2})f_{2}-2x_{1}x_{2}f_{1}\right)\\
 & \phantom{=\bnabla\bwedge\biggl[}+\sin(\theta)\left(\int_{\mathbb{R}^{2}}2x_{1}x_{2}f_{2}-3x_{2}^{2}f_{1}-x_{1}^{2}f_{1}\right)+\cos(\theta)\left(\int_{\mathbb{R}^{2}}3x_{1}^{2}f_{2}+x_{2}^{2}f_{2}-2x_{1}x_{2}f_{1}\right)\biggr]\\
 & =\bnabla\bwedge\left[\chi\bC_{2}[\bff]\bcdot\boldsymbol{\Psi}_{2}\right]=\bC_{2}[\bff]\bcdot\mathbf{E}_{2}\,,
\end{align*}
and
\begin{align*}
s_{2} & =\frac{-\chi}{\pi r^{3}}\biggl[\cos(3\theta)\left(\int_{\mathbb{R}^{2}}(x_{1}^{2}-x_{2}^{2})f_{1}-2x_{1}x_{2}f_{2}\right)+\sin(3\theta)\left(\int_{\mathbb{R}^{2}}(x_{1}^{2}-x_{2}^{2})f_{2}+2x_{1}x_{2}f_{1}\right)\biggr]\\
 & =\bC_{2}[\bff]\bcdot\be_{2}\,.
\end{align*}
Moreover, for $\left|\bx\right|\geq2$ we have explicitly
\begin{align*}
\boldsymbol{S}_{2} & =\frac{A_{1}}{4\pi r^{2}}\bigl(\cos(2\theta),\:\sin(2\theta)\bigr)+\frac{A_{2}}{4\pi r^{2}}\bigl(\sin(2\theta),\:-\cos(2\theta)\bigr)\\
 & \phantom{==}+\frac{A_{3}}{4\pi r^{2}}\bigl(\cos(2\theta)+\cos(4\theta),\:\sin(4\theta)\bigr)+\frac{A_{4}}{4\pi r^{2}}\bigl(\sin(2\theta)+\sin(4\theta),\:-\cos(4\theta)\bigr)\,,
\end{align*}
where $\boldsymbol{A}\in\mathbb{R}^{4}$ is given in terms of the
second momenta by
\begin{align*}
A_{1} & =\frac{C_{21}-C_{23}}{2}\,, & A_{2} & =\frac{C_{24}-C_{22}}{2}\,, & A_{3} & =-C_{11}\,, & A_{4} & =C_{22}\,.
\end{align*}
\end{proof}
\begin{rem}
\label{rem:compatibility-stokes-paradox}The zeroth order $\bS_{0}$
grows like $\log r$ at infinity, which is the well-known Stokes paradox.
\end{rem}

\section{\label{sec:compatibility-symmetries}Symmetries and compatibility
conditions}

We consider the discrete symmetries\index{Symmetry!compatibility conditions}
represented on \figref{symmetries} and in particular their implications
on the asymptotic terms of the solution of the Stokes system:

{\renewcommand\theenumi{(\alph{enumi})} \renewcommand\labelenumi{\theenumi} 
\begin{enumerate}
\item The central symmetry\index{Symmetry!central}\index{Central symmetry},
\begin{equation}
\bff(\bx)=-\bff(-\bx)\label{eq:symmetry-central-second}
\end{equation}
cancels the zeroth order of the asymptotic expansion, because
\[
\boldsymbol{C}_{0}[\bff]=\int_{\mathbb{R}^{2}}\bff=\bzero\,.
\]

\item The symmetry with respect to the $x_{2}$-axis\index{Symmetry!axial}\index{Axial symmetry},
\begin{equation}
\begin{aligned}f_{1}(x_{1},x_{2}) & =-f_{1}(-x_{1},x_{2})\,,\\
f_{2}(x_{1},x_{2}) & =f_{2}(-x_{1},x_{2})\,,
\end{aligned}
\label{eq:symmetry-x1}
\end{equation}
implies that
\begin{align*}
\int_{\mathbb{R}^{2}}\left(x_{1}f_{2}+x_{2}f_{1}\right) & =0\,, & \int_{\mathbb{R}^{2}}\left(x_{1}f_{2}-x_{2}f_{1}\right) & =0\,,
\end{align*}
so that the last two components of $\boldsymbol{C}_{1}[\bff]$ vanish.
\item By considering the symmetry with respect to the $x_{1}$-axis rotated
by $\frac{\pi}{2}$,
\begin{equation}
\begin{aligned}f_{1}(x_{1},x_{2}) & =f_{2}(x_{2},x_{1})\,,\\
f_{2}(x_{1},x_{2}) & =f_{1}(x_{2},x_{1})\,,
\end{aligned}
\label{eq:symmetry-x1-rotated}
\end{equation}
we have
\begin{align*}
\int_{\mathbb{R}^{2}}\left(x_{1}f_{1}-x_{2}f_{2}\right) & =0\,, & \int_{\mathbb{R}^{2}}\left(x_{1}f_{2}-x_{2}f_{1}\right) & =0\,,
\end{align*}
so that only the second component of $\boldsymbol{C}_{1}[\bff]$ is
non zero.
\item By combining the central symmetry \eqref{symmetry-central-second},
and the symmetry with respect to the $x_{2}$-axis \eqref{symmetry-x1},
we obtain two axes of symmetry coinciding with the coordinate axes,
\begin{equation}
\begin{aligned}f_{1}(x_{1},x_{2}) & =f_{1}(x_{1},-x_{2})=-f_{1}(-x_{1},x_{2})\,,\\
f_{2}(x_{1},x_{2}) & =-f_{2}(x_{1},-x_{2})=f_{2}(-x_{1},x_{2})\,.
\end{aligned}
\label{eq:symmetry-galdi}
\end{equation}

\item By combining the symmetry with respect to the rotated $x_{1}$-axis
\eqref{symmetry-x1-rotated} together with the central symmetry \eqref{symmetry-central-second},
we obtain,
\begin{equation}
\begin{aligned}f_{1}(x_{1},x_{2}) & =f_{2}(x_{2},x_{1})=-f_{2}(-x_{2},-x_{1})\,,\\
f_{2}(x_{1},x_{2}) & =f_{1}(x_{2},x_{1})=-f_{1}(-x_{2},-x_{1})\,.
\end{aligned}
\label{eq:symmetry-yamazaki}
\end{equation}

\item Finally, by combining the symmetries \eqref{symmetry-galdi} and \eqref{symmetry-yamazaki},
which is equivalent to combining \eqref{symmetry-x1} and \eqref{symmetry-x1-rotated},
we obtain that the first two asymptotic terms vanish,
\begin{align*}
\boldsymbol{C}_{0}[\bff] & =\bzero\,, & \boldsymbol{C}_{1}[\bff] & =\bzero\,.
\end{align*}

\end{enumerate}
}

For the second order, the situation somewhat astonishing: the central
symmetry directly imply that $\boldsymbol{C}_{2}[\bff]=\bzero$ because
$\boldsymbol{C}_{2}[\bff]$ consists of moments of order two. We summarize
the implications of the symmetries on the asymptotic terms in the
following table:

\begin{center}
\begin{tabular}{llccccccccccc}
\toprule[2pt]Symmetries & \multicolumn{4}{c}{$\boldsymbol{C}_{0}[\bff]$} & \multicolumn{3}{c}{$\boldsymbol{C}_{1}[\bff]$} &  & \multicolumn{4}{c}{$\boldsymbol{C}_{2}[\bff]$}\tabularnewline
\cmidrule(l){1-1}\cmidrule(lr){3-4}\cmidrule(lr){6-8}\cmidrule(lr){10-13}$\text{(a)}$ &  & $0$ & $0$ &  & $*$ & $*$ & $*$ &  & $0$ & $0$ & $0$ & $0$\tabularnewline
$\text{(b)}$ &  & $*$ & $*$ &  & $*$ & $0$ & $0$ &  & $0$ & $*$ & $0$ & $*$\tabularnewline
$\text{(c)}$ &  & $*$ & $*$ &  & $0$ & $*$ & $0$ &  & $*$ & $*$ & $*$ & $*$\tabularnewline
$\text{(d)}=\text{(a)}+\text{(b)}$ &  & $0$ & $0$ &  & $*$ & $0$ & $0$ &  & $0$ & $0$ & $0$ & $0$\tabularnewline
$\text{(e)}=\text{(a)}+\text{(c)}$ &  & $0$ & $0$ &  & $0$ & $*$ & $0$ &  & $0$ & $0$ & $0$ & $0$\tabularnewline
$\text{(f)}=\text{(d)}+\text{(e)}=\text{(b)}+\text{(c)}$ &  & $0$ & $0$ &  & $0$ & $0$ & $0$ &  & $0$ & $0$ & $0$ & $0$\tabularnewline\bottomrule[2pt]
\end{tabular}
\par\end{center}

\begin{figure}[h]
\includefigure{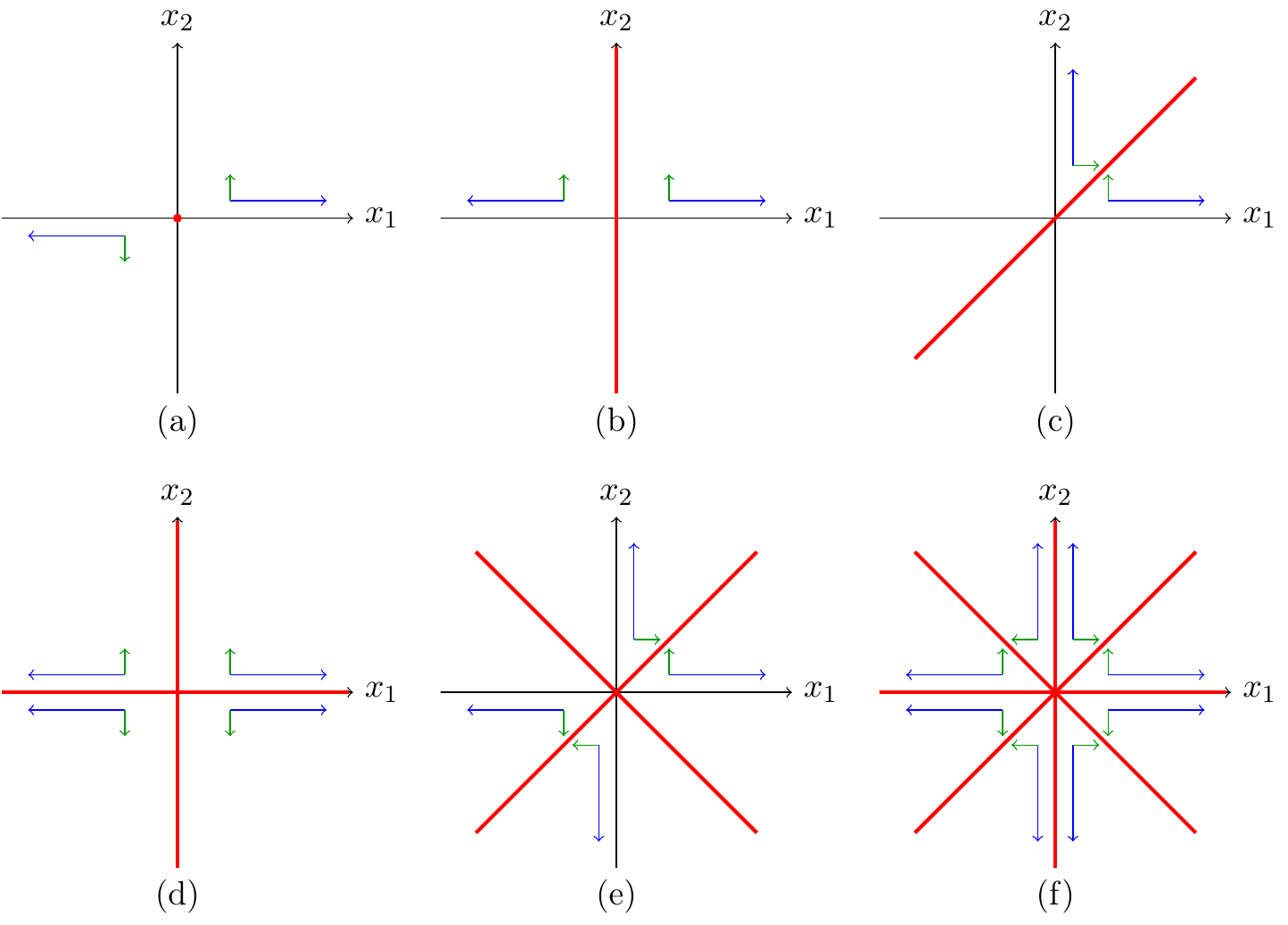}

\caption{\label{fig:symmetries}Sketch of the discrete symmetries (a)-(f) respectively
given by \eqref{symmetry-central-second}-\eqref{symmetry-yamazaki}.
The axes of symmetry are drawn in red, the first component of $\protect\bff$
in blue and the second in green.}
\end{figure}

\section{\label{sec:failure-asymptotic}Failure of standard asymptotic expansion}

The aim of this paragraph is to show that even in the case the source
force $\bff\in C_{0}^{\infty}(\mathbb{R}^{2})$ has zero mean, inverting
the Stokes operator on the nonlinearity in an attempt to solve \eqref{intro-ns}
for $\Omega=\mathbb{R}^{2}$ leads to fundamental problems concerning
the behavior at infinity. We consider a source force $\bff\in C_{0}^{\infty}(\mathbb{R}^{2})$
with zero net force,
\[
\int_{\mathbb{R}^{2}}\bff=\bzero\,.
\]
By iteratively inverting the Stokes operator the aim is to generate
an expansion in term of $\nu$ of the solution $\bu$ of \eqref{intro-ns}
with the source term $\bff$ multiplied by $\nu$,
\begin{align*}
\Delta\bu-\bnabla p & =\bu\bcdot\bnabla\bu+\nu\bff\,, & \bnabla\bcdot\bu & =0\,,
\end{align*}
in the form
\[
\bu\approx\sum_{n=1}^{\infty}\nu^{n}\bu_{n}\qquad\text{as}\qquad\nu\approx0\,.
\]
The following result shows that the successive iterates $\bu_{n}$
decay in general less and less at infinity, so the question of the
convergence of the previous series is highly nontrivial.
\begin{prop}
The first order has the following asymptotic expansion,
\begin{align*}
\bu_{1} & =\frac{-1}{4\pi r}\left[A\cos(2\theta)\be_{r}+B\sin(2\theta)\be_{r}+M\be_{\theta}\right]+O(r^{-2+\varepsilon})\,,\\
p_{1} & =\frac{1}{2\pi r^{2}}\left(A\cos(2\theta)+B\sin(2\theta)\right)+O(r^{-3+\varepsilon})\,,
\end{align*}
for any $\varepsilon>0$, the second order satisfies
\begin{align*}
\bu_{2} & =\frac{M\log r}{2(4\pi)^{2}r}\left(A\sin(2\theta)-B\cos(2\theta)\right)\be_{r}+O(r^{-1})\,,\\
p_{2} & =\frac{M\log r}{(4\pi)^{2}r^{2}}\left(A\sin(2\theta)-B\cos(2\theta)\right)+O(r^{-2})\,,
\end{align*}
and finally, the expansion of the third order is given by
\begin{align*}
\bu_{3} & =\frac{M^{2}\log^{2}r}{(8\pi)^{3}r}\left(A\cos(2\theta)+B\sin(2\theta)\right)\be_{r}+O(r^{-1}\log r)\,,\\
p_{3} & =\frac{2M^{2}\log^{2}r}{(8\pi)^{3}r^{2}}\left(A\cos(2\theta)+B\sin(2\theta)\right)+O(r^{-1}\log r)\,,
\end{align*}
for $M\neq0$ and by
\begin{align*}
\bu_{3} & =-\frac{\left(A^{2}+B^{2}\right)\log r}{12(8\pi)^{3}r}\left(A\cos(2\theta)+B\sin(2\theta)\right)\be_{r}+O(r^{-1})\,,\\
p_{3} & =-\frac{\left(A^{2}+B^{2}\right)\log r}{12(8\pi)^{3}r^{2}}\left(A\cos(2\theta)+B\sin(2\theta)\right)+O(r^{-1})\,,
\end{align*}
for $M=0$. The constants $A,B,M\in\mathbb{R}$ are given by
\begin{align*}
A & =\int_{\mathbb{R}^{2}}\left(x_{1}f_{1}-x_{2}f_{2}\right)\,, & B & =\int_{\mathbb{R}^{2}}\left(x_{1}f_{2}+x_{2}f_{1}\right)\,, & M & =\int_{\mathbb{R}^{2}}\left(x_{1}f_{2}-x_{2}f_{1}\right)\,.
\end{align*}
Therefore, unless $A=B=0$, the third order does not decay like $r^{-1}$
and therefore decays less than the Stokes solution $\bu_{1}$.\end{prop}
\begin{proof}
The first order is given by the solution of the Stokes equation,
\begin{align*}
\Delta\bu_{1}-\bnabla p_{1} & =\bff\,, & \bnabla\bcdot\bu_{1} & =0\,.
\end{align*}
By using the asymptotic expansion of the solution of the Stokes equation
obtained in \lemref{lin-stokes-asy,lin-stokes-asy-explicit}, we have
\begin{align*}
\bu_{1} & =\bS_{1}+\bR_{1}\,,\\
p_{1} & =s_{1}+r_{1}\,,
\end{align*}
where $\bS_{1}\in\mathcal{B}_{1,1}$, $\bR_{1}\in\mathcal{B}_{1,2}$,
$s_{1}\in\mathcal{B}_{0,2}$ and $r_{1}\in\mathcal{B}_{0,3}$. Explicitly,
we have
\begin{align*}
\bS_{1} & =\bC_{1}\bcdot\mathbf{E}_{1}\,, & s_{1} & =\bC_{1}\bcdot\be_{1}\,,
\end{align*}
where $\bC_{1}=\left(A,B,M\right)\in\mathbb{R}^{3}$ and
\begin{align*}
\mathbf{E}_{1} & =\bnabla\bwedge\left(\chi\boldsymbol{\Psi}_{1}\right)\,, & \boldsymbol{\Psi}_{1} & =\frac{1}{8\pi}\bigl(\sin(2\theta),\:-\cos(2\theta),\:1-2\log r\bigr)\,,
\end{align*}
\[
\be_{1}=\frac{\chi}{2\pi r^{2}}\bigl(\cos(2\theta),\:\sin(2\theta),\:0\bigr)\,.
\]
The second order has to satisfy the equation
\begin{equation}
\Delta\bu_{2}-\bnabla p_{2}=\bu_{1}\bcdot\bnabla\bu_{1}\,.\label{eq:nonlinear-Stokes-order2}
\end{equation}
However, since $\bS_{1}\bcdot\bnabla\bS_{1}\in\mathcal{B}_{0,3}$,
we cannot apply \lemref{lin-stokes-asy} in order to obtain the asymptotic
expansion of $\bu_{2}$. We make an ansatz that explicitly cancels
this term for $r>2$. We make the following ansatz for the stream
function,
\[
\psi_{2}=f_{2}(\theta)+g_{2}(\theta)\log r\,,
\]
and consider the equation for the vorticity
\[
\Delta^{2}\psi_{2}=\bnabla\bwedge\left(\bS_{1}\bcdot\bnabla\bS_{1}\right)\,,
\]
for $r>2$. We obtain the following ordinary differential equations,
\begin{align*}
g_{2}^{(4)}+4g_{2}^{(2)} & =0\,,\\
f_{2}^{(4)}+4f_{2}^{(2)}-4g_{2}^{(2)} & =\frac{1}{4\pi^{2}}\left(M-A\sin(2\theta)+B\cos(2\theta)\right)\left(A\cos(2\theta)+B\sin(2\theta)\right)\,.
\end{align*}
The periodic solutions for $g_{2}$ are
\[
g_{2}(\theta)=\lambda_{A}\cos(2\theta)+\lambda_{B}\sin(2\theta)\,.
\]
Periodic solutions for $f_{2}$ exist if and only if
\begin{align*}
\lambda_{A} & =\frac{AM}{(8\pi)^{2}}\,, & \lambda_{B}=\frac{BM}{(8\pi)^{2}}\,,
\end{align*}
and a particular solution is given by
\[
f_{2}(\theta)=\frac{1}{6(16\pi)^{2}}\left(\left(B^{2}-A^{2}\right)\sin(4\theta)+2AB\cos(4\theta)\right)\,.
\]
Therefore, by defining
\begin{align*}
\boldsymbol{A}_{2} & =\bnabla\bwedge\left(\chi\psi_{2}\right)\,,\\
a_{2} & =\frac{\chi}{r^{2}}\left[\left(1-2\log r\right)g_{1}^{\prime}(\theta)-2f_{1}^{\prime}(\theta)-\frac{A^{2}+B^{2}+2M^{2}}{(8\pi)^{2}}\right]\,,
\end{align*}
we obtain
\[
\Delta\boldsymbol{A}_{2}-\bnabla a_{2}=\bS_{1}\bcdot\bnabla\bS_{1}+\boldsymbol{\delta}_{2}\,,
\]
where $\boldsymbol{\delta}_{2}\in C_{0}^{\infty}(\mathbb{R}^{2})$
has compact support. Now by applying \lemref{lin-stokes-asy-explicit}
to \eqref{nonlinear-Stokes-order2}, we obtain
\begin{align*}
\bu_{2} & =\boldsymbol{A}_{2}+\bS_{2}+\bR_{2}\,, & p_{2} & =a_{2}+s_{2}+r_{2}\,,
\end{align*}
where $\bS_{2}\in\mathcal{B}_{1,1}$, $\bR_{2}\in\mathcal{B}_{1,2-\varepsilon}$,
$s_{2}\in\mathcal{B}_{0,2}$, and $r_{2}\in\mathcal{B}_{0,3-\varepsilon}$
for all $\varepsilon>0$. Again, we have $\bS_{2}=\bC_{2}\bcdot\mathbf{E}_{1}$
and $s_{2}=\bC_{2}\bcdot\be_{1}$ where $\bC_{2}\in\mathbb{R}^{3}$.
The terms $\boldsymbol{A}_{2}$ and $a_{2}$ contain explicit logarithms
when $M\neq0$ and $A^{2}+B^{2}\neq0$, 
\begin{align*}
\bu_{2} & =\frac{M\log r}{2(4\pi)^{2}r}\left(A\sin(2\theta)-B\cos(2\theta)\right)\be_{r}+O(r^{-1})\,,\\
p_{2} & =\frac{M\log r}{(4\pi)^{2}r^{2}}\left(A\sin(2\theta)-B\cos(2\theta)\right)+O(r^{-2})\,.
\end{align*}
In case where $M=0$, the second order has no logarithm, \emph{i.e.}
$\bu_{2}\in\mathcal{B}_{1,1}$ and $p_{2}\in\mathcal{B}_{0,2}$. However,
we will see that the third order has logarithms as soon as $A^{2}+B^{2}\neq0$.
The third order has to satisfy
\[
\Delta\bu_{3}-\bnabla p_{3}=\bu_{1}\bcdot\bnabla\bu_{2}+\bu_{2}\bcdot\bnabla\bu_{1}\,.
\]
In the same spirit as before, we make an ansatz in order to explicitly
cancel the terms of the right-hand-side that are not $o(r^{-3})$.
We make the following ansatz for the approximated stream function
at third order,
\[
\psi_{3}=f_{3}(\theta)+g_{3}(\theta)\log r+h_{3}(\theta)\log^{2}r\,.
\]
The periodic solutions are given by
\begin{align*}
h_{3}(\theta) & =\frac{M}{32\pi}g_{2}^{\prime}(\theta)\,,\\
g_{3}(\theta) & =\frac{M}{16\pi}f_{2}^{\prime}(\theta)+\frac{A^{2}+B^{2}-6M^{2}}{3(16\pi)^{3}}\left(A\sin(2\theta)-B\cos(2\theta)\right)\\
 & \phantom{=}+\frac{M}{2(4\pi)^{2}}\left(C_{21}\cos(2\theta)+C_{22}\sin(2\theta)\right)\,,\\
f_{3}(\theta) & =\frac{1}{9(32\pi)^{3}}\left(B\left(B^{2}-3A^{2}\right)\cos(6\theta)+A\left(A^{2}-3B^{2}\right)\sin(6\theta)\right)\\
 & \phantom{=}+\frac{1}{6(8\pi)^{2}}\left(\left(AC_{22}+BC_{21}\right)\cos(4\theta)+\left(BC_{22}-AC_{21}\right)\sin(4\theta\right)\,.
\end{align*}
Therefore, by defining $\boldsymbol{A}_{3}=\bnabla\bwedge\left(\chi\psi_{3}\right)$
and $a_{3}$ as a suitable pressure that we do not write here explicitly,
we obtain that
\[
\Delta\boldsymbol{A}_{3}-\bnabla a_{3}=\bS_{1}\bcdot\bnabla\left(\boldsymbol{A}_{2}+\bS_{2}\right)+\left(\boldsymbol{A}_{2}+\bS_{2}\right)\bcdot\bnabla\bS_{1}+\boldsymbol{\delta}_{3}\,,
\]
where $\boldsymbol{\delta}_{3}\in C_{0}^{\infty}(\mathbb{R}^{2})$.
We then obtain the following asymptotic expansion for $\bu_{3}$,
\begin{align*}
\bu_{3} & =\boldsymbol{A}_{3}+\bC_{3}\bcdot\mathbf{E}_{1}+\bR_{3}\,, & p_{3} & =a_{3}+\bC_{3}\bcdot\be_{1}+r_{2}\,,
\end{align*}
where $\bC_{3}\in\mathbb{R}^{3}$ and $\bR_{3}\in\mathcal{B}_{1,2-2\varepsilon}$,
$r_{3}\in\mathcal{B}_{0,3-2\varepsilon}$ for all $\varepsilon>0$.
By explicit calculations, the asymptotic expansion of the third order
is proven.
\end{proof}

\section{\label{sec:compatibility-ns}Navier-Stokes equations with compatibility
conditions}

We look at strong solutions to the stationary Navier-Stokes equations
in $\mathbb{R}^{2}$,
\begin{align}
\Delta\bu-\bnabla p & -\bu\bcdot\bnabla\bu=\bff\,, & \bnabla\bcdot\bu & =0\,, & \lim_{|\bx|\to\infty}\bu & =\bzero\,,\label{eq:ns-force}
\end{align}
and show that for source-terms $\bff$ with zero mean and in a space
of co-dimension three, the Navier-Stokes equations admit a solution
decaying like $\left|\bx\right|^{-2}$:
\begin{thm}
\label{thm:strong-under-compatibility}For all $\varepsilon\in\left(0,1\right)$,
there exists $\nu>0$ such that for any $\bk\in\mathcal{B}_{0,4+\varepsilon}$
satisfying
\begin{align*}
\left\Vert \bk;\mathcal{B}_{0,4+\varepsilon}\right\Vert  & \leq\nu\,, & \int_{\mathbb{R}^{2}}\bk & =\bzero\,,
\end{align*}
there exists $\boldsymbol{a}\in\mathbb{R}^{3}$ such that there exists
$\bu$ and $p$ satisfying \eqref{ns-force} with\index{Existence!strong solutions!under compatibility conditions}\index{Compatibility conditions!Navier-Stokes equations}\index{Navier-Stokes equations!existence of strong solutions!under compatibility
conditions}\index{Asymptotic behavior!Navier-Stokes solutions!under compatibility conditions}\index{Strong solutions!under compatibility conditions}\index{Navier-Stokes equations!asymptotic behavior!for bu_{infty}=bzero
@for $\bu_{\infty}=\bzero$}
\[
\bff=\bk+\frac{\e^{-\left|\bx\right|{}^{2}}}{\pi}\left[a_{1}\left(x_{1},-x_{2}\right)+a_{2}\left(x_{2},x_{1}\right)+a_{3}\left(-x_{2},x_{1}\right)\right]\,.
\]
Moreover, there exists $\boldsymbol{A}\in\mathbb{R}^{4}$ such that
\begin{align*}
\bu & =\frac{A_{1}}{4\pi r^{2}}\bigl(\cos(2\theta),\:\sin(2\theta)\bigr)+\frac{A_{2}}{4\pi r^{2}}\bigl(\sin(2\theta),\:-\cos(2\theta)\bigr)\\
 & \phantom{=}+\frac{A_{3}}{4\pi r^{2}}\bigl(\cos(2\theta)+\cos(4\theta),\:\sin(4\theta)\bigr)+\frac{A_{4}}{4\pi r^{2}}\bigl(\sin(2\theta)+\sin(4\theta),\:-\cos(4\theta)\bigr)+O(r^{-2-\varepsilon})\,.
\end{align*}
\end{thm}
\begin{proof}
We perform a fixed point argument on $\bu$ in the space $\mathcal{B}_{1,2}$.
We have
\begin{align}
\Delta\bu-\bnabla p & =\bN\,, & \bnabla\bcdot\bu & =0\,,\label{eq:compatibility-stokes-N}
\end{align}
with
\[
\bN=\bff+\bnabla\bcdot\left(\bu\otimes\bu\right)\,.
\]
By using \lemref{lin-stokes-asy-explicit}, the compatibility conditions
for the solution $\bu$ of the Stokes system \eqref{compatibility-stokes-N}
to decay faster than $r^{-1}$ are $\boldsymbol{C}_{0}[\bN]=\bzero$
and $\boldsymbol{C}_{1}[\bN]=\bzero$. By using the explicit form
of $\bN$, we have 
\begin{align*}
\boldsymbol{C}_{0}[\bN] & =\bzero\,, & \boldsymbol{C}_{1}[\bN] & =\boldsymbol{a}+\Lambda(\bu)\,,
\end{align*}
where
\[
\Lambda(\bu)=\int_{\mathbb{R}^{2}}\bigl(x_{1}k_{1}-x_{2}k_{2}-u_{1}u_{1}+u_{2}u_{2},\, x_{1}k_{2}+x_{2}k_{1}-2u_{1}u_{2},\, x_{1}k_{2}-x_{2}k_{1}\bigr)\,.
\]
By defining $\boldsymbol{a}=-\Lambda(\bu)$, the compatibility conditions
of the Stokes system are satisfied, and since $\bN\in\mathcal{B}_{0,4+\varepsilon}$,
then \lemref{lin-stokes-asy} proves that $\bu\in\mathcal{B}_{1,2}$.

Since,
\begin{align*}
\left|\Lambda(\bu)\right| & \leq6\left\Vert \bk;\mathcal{B}_{0,4+\varepsilon}\right\Vert \int_{\mathbb{R}^{2}}\frac{1}{1+\left|\bx\right|^{3+\varepsilon}}\rd^{2}\bx+3\left\Vert \bu;\mathcal{B}_{0,2}\right\Vert ^{2}\int_{\mathbb{R}^{2}}\frac{1}{1+\left|\bx\right|^{4}}\rd^{2}\bx\\
 & \le8\left\Vert \bk;\mathcal{B}_{0,4+\varepsilon}\right\Vert +3\left\Vert \bu;\mathcal{B}_{0,2}\right\Vert ^{2},
\end{align*}
we have
\begin{align*}
\left\Vert \bN;\mathcal{B}_{0,4+\varepsilon}\right\Vert  & \leq\left\Vert \boldsymbol{k};\mathcal{B}_{0,4+\varepsilon}\right\Vert +\left\Vert \bu;\mathcal{B}_{1,2}\right\Vert ^{2}+\left|\Lambda(\bu)\right|\\
 & \leq9\left\Vert \boldsymbol{k};\mathcal{B}_{0,4+\varepsilon}\right\Vert +4\left\Vert \bu;\mathcal{B}_{1,2}\right\Vert ^{2}.
\end{align*}
By hypothesis $\left\Vert \boldsymbol{k};\mathcal{B}_{0,4+\varepsilon}\right\Vert \leq\nu$,
so by taking $\varepsilon>0$ small enough, we can perform a fixed
point argument which shows that the Navier-Stokes system admits a
solution $\bu\in\mathcal{B}_{1,2}$.

Moreover, by using \lemref{lin-stokes-asy} and the explicit form
shown in \lemref{lin-stokes-asy-explicit}, the asymptotic behavior
is proven.
\end{proof}
Under symmetry \eqref{symmetry-yamazaki} sketched on \figref{symmetries}f,
the compatibility conditions $\boldsymbol{C}_{0}[\bN]=\bzero$ and
$\boldsymbol{C}_{1}[\bN]=\bzero$ are satisfied for $a_{1}=a_{2}=a_{3}=0$,
so the previous theorem shows that $\bu$ decays faster than $r^{-2}$:
\begin{cor}
\label{cor:strong-under-symmetries}For all $\varepsilon\in\left(0,1\right)$,
there exists $\nu>0$ such that for any $\bff\in\mathcal{B}_{0,4+\varepsilon}$
satisfying\index{Existence!strong solutions!under symmetries}\index{Navier-Stokes equations!existence of strong solutions!under symmetries}\index{Asymptotic behavior!Navier-Stokes solutions!under symmetries}\index{Strong solutions!under symmetries}\index{Navier-Stokes equations!asymptotic behavior!for bu_{infty}=bzero
@for $\bu_{\infty}=\bzero$}
\begin{align*}
\left\Vert \bff;\mathcal{B}_{0,4+\varepsilon}\right\Vert  & \leq\nu\,, & \int_{\mathbb{R}^{2}}\bff & =\bzero\,,
\end{align*}
and the symmetry conditions \eqref{symmetry-galdi} and \eqref{symmetry-yamazaki},
there exists $\bu$ and $p$ satisfying \eqref{ns-force}, and moreover
$\bu=O(r^{-2-\varepsilon})$ and $p=O(r^{-3-\varepsilon})$.\end{cor}
\begin{proof}
Since $\bff$ satisfies the symmetry conditions \eqref{symmetry-galdi}
and \eqref{symmetry-yamazaki}, due to the invariance of the Navier-Stokes
equation under axial symmetries \eqref{axial-symmetry}, $\bu$ satisfies
the same symmetry conditions, as well as the nonlinearity $\bu\bcdot\bnabla\bu$.
Therefore, we can apply \thmref{strong-under-compatibility} with
$\bff=\bk$ and $\boldsymbol{a}=\bzero$ \emph{i.e.} $\Lambda(\bu)=\bzero$.
\end{proof}
The exact solution $\frac{-M}{4\pi r}\be_{\theta}$ of the Navier-Stokes
equations generates a net torque and therefore can be used to lift\index{Compatibility conditions!lift by using an exact solution}\index{Lift of compatibility conditions}
the third component of $\bC_{1}[\bN]$ corresponding to the net torque.
More precisely, we can enlarge the class of source terms $\bff$ to
a subspace of co-dimension two: 
\begin{thm}
\label{thm:compatibility-strong-mu}For all $\varepsilon\in\left(0,1\right)$,
there exists $\nu>0$, such that for any $\bk\in\mathcal{B}_{0,3+\varepsilon}$
satisfying
\begin{align*}
\left\Vert \bk;\mathcal{B}_{0,3+\varepsilon}\right\Vert  & \leq\nu\,, & \int_{\mathbb{R}^{2}}\bk & =\bzero\,,
\end{align*}
there exists $\boldsymbol{a}\in\mathbb{R}^{2}$ such that there exists
$\bu$ and $p$ satisfying \eqref{ns-force} with
\[
\bff=\bk+\frac{\e^{-\left|\bx\right|{}^{2}}}{\pi}\left[a_{1}\left(x_{1},-x_{2}\right)+a_{2}\left(x_{2},x_{1}\right)\right].
\]
Moreover, 
\[
\bu=-\frac{M}{4\pi r}\be_{\theta}+O(r^{-1-\varepsilon})\,,
\]
where
\[
M=\int_{\mathbb{R}^{2}}\bx\bwedge\bk=\int_{\mathbb{R}^{2}}\bx\bwedge\bff\,.
\]
\end{thm}
\begin{proof}
In order to lift the compatibility condition corresponding to the
net torque, we consider the solution 
\begin{align*}
\bu_{0} & =-\frac{M}{4\pi}\bnabla\bwedge\left(\chi(r)\log(r)\right)\,, & p_{0} & =\frac{-1}{2}\left(\frac{M\chi(r)}{4\pi r}\right)^{2},
\end{align*}
which is an exact solution of the Stokes and Navier-Stokes equations
for $r\geq2$. So we have
\[
\Delta\bu_{0}-\bnabla p_{0}-\bu_{0}\bcdot\bnabla\bu_{0}=\bff_{\!0}\,,
\]
where $\bff_{\!0}$ is a source force of compact support. Since $\bu_{0}$
and $p_{0}$ are invariant under rotations, $\boldsymbol{C}_{0}[\bff_{\!0}]=\bzero$
and $\boldsymbol{C}_{1}[\bff_{\!0}]=\left(0,0,*\right)$. To determine
the last unknown, we integrate
\[
\int_{\mathbb{R}^{2}}\bx\bwedge\bff_{\!0}=-\frac{M}{2}\int_{0}^{\infty}\left[r\log r\,\chi^{(3)}(r)+\left(\log r+3\right)\chi^{\prime\prime}(r)-\frac{\log r+1}{r}\chi^{\prime}(r)\right]\rd r=M\,.
\]
Therefore, we have
\begin{align*}
\boldsymbol{C}_{0}[\bff_{\!0}] & =\bzero\,, & \boldsymbol{C}_{1}[\bff_{\!0}] & =\left(0,0,-M\right)\,.
\end{align*}
By writing $\bu$ and $p$ as $\bu=\bu_{0}+\bu_{1}$, $p=p_{0}+p_{1}$,
the Navier-Stokes equations become
\begin{align*}
\Delta\bu_{1}-\bnabla p_{1} & =\bN\,, & \bnabla\bcdot\bu_{1} & =0\,,
\end{align*}
with
\[
\bN=\bff+\bnabla\bcdot\left(\bu_{0}\otimes\bu_{1}+\bu_{1}\otimes\bu_{0}+\bu_{1}\otimes\bu_{1}\right)-\bff_{\!0}\,.
\]
The aim is to perform a fixed point on $\bu_{1}$ in the space $\mathcal{B}_{1,1+\varepsilon}$.
Since $\bff$ has zero mean by hypothesis, $\boldsymbol{C}_{0}[\bN]=\bzero$.
By defining $M=\int_{\mathbb{R}^{2}}\bx\bwedge\bff$ and by using
\propref{invariants}, the third component of $\boldsymbol{C}_{1}[\bN]$,
which is the net torque, is given by
\begin{align*}
\int_{\mathbb{R}^{2}}\bx\bwedge\bN & =\int_{\mathbb{R}^{2}}\bnabla\bcdot\left[\left(\bu_{0}\otimes\bu_{1}+\bu_{1}\otimes\bu_{0}+\bu_{1}\otimes\bu_{1}\right)\bcdot\bx^{\perp}\right]\\
 & =\lim_{R\to\infty}\int_{\partial B(\bzero,R)}\bx^{\perp}\bcdot\left(\bu_{0}\otimes\bu_{1}+\bu_{1}\otimes\bu_{0}+\bu_{1}\otimes\bu_{1}\right)\bcdot\bn\,.
\end{align*}
Since $\bu_{0}\in\mathcal{B}_{0,1}$ and $\bu_{1}\in\mathcal{B}_{0,1+\varepsilon}$
, we obtain 
\[
\left|\int_{\mathbb{R}^{2}}\bx\bwedge\bN\right|\leq\left\Vert \bu_{1};\mathcal{B}_{0,1+\varepsilon}\right\Vert \left(\left\Vert \bu_{0};\mathcal{B}_{0,1}\right\Vert +\left\Vert \bu_{1};\mathcal{B}_{0,1+\varepsilon}\right\Vert \right)\lim_{R\to\infty}R^{-\varepsilon}=0\,,
\]
so the compatibility condition for the net torque is automatically
fulfilled. In the same way as in the proof of \thmref{strong-under-compatibility},
we have 
\begin{align*}
\boldsymbol{C}_{0}[\bN] & =\bzero\,, & \boldsymbol{C}_{1}[\bN] & =\bigl(\boldsymbol{a}+\Lambda(\bu_{1}),0\bigr)\,,
\end{align*}
where
\[
\Lambda(\bu_{1})=\int_{\mathbb{R}^{2}}\bigl(x_{1}k_{1}-x_{2}k_{2}-u_{1}u_{1}+u_{2}u_{2},\, x_{1}k_{2}+x_{2}k_{1}-2u_{1}u_{2}\bigr)\,.
\]
By defining $\boldsymbol{a}=-\Lambda(\bu)$, the compatibility conditions
of the Stokes system to decay faster than $r^{-1}$ are satisfied.
Therefore, it remains to bound $\bN$ is order to apply a fixed point
argument. We have
\[
\left\Vert \bu_{0};\mathcal{B}_{1,1}\right\Vert \lesssim\left|M\right|\lesssim\left\Vert \bk;\mathcal{B}_{0,3+\varepsilon}\right\Vert \leq\nu\,,
\]
and
\begin{align*}
\left|\Lambda(\bu)\right| & \lesssim\left\Vert \bk;\mathcal{B}_{0,3+\varepsilon}\right\Vert +\left\Vert \bu_{1};\mathcal{B}_{0,1+\varepsilon}\right\Vert \left(\left\Vert \bu_{0};\mathcal{B}_{0,1}\right\Vert +\left\Vert \bu_{1};\mathcal{B}_{0,1+\varepsilon}\right\Vert \right)\\
 & \lesssim\nu+\nu\left\Vert \bu_{1};\mathcal{B}_{0,1+\varepsilon}\right\Vert +\left\Vert \bu_{1};\mathcal{B}_{0,1+\varepsilon}\right\Vert ^{2},
\end{align*}
so
\begin{align*}
\left\Vert \bN;\mathcal{B}_{0,4+\nu}\right\Vert  & \leq\left\Vert \boldsymbol{k};\mathcal{B}_{0,3+\varepsilon}\right\Vert +\left\Vert \bu_{1};\mathcal{B}_{1,1+\varepsilon}\right\Vert \left(\left\Vert \bu_{0};\mathcal{B}_{1,1}\right\Vert +\left\Vert \bu_{1};\mathcal{B}_{1,1+\varepsilon}\right\Vert \right)+\left|\Lambda(\bu)\right|\\
 & \lesssim\nu+\nu\left\Vert \bu_{1};\mathcal{B}_{1,1+\varepsilon}\right\Vert +\left\Vert \bu_{1};\mathcal{B}_{1,1+\varepsilon}\right\Vert ^{2}.
\end{align*}
Consequently, by applying \lemref{lin-stokes-asy}, we can perform
a fixed point argument on $\bu_{1}\in\mathcal{B}_{1,1+\varepsilon}$
which proves the existence of a solution $\bu=\bu_{0}+\bu_{1}$ of
the Navier-Stokes system together with the claimed asymptotic behavior,
since $\bu_{0}=\frac{-M}{4\pi r}\be_{\theta}$ for $r\ge2$.\end{proof}
\begin{rem}
This theorem shows that the knowledge of one suitable explicit solution
of the Navier-Stokes equations can be used to lift one compatibility
condition and enlarge the space of source forces $\bff$ for which
we can solve the problem. The compatibility condition we lifted is
the one related to net torque, which is an invariant quantity, so
we do not need to adjust $M$ inside the fixed point, \emph{i.e.}
$M$ depends only on $\bff$ not on $\bu_{1}$. In the case where
we try to lift a compatibility condition that is not a conserved quantity,
we would have to adjust the parameter of the explicit solution at
each iteration of the fixed point argument.
\end{rem}

\begin{rem}
The method used in this theorem cannot be applied to the case where
$\bff$ has nonzero mean for the following reason. In order to treat
the nonlinearity by inverting the Stokes operator on it, the explicit
solution $\bu_{0}$ that lifts the compatibility condition has to
be in the space $\mathcal{B}_{1,1}$ and the perturbation $\bu_{1}$
in $\mathcal{B}_{1,1+\varepsilon}$ for some $\varepsilon>0$, otherwise
the inversion of the Stokes operator on the nonlinearity leads to
logarithms, and the fixed point argument cannot be closed. But we
cannot lift the mean value of the force $\bF$ with an explicit solution
$\bu_{0}\in\mathcal{B}_{1,1}$: if $\bu_{0}\in\mathcal{B}_{1,1}$
and $p_{0}\in\mathcal{B}_{0,2}$, we have
\[
\mathbf{T}_{0}=\bnabla\bu_{0}+\left(\bnabla\bu_{0}\right)^{T}-p_{0}\mathbf{1}-\bu_{0}\otimes\bu_{0}\in\mathcal{B}_{0,2}\,,
\]
so by using \propref{invariants},
\[
\left|\bF_{0}\right|=\left|\int_{\mathbb{R}^{2}}\bff_{0}\right|=\lim_{R\to\infty}\biggl|\int_{B(\bzero,R)}\!\mathbf{T}_{0}\bn\biggr|\leq\left\Vert \mathbf{T}_{0};\mathcal{B}_{0,2}\right\Vert \lim_{R\to\infty}R^{-1}=0\,.
\]

\end{rem}

\chapter{\label{chap:link-Stokes-NS}On the asymptotes of the Stokes and Navier-Stokes
equations}

We consider the Navier-Stokes equations in the exterior domain $\Omega=\mathbb{R}^{2}\setminus B$
where $B$ is a compact and simply connected set with smooth boundary,
\begin{equation}
\begin{aligned}\Delta\bu-\bnabla p & =\nu\bu\bcdot\bnabla\bu+\bff\,, & \bnabla\bcdot\bu & =0\,,\\
\left.\bu\right|_{\partial B} & =\bu^{*}\,, & \lim_{|\bx|\to\infty}\bu & =\bzero\,,
\end{aligned}
\label{eq:link-ns}
\end{equation}
where $\nu\in\mathbb{R}$ is a parameter, $\bu^{*}$ is a boundary
condition and $\bff$ a source force. These equations admit four invariant
quantities (see \propref{invariants}): the flux $\Phi$, the net
force $\bF$, and the net torque $M$,
\begin{align*}
\Phi= & \int_{\partial B}\bu\bcdot\bn\,, & \bF & =\int_{\Omega}\bff+\int_{\partial B}\mathbf{T}\bn\,, & M & =\int_{\Omega}\bx\bwedge\bff+\int_{\partial B}\bx\bwedge\mathbf{T}\bn\,,
\end{align*}
where $\mathbf{T}$ is the stress tensor including the convective
part,
\[
\mathbf{T}=\bnabla\bu+\left(\bnabla\bu\right)^{T}-p\,\boldsymbol{1}-\nu\bu\otimes\bu\,.
\]
For $\nu=0$, the system \eqref{link-ns} is linear and is called
the Stokes equations, whereas if $\nu\neq0$, the equations are the
Navier-Stokes equations. \citet{Deuring.Galdi-AsymptoticBehaviorof2000}
proved that in three dimensions, the solution of the Navier-Stokes
equations cannot be asymptotic to the Stokes fundamental solution.
The aim of this chapter is to prove an analog result in two dimensions.
In contrast to the three-dimensional case, the requirement that the
velocity vanishes at infinity imposes that the net force vanishes
for the Stokes equations. The asymptotic expansion of the Stokes equations
up to order $r^{-1}$ has four real parameters. Moreover, the velocity
of the Navier-Stokes equations can be asymptotic only to two of the
four terms in $r^{-1}$ of the Stokes asymptote. The existence of
such a solution was proven in \thmref{compatibility-strong-mu}. These
two terms are the two harmonic functions\index{Exact solutions!harmonic}
decaying like $r^{-1}$ and therefore the asymptotic expansion of
the pressure up to order $r^{-2}$ cannot coincide since the pressure
term of an harmonic function is given by $\frac{\nu}{2}\left|\bu\right|^{2}$.

The following theorem provides the main result of this chapter:
\begin{thm}
\label{thm:link-asy-S-NS}Let $\varepsilon\in\left(0,1\right)$, $\nu\in\mathbb{R}$,
$\bff\in C^{0}(\Omega)$ such that $\left(1+\left|\bx\right|^{3+\varepsilon}\right)\bff\in L^{\infty}(\Omega)$
and let $\left(\bu,p\right)\in C^{2}(\Omega)\times C^{1}(\Omega)$
be a solution of the Navier-Stokes equations \eqref{link-ns}.
\begin{enumerate}
\item If $\nu=0$, then there exists $\bA=(A_{0},A_{1},A_{2},A_{3})\in\mathbb{R}^{4}$
such that
\begin{align*}
\bu & =\bu_{1}+O(r^{-1-\varepsilon})\,, & \bnabla\bu & =\bnabla\bu_{1}+O(r^{-1-\varepsilon})\,, & p & =p_{1}+O(r^{-2-\varepsilon})\,,
\end{align*}
where
\begin{align*}
\bu_{1} & =\frac{1}{4\pi r}\left[2A_{0}\be_{r}-A_{1}\cos(2\theta)\be_{r}-A_{2}\sin(2\theta)\be_{r}-A_{3}\be_{\theta}\right],\\
p_{1} & =\frac{-1}{4\pi r^{2}}\left[A_{1}\cos(2\theta)+A_{2}\sin(2\theta)\right].
\end{align*}
Moreover, the net force vanishes, $\bF=\bzero$, the parameters of
the vector $\bA$ can be expressed in terms of integrals involving
$\bu$ and $\bff$, and in particular
\begin{align*}
A_{0} & =\Phi\,, & A_{3}= & M\,.
\end{align*}

\item If $\nu\neq0$ and $\bu$ satisfies
\[
\bu=\bu_{1}+O(r^{-1-\varepsilon})\,,
\]
for some $\bA=(A_{0},A_{1},A_{2},A_{3})\in\mathbb{R}^{4}$, then $A_{1}=A_{2}=0$
and $A_{0}=\Phi$. If in addition $p$ satisfies
\[
p=p_{1}+O(r^{-2-\varepsilon})\,,
\]
then $\Phi=M=0$, so $\bA=\bzero$ and $\bu=O(r^{-1-\varepsilon})$.
Moreover if
\[
\bnabla\bu=\bnabla\bu_{1}+O(r^{-1-\varepsilon})\,,
\]
then the net force vanishes $\bF=\bzero$ and the net torque is $M=A_{3}$.
\end{enumerate}
\end{thm}

\section{Truncation procedure\index{Truncation procedure}}

The aim of this section is to show that by using a cut-off procedure
we can get rid of the body and consider modified equations in $\mathbb{R}^{2}$.

Since $B$ is compact, there exits $R>1$ such that $B\subset B(\bzero,R)$,
where $B(\bzero,R)$ is the ball of radius $R$ centered at the origin.
We denote by $\chi$ a smooth cut-off function such that $\chi(r)=0$
for $0\leq r\leq R$ and $\chi(r)=1$ if $r\geq2R$. The flux is defined
by
\[
\Phi=\int_{\partial B}\bu\bcdot\bn\,.
\]
To deal with the flux in $\mathbb{R}^{2}$, we define the following
smooth flux carrier\index{Flux carrier},
\begin{align*}
\bsigma & =\frac{\chi(r)}{2\pi r}\be_{r}\,, & \sigma & =\frac{\chi^{\prime}(r)}{2\pi r}\,,
\end{align*}
which is smooth in $\mathbb{R}^{2}$ and an exact solution of \eqref{link-ns}
for $\nu=0$ in $\mathbb{R}^{2}\setminus B(\bzero,2R)$ with $\bff=\bzero$
and $\bu^{*}=\frac{\be_{r}}{2\pi r}$.
\begin{prop}
\label{prop:truncation}Let $\left(\bu,p\right)\in C^{2}(\Omega)\times C^{1}(\Omega)$
be a solution of \eqref{link-ns}. Then there exists a solution $\left(\bar{\bu},\bar{p}\right)\in C^{2}(\mathbb{R}^{2})\times C^{1}(\mathbb{R}^{2})$
of
\begin{align}
\Delta\bar{\bu}-\bnabla\bar{p} & =\nu\bar{\bu}\bcdot\bnabla\bar{\bu}+\bar{\bff}\,, & \bnabla\bcdot\bar{\bu} & =\Phi\bnabla\bcdot\bsigma\,, & \lim_{\left|\bx\right|\to\infty} & \bu=\bzero\,,\label{eq:ns-plane}
\end{align}
in $\mathbb{R}^{2}$ such that $\bu=\bar{\bu}$, $p=\bar{p}$, and
$\bff=\bar{\bff}$ in $\mathbb{R}^{2}\setminus B(\bzero,2R)$.\end{prop}
\begin{proof}
First of all let $\bv=\bu-\Phi\bsigma$ and $q=p-\Phi\sigma$ so that
$\bv$ has zero flux,
\[
\int_{\partial B(\bzero,R)}\bv\bcdot\bn=\int_{\partial B}\bu\bcdot\bn-\Phi\int_{\partial B(\bzero,R)}\bsigma\bcdot\bn=0\,,
\]
and therefore the function
\[
\psi(\bx)=\int_{\gamma(\bx)}\bv^{\perp}\bcdot\rd\bx\,,
\]
where $\gamma(\bx)$ is any smooth curve connecting $\left(R,0\right)$
to $\bx$ is a stream function for $\bv$, \emph{i.e.} $\bv=\bnabla\bwedge\psi$
in $\mathbb{R}^{2}\setminus B(\bzero,R)$. Since $\psi\in C^{2}(\mathbb{R}^{2}\setminus B(\bzero,R))$,
by defining
\begin{align*}
\bar{\bu} & =\Phi\bsigma+\bar{\bv}\,, & \bar{\bv} & =\bnabla\bwedge\left(\chi\psi\right)\,, & \bar{p} & =\Phi\sigma+\chi q\,,
\end{align*}
we have $\left(\bar{\bu},\bar{p}\right)\in C^{2}(\mathbb{R}^{2})\times C^{1}(\mathbb{R}^{2})$,
$\bar{\bu}=\bu$ and $\bar{p}=p$ for $r\geq2R$. By plugging $\left(\bar{\bu},\bar{p}\right)$
into \eqref{ns-plane}, we obtain
\begin{align*}
\Delta\bar{\bu}-\bnabla\bar{p}-\nu\bar{\bu}\bcdot\bnabla\bar{\bu} & =\chi\bff+\boldsymbol{\delta}\,, & \bnabla\bcdot\bar{\bu} & =\Phi\bnabla\bcdot\bsigma\,,
\end{align*}
where $\boldsymbol{\delta}\in C_{0}^{1}(\mathbb{R}^{2})$ and $\bnabla\bcdot\bsigma\in C_{0}^{\infty}(\mathbb{R}^{2})$
have support only on $B(\bzero,2R)$. The proposition is proved by
taking $\bar{\bff}=\chi\bff+\boldsymbol{\delta}$.
\end{proof}

\section{Stokes equations}

In this section, we prove the first part of the theorem concerning
the linear case: $\nu=0$. We first define weighted $L^{\infty}$-spaces:
\begin{defn}[function spaces]
For $q\geq0$, we define the weight
\[
w_{q}(\bx)=\begin{cases}
1+\left|\bx\right|^{q}\,, & q>0\,,\\
\left[\log\left(2+\left|\bx\right|\right)\right]^{-1}\,, & q=0\,,
\end{cases}
\]
and the associated Banach space for $k\in\mathbb{N}$,
\[
\mathcal{B}_{k,q}=\left\{ f\in C^{k}(\mathbb{R}^{n})\,:\: w_{q+\left|\alpha\right|}D^{\alpha}f\in L^{\infty}(\mathbb{R}^{n})\:\forall\left|\alpha\right|\leq k\right\} ,
\]
with the norm
\[
\bigl\Vert f;\mathcal{B}_{k,q}\bigr\Vert=\max_{\left|\alpha\right|\leq k}\sup_{\bx\in\mathbb{R}^{n}}w_{q+\left|\sigma\right|}\left|D^{\alpha}f\right|\,.
\]
\end{defn}
\begin{prop}
\label{prop:stokes}If $\bar{\bff}\in\mathcal{B}_{0,3+\varepsilon}$
for $\varepsilon\in\left(0,1\right)$, the solution of \eqref{ns-plane}
with $\nu=0$ satisfies
\begin{align*}
\bar{\bu} & =\bar{\bu}_{1}+\tilde{\bu}\,, & \bar{p} & =\bar{p}_{1}+\tilde{p}\,,
\end{align*}
where $\tilde{\bu}\in\mathcal{B}_{1,1+\varepsilon}$, $\tilde{p}\in\mathcal{B}_{0,2+\varepsilon}$,
and
\begin{align*}
\bar{\bu}_{1} & =\bA\bcdot\left(\bsigma,\mathbf{E}_{1}\right)\,, & \bar{p}_{1} & =\bA\bcdot\left(\sigma,\be_{1}\right)\,,
\end{align*}
for some $\bA\in\mathbb{R}^{4}$, with $A_{0}=\Phi$. The first order
of the asymptotic expansion $\mathbf{E}_{1}$ (which is a tensor of
type $\left(2,3\right)$) and $\be_{1}$ are defined in \lemref{lin-stokes-asy-explicit}.\end{prop}
\begin{proof}
By plugging $\bar{\bu}=\Phi\bsigma+\bar{\bv}$ and $\bar{p}=\Phi\sigma+\bar{q}$
in \eqref{ns-plane}, we obtain
\begin{align*}
\Delta\bar{\bv}-\bnabla\bar{q} & =\bar{\bff}\,, & \bnabla\bcdot\bar{\bv} & =\bzero\,.
\end{align*}
By \lemref{lin-stokes-asy}, we have
\begin{align*}
\bar{\bv} & =\bS_{0}+\bS_{1}+\tilde{\bu}\,, & \bar{q} & =s_{0}+s_{1}+\tilde{p}\,,
\end{align*}
where $\bS_{i}\in\mathcal{B}_{1,i}$, $\tilde{\bu}\in\mathcal{B}_{1,1+\varepsilon}$,
$s_{i}\in\mathcal{B}_{0,i+1}$ and $\tilde{p}\in\mathcal{B}_{0,2+\varepsilon}$.
In particular, by using \lemref{lin-stokes-asy-explicit}, the terms
are given by
\begin{align*}
\bS_{0} & =\bC_{0}\bcdot\mathbf{E}_{0}\,, & s_{0} & =\bC_{0}\bcdot\be_{0}\,,\\
\bS_{1} & =\bC_{1}\bcdot\mathbf{E}_{1}\,, & s_{1} & =\bC_{1}\bcdot\be_{1}\,,
\end{align*}
where $\bC_{0},\bC_{1}\in\mathbb{R}^{3}$ are given by integrals of
$\bar{\bff}$. The term $\mathbf{E}_{0}$ grows at infinity like $\log r$
and since the velocity has to be zero at infinity, the term $\bS_{0}$
has to vanish, so $\bC_{0}=\bzero$.
\end{proof}

\section{Navier-Stokes equations}

In this section we prove the second part of the theorem concerning
the case where $\nu\neq0$. The term $\bar{\bu}_{1}\in\mathcal{B}_{1,1}$
generates a nonlinear term $\nu\bar{\bu}_{1}\bcdot\bnabla\bar{\bu}_{1}\in\mathcal{B}_{0,3}$,
so we cannot apply \propref{stokes} to solve the following Stokes
system,
\begin{align*}
\Delta\bar{\bu}_{2}-\bnabla\bar{p}_{2} & =\nu\bar{\bu}_{1}\bcdot\bnabla\bar{\bu}_{1}\,, & \bnabla\bcdot\bar{\bu}_{2} & =0\,.
\end{align*}
In the following key lemma, we explicitly construct a solution to
this system up to a compactly supported function and determine its
asymptotic behavior.
\begin{lem}
\label{lem:explicit-sol-u2}There exists a smooth solution $\left(\bar{\bu}_{2},\bar{p}_{2}\right)$
of the equations
\begin{align*}
\Delta\bar{\bu}_{2}-\bnabla\bar{p}_{2} & =\nu\bar{\bu}_{1}\bcdot\bnabla\bar{\bu}_{1}+\boldsymbol{\delta}_{2}\,, & \bnabla\bcdot\bar{\bu}_{2} & =0\,,
\end{align*}
in $\mathbb{R}^{2}$ where $\boldsymbol{\delta}_{2}\in C_{0}^{1}(\mathbb{R}^{2})$
has compact support, such that for $r\geq2R$,
\begin{equation}
\begin{aligned}\bar{\bu}_{2} & =\frac{\nu\mathcal{A}_{1}\mathcal{A}_{2}}{(8\pi)^{2}r}\left[\log r\,\sin\left(2\theta+\theta_{1}\right)\be_{r}+\cos\left(2\theta+\theta_{1}\right)\be_{\theta}\right]+\frac{\nu\mathcal{A}_{1}^{2}}{6(8\pi)^{2}r}\sin(4\theta+\theta_{2})\be_{r}\\
\bar{p}_{2} & =\frac{\nu\mathcal{A}_{1}\mathcal{A}_{2}}{32\pi^{2}r^{2}}\left(2\log r-1\right)\sin\left(2\theta+\theta_{1}\right)+\frac{\nu\mathcal{A}_{1}^{2}}{3(8\pi)^{2}r^{2}}\sin(4\theta+\theta_{2})-\frac{\nu\left(\mathcal{A}_{1}^{2}+\mathcal{A}_{2}^{2}\right)}{(8\pi)^{2}r^{2}}\,,
\end{aligned}
\label{eq:asy-v2-q2}
\end{equation}
where
\begin{align*}
\mathcal{A}_{1} & =\sqrt{A_{1}^{2}+A_{2}^{2}}\,, & \mathcal{A}_{2} & =\sqrt{4A_{0}^{2}+A_{3}}\,,
\end{align*}
and $\theta_{1}$, $\theta_{2}$ are angles related to $A_{1}$ and
$A_{2}$.\end{lem}
\begin{proof}
We make an ansatz that explicitly cancel this term for $r>2R$. We
make the following ansatz for the stream function,
\[
\psi_{2}=f_{2}(\theta)+g_{2}(\theta)\log r\,,
\]
and consider the equation of the vorticity
\[
\Delta^{2}\psi_{2}=\bnabla\bwedge\left(\nu\bar{\bu}_{1}\bcdot\bnabla\bar{\bu}_{1}\right)\,,
\]
for $r>2R$. We obtain the following ordinary differential equations,
\begin{align*}
g_{2}^{(4)}+4g_{2}^{(2)} & =0\,,\\
f_{2}^{(4)}+4f_{2}^{(2)}-4g_{2}^{(2)} & =\frac{\nu\mathcal{A}_{1}}{8\pi^{2}}\left[2\mathcal{A}_{2}\cos(2\theta+\theta_{1})+\mathcal{A}_{1}\cos(4\theta+\theta_{2})\right],
\end{align*}
where $\theta_{2}$ and $\theta_{4}$ are angles expressed in terms
of $A$ and $B$. The periodic solutions for $g_{2}$ are
\[
g_{2}(\theta)=\lambda\cos(2\theta+\theta_{0})\,.
\]
Periodic solutions for $f_{2}$ exist if and only if
\begin{align*}
\lambda & =\frac{\nu\mathcal{A}_{1}\mathcal{A}_{2}}{(8\pi)^{2}}\,, & \theta_{0}=\theta_{2}\,,
\end{align*}
and a particular solution is given by
\[
f_{2}(\theta)=\frac{\nu\mathcal{A}_{1}^{2}}{6(16\pi)^{2}}\left(\cos(4\theta+\theta_{2})\right).
\]
Therefore, by defining
\begin{align*}
\bar{\bu}_{2} & =\bnabla\bwedge\left(\chi\psi_{2}\right),\\
\bar{p}_{2} & =\frac{\nu\chi}{r^{2}}\left[\left(1-2\log r\right)g_{1}^{\prime}(\theta)-2f_{1}^{\prime}(\theta)-\frac{\mathcal{A}_{1}^{2}+2\mathcal{A}_{2}^{2}}{(8\pi)^{2}}\right],
\end{align*}
the lemma is proven.
\end{proof}
By applying this lemma we obtain:
\begin{prop}
\label{prop:navier-stokes}Let $\varepsilon\in(0,1)$, $\bar{\bff}\in\mathcal{B}_{0,3+\varepsilon}$
and $\left(\bar{\bu},\bar{p}\right)\in C^{2}(\Omega)\times C^{1}(\Omega)$
be a solution of \eqref{ns-plane} for $\nu\neq0$. If $\bar{\bu}$
is asymptotic to the solution of the Stokes equations, \emph{i.e.
\[
\bar{\bu}=\bA\bcdot\left(\bsigma,\mathbf{E}_{1}\right)+\tilde{\bu}\,,
\]
}for some $\bA=\left(A_{0},A_{1},A_{2},A_{3}\right)\in\mathbb{R}^{4}$
and $\tilde{\bu}\in\mathcal{B}_{0,1+\varepsilon}$ then $A_{0}=\Phi$
and $A_{1}=A_{2}=0$. Moreover if $\bar{p}$ is asymptotic to the
solution of the stokes equations, \emph{i.e.}
\[
\bar{p}=\bA\bcdot\left(\bsigma,\mathbf{E}_{1}\right)+\tilde{p}
\]
for some $\tilde{p}\in\mathcal{B}_{0,2+\varepsilon}$, then $\bA=\bzero$.\end{prop}
\begin{proof}
We write the solution as
\begin{align*}
\bar{\bu} & =\bar{\bu}_{1}+\tilde{\bu}\,, & \bar{p} & =\bar{p}_{1}+\tilde{p}\,,
\end{align*}
where
\begin{align*}
\bar{\bu}_{1} & =A_{0}\bsigma+\left(A_{1},A_{2},A_{3}\right)\bcdot\mathbf{E}_{1}\,,\\
\bar{p}_{1} & =A_{0}\sigma+\left(A_{1},A_{2},A_{3}\right)\bcdot\be_{1}\,.
\end{align*}
Since $\bnabla\bcdot\bar{\bu}_{1}=A_{0}\bnabla\bcdot\bsigma$, the
system \eqref{ns-plane} becomes explicitly
\begin{align}
\Delta\tilde{\bu}-\bnabla\tilde{p} & =\bar{\bu}\bcdot\bnabla\bar{\bu}+\bar{\bff}\,, & \bnabla\bcdot\tilde{\bu} & =\left(\Phi-A_{0}\right)\bnabla\bcdot\bsigma\,.\label{eq:ns-bar}
\end{align}
For any $n\geq2R$, we have
\[
\int_{B(\bzero,n)}\bnabla\bcdot\bsigma=\int_{\partial B(\bzero,n)}\bsigma\bcdot\bn=1\,,
\]
and therefore by using \eqref{ns-bar}, we obtain
\[
\left(\Phi-A_{0}\right)=\int_{B(\bzero,n)}\bnabla\bcdot\tilde{\bu}=\int_{\partial B(\bzero,n)}\tilde{\bu}\,.
\]
By hypothesis $\tilde{\bu}\in\mathcal{B}_{0,1+\varepsilon}$ and we
have
\[
\left|\Phi-A_{0}\right|\leq\int_{\partial B(\bzero,n)}\left|\tilde{\bu}\right|\leq\left\Vert \tilde{\bu};\mathcal{B}_{0,1+\varepsilon}\right\Vert \int_{\partial B(\bzero,n)}\frac{1}{\left|\bx\right|^{1+\varepsilon}}\leq2\pi\left\Vert \tilde{\bu};\mathcal{B}_{0,1+\varepsilon}\right\Vert n^{-\varepsilon}\,,
\]
so by taking the limit $n\to\infty$, we obtain that $A_{0}=\Phi$.
By \lemref{explicit-sol-u2}, $\left(\bar{\bu}_{2},\bar{p}_{2}\right)$
satisfies
\begin{align*}
\Delta\bar{\bu}_{2}-\bnabla\bar{p}_{2} & =\nu\bar{\bu}_{1}\bcdot\bnabla\bar{\bu}_{1}+\boldsymbol{\delta}_{2}\,, & \bnabla\bcdot\bar{\bu}_{2} & =0\,,
\end{align*}
where $\boldsymbol{\delta}_{2}\in C_{0}^{1}(\mathbb{R}^{2})$. By
defining $\tilde{\bu}=\bar{\bu}_{2}+\bar{\bu}_{3}$ and $\tilde{p}=\bar{p}_{2}+\bar{p}_{3}$,
the system \eqref{ns-bar} is equivalent to
\begin{align*}
\Delta\bar{\bu}_{3}-\bnabla\bar{p}_{3} & =\bar{\bff}-\boldsymbol{\delta}_{2}+\bnabla\bcdot\mathbf{N}\,, & \bnabla\bcdot\bar{\bu}_{3} & =\bzero\,,
\end{align*}
where
\[
\boldsymbol{N}=\nu\bar{\bu}_{1}\otimes\bar{\bv}+\nu\bar{\bv}\otimes\bar{\bu}_{1}+\nu\bar{\bv}\otimes\bar{\bv}\in\mathcal{B}_{0,2+\varepsilon}\,.
\]
The solution of this system can be represented after an integration
by parts by
\begin{align*}
\bar{\bu}_{3} & =\mathbf{E}*\left(\bar{\bff}-\boldsymbol{\delta}_{2}\right)+\bnabla\mathbf{E}*\mathbf{N}\,,\\
\bar{p}_{3} & =\be*\left(\bar{\bff}-\boldsymbol{\delta}_{2}\right)+\bnabla\be*\mathbf{N}\,.
\end{align*}
The asymptotic expansion of the first term of the right-hand-side
was already given in \propref{stokes}. The asymptote of the second
term of the right-hand-side was already computed in \lemref{lin-stokes-asy}
in the estimate of the derivatives of the velocity and of the pressure.
Therefore, we obtain that there there exists $\bC\in\mathbb{R}^{3}$
such that
\begin{align*}
\bar{\bu}_{3} & =\bC_{0}\bcdot\mathbf{E}_{0}+\bC_{1}\bcdot\mathbf{E}_{1}+O(r^{-1-\varepsilon})\,,\\
\bar{p}_{3} & =\bC_{0}\bcdot\mathbf{E}_{0}+\bC_{1}\bcdot\be_{1}+O(r^{-2-\varepsilon})\,.
\end{align*}
Since by hypothesis $\tilde{\bu}=\bar{\bu}_{2}+\bar{\bu}_{3}\in\mathcal{B}_{0,1+\varepsilon}$,
we deduce that $\bC_{0}=\bzero$, otherwise the solution grows at
infinity. Then in view of \eqref{asy-v2-q2}, we obtain that $\mathcal{A}_{1}=0$
so $A_{1}=A_{2}=0$, and finally we deduce that $\bC_{1}=\bzero$.
Finally if moreover we assume that $\tilde{p}\in\mathcal{B}_{0,2+\varepsilon}$,
then in view of \eqref{asy-v2-q2} we obtain that $A_{0}=A_{3}=0$,
so $\bA=\bzero$.
\end{proof}

\begin{proof}[Proof of \thmref{link-asy-S-NS}]
By \propref{truncation}, we can transform the original equations
in $\Omega$ to \eqref{ns-plane} in $\mathbb{R}^{2}$. Then \propref{stokes,navier-stokes}
prove respectively the first part and the second part of the theorem.
These propositions also show that $A_{0}=\Phi$. The determination
of the net force and of the component $A_{3}$ of $\bA$ are now deduced
by using the asymptotic behavior of $\bu$ and $\bnabla\bu$. First
of all, by the same argument as used in Since the net force $\bF$
is an invariant quantity, we find in the truncated domain $\Omega_{n}=\Omega\cap B(\bzero,n)$
that
\[
\int_{\Omega_{n}}\bff=\int_{\partial B(\bzero,n)}\mathbf{T}\bn-\int_{\partial B}\mathbf{T}\bn\,,
\]
for all $n\geq R$, and therefore
\[
\bF=\int_{\Omega}\bff+\int_{\partial B}\mathbf{T}\bn=\lim_{n\to\infty}\int_{\partial B(\bzero,n)}\mathbf{T}\bn\,.
\]
Therefore, if $\tilde{\bu}\in\mathcal{B}_{1,1+\varepsilon}$, then
$\bu\in\mathcal{B}_{1,1}$ and $\mathbf{T}=O(\left|\bx\right|^{-2})$
so by taking the limit $n\to\infty$, we deduce that $\bF=\bzero$.
By using the same procedure for the net torque, we obtain,
\[
M=\int_{\Omega}\bx\bwedge\bff+\int_{\partial B}\bx\bwedge\mathbf{T}\bn=\lim_{n\to\infty}\int_{\partial B(\bzero,n)}\bx\bwedge\mathbf{T}\bn\,.
\]
and since $\mathbf{T}=\bnabla\bu_{1}+\left(\bnabla\bu_{1}\right)^{T}-p_{1}\,\boldsymbol{1}-\nu\bu_{1}\otimes\bu_{1}+O(\left|\bx\right|^{-2})$
we obtain by an explicit calculation that $M=A_{3}$.
\end{proof}

\chapter{\label{chap:wake-sim}On the general asymptote with vanishing velocity
at infinity}

In this chapter, we analyze the existence of solutions for the two-dimensional
Navier-Stokes equations converging to zero at infinity. A crucial
point towards showing the existence of such solutions is to determine
the asymptotic decay and behavior of the solution. The aim is to determine
the two-dimensional analog of the \citet{Landau-newexactsolution1944}
solution which plays a crucial role in three-dimensions \citep{Korolev.Sverak-largedistanceasymptotics2011}.
In the supercritical regime, i.e. when the net force is nonzero, we
provide an asymptotic solution $\bUF$ with a wake structure and decaying
like $|\bx|^{-1/3}$ and conjecture that all solutions with a nonzero
net force $\bF$ will behaves at infinity like $\bUF$ at least for
small data. The asymptotic behavior $\bUF$ was found by \citet{Guillod-Asymptoticbehaviour2013}
in Cartesian coordinates and here we use a conformal change of coordinates,
which simplifies and provides a better understanding of the asymptote.
Finally, we perform numerical simulations to analyze the validity
of the conjecture and to determine the possible asymptotic behaviors
when the net force vanishes. In this later case, the general asymptotic
behavior seems to be very far from trivial.

\section{Introduction}

As already said in the introduction, the Navier-Stokes equations in
three dimensions are critical if $\bF\neq\bzero$ and the velocity
decays like $\left|\bx\right|^{-1}$ and is asymptotic to the \citet{Landau-newexactsolution1944}
solution. If $\bF=\bzero$, the three-dimensional equations are subcritical:
the velocity decays like $\left|\bx\right|^{-2}$ and is asymptotic
to the Stokes solution. In two dimensions, the velocity field has
to decay less than $\left|\bx\right|^{-1/2}$ in order to generate
a not zero net force, so the equations are supercritical if $\bF\neq\bzero$.
If $\bF=\bzero$, the two-dimensional Navier-Stokes equations are
critical as the three-dimensional ones for $\bF\neq\bzero$, however,
there are crucial differences that make the two-dimensional problem
much more difficult.

We now review the results on the three-dimensional case. The Stokes
fundamental solution decays like $\left|\bx\right|^{-1}$ and in case
$\bF=\bzero$ like $\left|\bx\right|^{-2}$. Therefore, in case $\bF=\bzero$,
the Navier-Stokes equations \eqref{compatibility-ns-force} in $\mathbb{R}^{3}$
can be solved for small $\bff$ by a fixed point argument in a space
of function decay faster than $\left|\bx\right|^{-1}$ in which the
Stokes operator is well-posed. In case $\bF\neq\bzero$, one needs
a two-parameters family of explicit solutions that lifts the compatibility
condition $\bF=\bzero$ and makes the Stokes operator well-posed.
This family of explicit solution was found by \citet{Landau-newexactsolution1944}.
For any $\bF\in\mathbb{R}^{2}$, the Landau solution $\left(\bUF,P_{\bF}\right)$
is an exact and explicit solution of \eqref{compatibility-ns-force}
in $\mathbb{R}^{3}$ with $\bff(\bx)=\bF\delta^{3}(\bx)$, so having
a net force $\bF$. By defining $\bu=\bUF+\bv$ and $p=P_{\bF}+q$
the Landau solution lifts the compatibility condition: the Navier-Stokes
equations \eqref{compatibility-ns-force} become
\begin{align}
\Delta\bv-\bnabla q & =\boldsymbol{g}\,, & \bnabla\bcdot\bu & =0\,, & \lim_{|\bx|\to\infty}\bu & =\bzero\,,\label{eq:compatibility-ns-landau}
\end{align}
where now the source term
\[
\boldsymbol{g}=\bUF\bcdot\bnabla\bv+\bv\bcdot\bnabla\bUF+\bv\bcdot\bnabla\bv+\bff-\bF\delta^{3}\,,
\]
has zero mean. Since $\bUF$ is bounded by $\left|\bx\right|^{-1}$,
if $\bv$ is bounded by $\left|\bx\right|^{-2+\varepsilon}$, for
some $\varepsilon>0$, then by power counting, $\boldsymbol{g}$ decays
like $\left|\bx\right|^{-4+\varepsilon}$, so that the solution $\bv$
of this Stokes system is bounded by $\left|\bx\right|^{-2+\varepsilon}$.
This formal argument indicates that we can perform a fixed point argument
to show the existence of solutions satisfying
\begin{align*}
\bu & =\bUF+O(\left|\bx\right|^{-2+\varepsilon})\,, & p & =P_{F}+O(\left|\bx\right|^{-3+\varepsilon})\,.
\end{align*}
provided $\bff$ to be small enough. These formal considerations were
made rigorous by \citet{Korolev.Sverak-largedistanceasymptotics2011}.
There are two crucial points that make this idea to work. First the
Landau solutions decay like $\left|\bx\right|^{-1}$, so that the
term $\bUF\bcdot\bnabla\bv+\bv\bcdot\bnabla\bUF$ can be put with
the nonlinearity. Second, the compatibility condition which is the
mean of $\boldsymbol{g}$ does not depend on $\bv$, so that the lift
parameter $\bF$ can be taken as the mean value of $\bff$ from the
beginning and does not require an adaptation at each fixed point iteration.
The second property comes from the fact that the compatibility condition
is the net force which is an invariant quantity (see \propref{invariants}),
so
\[
\bF=\int_{\mathbb{R}^{2}}\bff=\lim_{R\to\infty}\int_{\partial B(\bzero,R)}\mathbf{T}\bn\,,
\]
where $\mathbf{T}$ is the stress tensor with the convective term
\eqref{stress-tensor}. Therefore, the net force $\bF$ depends only
on the asymptotic behavior of the solution, \emph{i.e.} on the Landau
solution and not on $\bv$.

The aim of this chapter is to determine an approximate solution of
the Navier-Stokes equations in two dimensions, which becomes more
and more accurate at large distances and might describe the general
asymptotic behavior of a solution of the two-dimensional Navier-Stokes
equations; in other words, the two-dimensional analog of the Landau
solution.

\section[Homogeneous asymptotic behavior for a nonzero
net force]{\label{sec:asy-euler}Homogeneous asymptotic behavior for a nonzero
net force%
\footnote{The explicit solution of the Euler equations presented here was brought to my attention by Matthieu Hillairet and to my knowledge was never published.
}}

If $\bF\neq\bzero$, the two-dimensional equations are, as already
said, in a supercritical regime, and the aim is to determine the asymptotic
behavior carrying the net force, as the \citet{Landau-newexactsolution1944}
solution does in three dimensions. By the previous power counting
argument, the net force cannot be generated by solutions decaying
faster than $\left|\bx\right|^{-1/2}$. However, if we make an ansatz
such that the velocity decays like $\left|\bx\right|^{-1/2}$ in all
directions, then $\bu$ has to be asymptotically a solution of the
stationary Euler equations
\begin{align}
\bu\bcdot\bnabla\bu+\bnabla p+\bff & =\bzero\,, & \bnabla\bcdot\bu & =0\,,\label{eq:intro-euler}
\end{align}
at large distances. Explicitly, if one takes the following ansatz
for the stream function,
\[
\psi_{0}(r,\theta)=r^{1/2}\varphi_{0}(\theta)\,,
\]
then
\begin{align}
\bu_{0} & =\frac{1}{2r^{1/2}}\left[-2\varphi_{0}^{\prime}(\theta)\,\be_{r}+\varphi_{0}(\theta)\be_{\theta}\right]\,, & p_{0} & =\frac{-A^{2}}{4r}\,,\label{eq:euler-sol}
\end{align}
is an exact solution of the Euler\index{Exact solutions!Euler equations}
equation \eqref{intro-euler} in $\mathbb{R}^{2}\setminus\left\{ \bzero\right\} $
provided $\varphi_{0}$ is a $2\pi$-periodic solution of the ordinary
differential equation
\[
2\varphi_{0}\varphi_{0}^{\prime\prime}+2\left(\varphi_{0}^{\prime}\right)^{2}+\varphi_{0}^{2}=A^{2}\,,
\]
for some $A\in\mathbb{R}$. The $2\pi$-periodic solutions of this
equation are given by
\[
\varphi_{0}(\theta)=A\sqrt{1-\lambda\cos(\theta-\theta_{0})}\,,
\]
with $A\in\mathbb{R}$, $\left|\lambda\right|<1$, and $\theta_{0}\in\mathbb{R}$.
Moreover, this is an exact solution of \eqref{intro-euler} in $\mathbb{R}^{2}$
in the sense of distributions with $\bff(\bx)=\bF\delta^{2}(\bx)$,
where
\[
\bF=\pi^{2}A^{2}\frac{1-\sqrt{1-\lambda^{2}}}{\lambda}\left(\cos\theta_{0},\sin\theta_{0}\right)\,.
\]
This exact solution therefore seems to be a very good candidate for
the asymptotic behavior of the two-dimensional Navier-Stokes equations
with a nonzero net force. However, this exact solution of the Euler
equations is very far from the asymptote that we observed in numerical
simulations, as shown later on. A mathematical explanation why this
cannot be the asymptotic behavior of the Navier-Stokes equations at
least for small data comes from the next order of the asymptotic expansion.

To analyze the possibility that the exact solution \eqref{euler-sol}
is the asymptote at large distances of a solution of the Navier-Stokes
equations \eqref{intro-ns-eq}, the idea is to determine a formal
asymptotic expansion for large values of $r$ which starts with the
leading term $\bu_{0}$. The idea of the asymptotic expansion is to
look at the solution in the form
\begin{align}
\bUF & =\sum_{i=0}^{n}\bu_{i}\,, & P_{\bF} & =\sum_{i=0}^{n}p_{i}\,,\label{eq:euler-exp-u}
\end{align}
with $\bu_{i}=O(r^{-(i+1)/2})$ and $p_{i}=O(r^{-(i+2)/2})$ such
that \eqref{euler-exp-u} is a solution of the Navier-Stokes equations
with a remainder $\bff=O(r^{-(5+i)/2})$ for some $n\geq0$. The case
$n=0$ is trivial because if $\bff=\Delta\bu_{0}=O(r^{-5/2})$, $\left(\bu_{0},p_{0}\right)$
is a solution of \eqref{intro-ns-eq}. We now consider the next order,
\emph{i.e.} $n=1$, and we choose the following Ansatz, 
\begin{align*}
\bu_{1}(r,\theta) & =\frac{1}{r}\left[-\varphi_{1}(\theta)\,\be_{r}+\mu\be_{\theta}\right]\,, & p_{1} & =\frac{\varrho_{1}(\theta)}{r^{3/2}}\,,
\end{align*}
where $\varphi_{1}$ and $\varrho_{1}$ are $2\pi$-periodic functions
we have to determine. By explicit calculations, we obtain that $\left(\bu_{0}+\bu_{1},p_{0}+p_{1}\right)$
is a solution of \eqref{intro-ns-eq} with some $\bff=O(r^{-3})$
only if $\varphi_{1}$ satisfies the following differential equation
\begin{equation}
\frac{4}{3}\left(\varphi_{0}^{4}\varphi_{1}^{\prime}\right)^{\prime}+\left(\varphi_{0}+4\varphi_{0}^{\prime\prime}\right)\varphi_{0}^{3}\varphi_{1}=R\,,\label{eq:euler-ode2}
\end{equation}
where
\[
R=\frac{\varphi_{0}^{3}}{6}\left[16\varphi_{0}^{(4)}-16\mu\varphi_{0}^{(3)}+40\varphi_{0}^{\prime\prime}-4\mu\varphi_{0}^{\prime}+9\varphi_{0}\right]\,.
\]
By an explicit calculation, we find
\[
\left(\varphi+4\varphi^{\prime\prime}\right)\varphi^{3}=A^{4}\left(1-\lambda^{2}\right)\,,
\]
so by integrating \eqref{euler-ode2} over a period, we obtain
\[
\int_{0}^{2\pi}\varphi_{1}(\theta)\rd\theta=\frac{1}{A^{4}\left(1-\lambda^{2}\right)}\int_{0}^{2\pi}R(\theta)\rd\theta=\frac{3\pi}{\sqrt{1-\lambda^{2}}}\,,
\]
where in the last step we used the explicit form of $\varphi$ to
integrate $R$. Therefore, the net flux carried by $\bUF=\bu_{0}+\bu_{1}$
is
\[
\Phi=\int_{S^{1}}\bUF\bcdot\bn=-\int_{0}^{2\pi}\left(\varphi_{0}^{\prime}(\theta)+\varphi_{1}(\theta)\right)\rd\theta=\frac{-3\pi}{\sqrt{1-\lambda^{2}}}\,,
\]
Since $\Phi\leq-3\pi$ independently of $A$, we conclude that $\left(\bUF,P_{\bF}\right)$
constructed in \eqref{euler-exp-u} cannot be the asymptotic behavior
of a solution of the Navier-Stokes equations at least for small data.
We remark, that for $\lambda=0$, then $\Phi=-3\pi$ and $\bUF=\bu_{0}+\bu_{1}$
is an exact solution of the Navier-Stokes equations in $\mathbb{R}^{2}\setminus\left\{ \bzero\right\} $
with $\bff=\bzero$ which was found by \citet[\S 11]{Hamel-SpiralfoermigeBewegungen1917},
\[
\bUF=\frac{-3}{2r}\be_{r}+\left(\frac{A}{2r^{1/2}}+\frac{\mu}{r}\right)\be_{\theta}\,.
\]
Another interpretation of this solution in terms of symmetries has
been given by \citet[\S 3]{Guillod-Generalizedscaleinvariant2015}.

\section{\label{sec:asy-wake}Inhomogeneous asymptotic behavior for a nonzero
net force\index{Asymptotic expansion!Navier-Stokes solutions!for bFneqbzero
@for $\protect\bF\protect\neq\protect\bzero$}\index{Navier-Stokes equations!asymptotic expansion!for bFneqbzero
@for $\protect\bF\protect\neq\protect\bzero$}}

In order to determine the asymptotic behavior of the solutions of
the Navier-Stokes equations for small data and $\bF\neq\bzero$, the
idea is to modify the homogeneous power counting introduced in \eqref{decay}
by introducing a preferred direction so that the equations become
almost critical at larges distances in a sense explained later. We
consider $D\subset\mathbb{C}$ defined by
\[
D=\left\{ \left(r\cos\theta,r\sin\theta\right)\,,r>0\text{ and }\theta\in(-\pi,\pi)\right\} \,,
\]
and the following change of coordinates $D\to D^{p}$, $z\mapsto\bar{z}=z^{p}$
for $0<p<1$, represented in \figref{conformal-map}. Explicitly,
the change of coordinates is given by
\begin{align*}
\bar{x}_{1} & =r^{p}\cos(p\theta)\,, & \bar{x}_{2} & =r^{p}\sin(p\theta)\,,
\end{align*}
and the scale factors are
\[
h_{1}=h_{2}=\frac{r^{1-p}}{p}=\frac{\left|\bar{\bx}\right|^{1/p-1}}{p}\,.
\]
\begin{figure}[h]
\includefigure{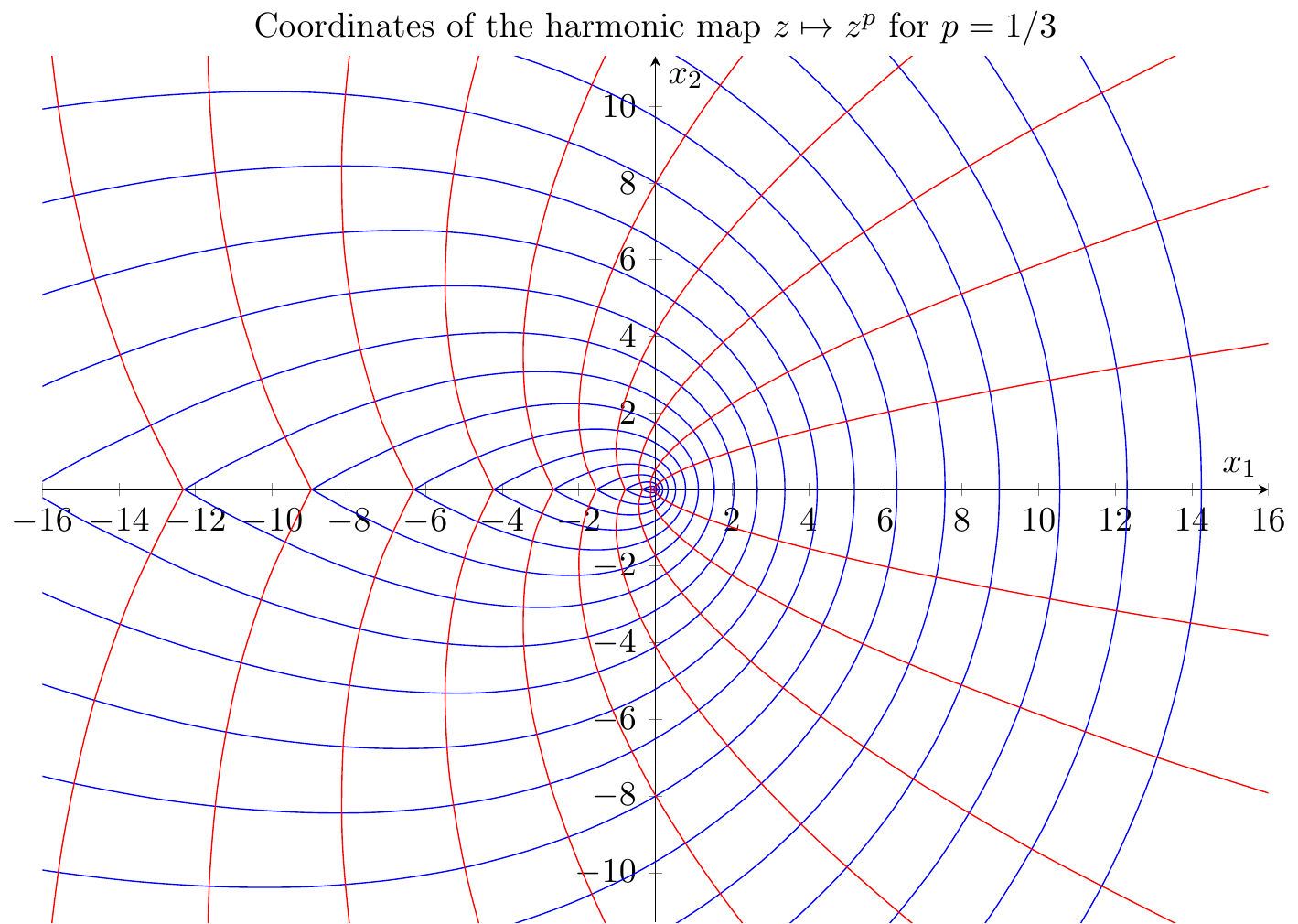}

\caption{\label{fig:conformal-map}Change of coordinates induced by the conformal
map $z\protect\mapsto z^{p}$ for $p=1/3$. The red lines corresponds
to constant values of $\bar{x}_{2}$ for $\bar{x}_{1}>\cot(p\pi)\left|\bar{x}_{2}\right|$
and the blue lines to constant values of $\bar{x}_{1}$ for $\left|\bar{x}_{2}\right|<\tan(p\pi)\bar{x}_{1}$.}
\end{figure}

The idea is now to look at large values of $\bar{x}_{1}$ with $\bar{x}_{2}$
fixed, so the scaling is as follows
\begin{align*}
\frac{\partial}{\partial\bar{x}_{1}} & \sim\bar{x}_{1}^{-1}\,, & \frac{\partial}{\partial\bar{x}_{2}} & \sim1\,,
\end{align*}
and therefore if the stream function grows like $\bar{x}_{1}^{1/p-\alpha-1}$
at fixed $\bar{x}_{2}$ for some $\alpha\geq0$, we have
\begin{align*}
\bu & \sim\begin{pmatrix}\bar{x}_{1}^{-\alpha} & \bar{x}_{1}^{-\alpha-1}\end{pmatrix}\,, & \bnabla\bu & \sim\begin{pmatrix}\bar{x}_{1}^{-\alpha-1/p} & \bar{x}_{1}^{-\alpha-1/p-1}\\
\bar{x}_{1}^{-\alpha-1/p+1} & \bar{x}_{1}^{-\alpha-1/p}
\end{pmatrix}\,,
\end{align*}
in the new basis $\left\{ \bar{\be}_{1},\bar{\be}_{2}\right\} $.
The laplacian is
\[
\Delta\bu\sim\begin{pmatrix}\bar{x}_{1}^{-\alpha-2/p+2} & \bar{x}_{1}^{-\alpha-2/p+1}\end{pmatrix}\,,
\]
and
\[
\bu\bcdot\bnabla\bu\sim\begin{pmatrix}\bar{x}_{1}^{-2\alpha-1/p} & \bar{x}_{1}^{-2\alpha-1/p-1}\end{pmatrix}\,,
\]
so with respect to this scaling the Navier-Stokes equations are critical
for $\alpha=1/p-2$. Moreover, the decay of the pressure that is compatible
is given by $p\sim\bar{x}_{1}^{-2\alpha-2}$. In these coordinates,
the net force is given by
\[
\bF=\lim_{\bar{x}_{1}\to\infty}\int_{-\tan(p\pi)\bar{x}_{1}}^{+\tan(p\pi)\bar{x}_{1}}\mathbf{T}\bcdot\bar{\be}_{1}h\rd\bar{x}_{2}\,.
\]
The stress tensor \eqref{stress-tensor} behaves like $\mathbf{T}\sim\bar{x}_{1}^{-2\alpha}$,
so by assuming that $\mathbf{T}$ decays fast enough in $\bar{x}_{2}$,
we obtain that $\bF=\bzero$ if $\alpha>\frac{1}{2}-\frac{1}{2p}$.
Therefore, the critical decay to obtain a nonzero net force is $\alpha=\frac{1}{2}-\frac{1}{2p}$
and if moreover we impose that the Navier-Stokes equations are critical,
\emph{i.e.} $\alpha=1/p-2$, we obtain the following result
\begin{align*}
\alpha & =1\,, & p & =\frac{1}{3}\,.
\end{align*}
As in the case of the homogeneous decay, we consider the following
ansatz for the stream function,
\begin{equation}
\psi_{0}(\bar{x}_{1},\bar{x}_{2})=\bar{x}_{1}\varphi_{0}(\bar{x}_{2})\,,\label{eq:wake-psi0}
\end{equation}
so we have
\begin{align}
\bu_{0} & =\frac{1}{3\left|\bar{\bx}\right|^{2}}\left[-\bar{x}_{1}\varphi_{0}^{\prime}(\bar{x}_{2})\,\bar{\be}_{1}+\varphi_{0}(\bar{x}_{2})\bar{\be}_{2}\right]\,, & p_{0} & =\frac{\rho_{0}(\bar{x}_{2})}{\bar{x}_{1}^{4}}\,.\label{eq:wake-ansatz}
\end{align}
By plugging \eqref{wake-ansatz} into the Navier-Stokes equations
\eqref{intro-ns-eq}, we obtain
\begin{align*}
\bff\bcdot\bar{\be}_{1} & =\frac{1}{27\bar{x}_{1}^{5}}\left(-\varphi_{0}^{(3)}(\bar{x}_{2})+\varphi_{0}(\bar{x}_{2})\varphi_{0}^{\prime\prime}(\bar{x}_{2})+\varphi_{0}^{\prime}(\bar{x}_{2})^{2}+O(\bar{x}_{1}^{-1})\right)\,,\\
\bff\bcdot\bar{\be}_{2} & =\frac{1}{27\bar{x}_{1}^{6}}\left(-3\varphi_{0}^{\prime\prime}(\bar{x}_{2})+2\bar{x}_{2}\varphi_{0}^{\prime}(\bar{x}_{2})^{2}-\varphi_{0}(\bar{x}_{2})\varphi_{0}^{\prime}(\bar{x}_{2})-9\rho_{0}^{\prime}(\bar{x}_{2})+O(\bar{x}_{1}^{-1})\right)\,,
\end{align*}
and by setting
\begin{align*}
\varphi_{0}(\bar{x}_{2}) & =-2a\tanh(a\bar{x}_{2})\,,\\
\rho_{0}(\bar{x}_{2}) & =\frac{4a^{2}}{27}\Big[4a\bar{x}_{2}\tanh(a\bar{x}_{2})-4\log\left(2\cosh(a\bar{x}_{2})\right)\\
 & \phantom{=\frac{4a^{2}}{27}\Big[}+\left(2a\bar{x}_{2}\tanh(a\bar{x}_{2})+7\text{sech}^{2}(a\bar{x}_{2})\right)\text{sech}(a\bar{x}_{2})\Big]\,,
\end{align*}
where $a>0$, we obtain that \eqref{wake-ansatz} is an exact solution
of the Navier-Stokes equations in $D$ with some $\bff=O(\bar{x}_{1}^{-6})\bar{\be}_{1}+O(\bar{x}_{1}^{-7})\bar{\be}_{2}=\left(O(r^{-6/3}),O(r^{-7/3})\right)$.
By an explicit calculation, the stress tensor including the convective
term is
\[
\mathbf{T}_{0}=\frac{-\varphi_{0}^{\prime}(\bar{x}_{2})^{2}}{9\bar{x}_{1}^{2}}\bar{\be}_{1}\otimes\bar{\be}_{1}+O(\bar{x}_{1}^{-3})\,,
\]
so the net force is then given by
\[
\bF=\lim_{\bar{x}_{1}\to\infty}\int_{-\tan(p\pi)\bar{x}_{1}}^{+\tan(p\pi)\bar{x}_{1}}\mathbf{T}\bcdot\bar{\be}_{1}h\rd\bar{x}_{2}=\int_{-\infty}^{+\infty}\left(\frac{-1}{3}\varphi_{0}^{\prime}(\bar{x}_{2})^{2},0\right)=\left(-\frac{16a^{3}}{9},0\right)\,.
\]
However, the stream function \eqref{wake-psi0} when expressed back
in the coordinates $(x_{1},x_{2})$ is not continuous along the line
$\left\{ (x_{1},0)\,,x_{1}<0\right\} $, there is a jump of order
$O(\left|x_{1}\right|^{1/3})$. This jump will be removed at the next
order.

The role of the next order is to improve the decay of the remainder
$\bff$, so we make the Ansatz,
\[
\psi_{1}(\bar{x}_{1},\bar{x}_{2})=\varphi_{1}(\bar{x}_{2})\,,
\]
in order to cancel the term decaying like $O(\bar{x}_{1}^{-6})\bar{\be}_{1}+O(\bar{x}_{1}^{-7})\bar{\be}_{2}$
in the remainder of the previous order. We have 
\begin{align*}
\bu_{1} & =\frac{-1}{3\left|\bar{\bx}\right|^{2}}\varphi_{1}^{\prime}(\bar{x}_{2})\,\bar{\be}_{1}\,, & p_{1} & =\frac{\rho_{1}(\bar{x}_{2})}{\bar{x}_{1}^{5}}\,.
\end{align*}
By plugging $\bu=\bu_{0}+\bu_{1}$ and $p=p_{0}+p_{1}$ into the Navier-Stokes
equations \eqref{intro-ns-eq}, we obtain
\begin{align*}
\bff\bcdot\bar{\be}_{1} & =\frac{1}{27\bar{x}_{1}^{6}}\left(-\varphi_{1}^{(3)}(\bar{x}_{2})+\varphi_{0}(\bar{x}_{2})\varphi_{1}^{\prime\prime}(\bar{x}_{2})+3\varphi_{0}^{\prime}(\bar{x}_{2})\varphi_{1}^{\prime}(\bar{x}_{2})+O(\bar{x}_{1}^{-1})\right)\,,\\
\bff\bcdot\bar{\be}_{2} & =\frac{1}{27\bar{x}_{1}^{6}}\left(-4\varphi_{1}^{\prime\prime}(\bar{x}_{2})+4\bar{x}_{2}\varphi_{0}^{\prime}(\bar{x}_{2})\varphi_{1}^{\prime}(\bar{x}_{2})-9\rho_{1}^{\prime}(\bar{x}_{2})+O(\bar{x}_{1}^{-1})\right)\,.
\end{align*}
So by setting
\begin{align*}
\varphi_{1}(\bar{x}_{2}) & =\sqrt{3}\left(\frac{2a\bar{x}_{2}}{3}-\tanh(a\bar{x}_{2})-a\bar{x}_{2}\sech^{2}(a\bar{x}_{2})\right)\,,\\
\rho_{1}(\bar{x}_{2}) & =\frac{2\sqrt{3}a}{27}\sech^{4}(a\bar{x}_{2})\left(6a^{2}\bar{x}_{2}^{2}-4az\sinh(2a\bar{x}_{2})+7\cosh(2a\bar{x}_{2})+7\right)\,,
\end{align*}
we obtain that $\bu=\bu_{0}+\bu_{1}$ and $p=p_{0}+p_{1}$ is an exact
solution of the Navier-Stokes equations \eqref{intro-ns} in $D$
with some $\bff=O(\bar{x}_{1}^{-7})\bar{\be}_{1}+O(\bar{x}_{1}^{-8})\bar{\be}_{2}=\left(O(r^{-7/3}),O(r^{-8/3})\right)$.
Moreover, the jump in the stream function $\psi=\psi_{0}+\psi_{1}$
on the line $\left\{ (x_{1},0)\,,x_{1}<0\right\} $ is now uniformly
bounded.

Therefore, we obtained the following result:
\begin{prop}
\label{prop:wake-UF-PF}For any $\bF\neq0$, there exists a solution
$(\bUF,P_{\bF})\in C^{\infty}(\mathbb{R}^{2})$ with some $\bff\in C^{\infty}(\mathbb{R}^{2})$
of the Navier-Stokes equations in $\mathbb{R}^{2}$ with $\bUF=O(\left|\bx\right|^{-1/3})$,
$P_{\bF}=O(\left|\bx\right|^{-2/3})$ and $\bff=\left(O(\left|\bx\right|^{-7/3}),O(\left|\bx\right|^{-8/3})\right)$.\end{prop}
\begin{proof}
By adding the term $\frac{3\sqrt{3}}{\pi}\arg(\bar{x}_{1}+\i\bar{x}_{2})$
to the stream function $\psi_{0}+\psi_{1}$ and also terms decaying
faster at infinity, we can construct a smooth stream function $\psi$
which generates a solution $(\bUF,P_{\bF})\in C^{\infty}(\mathbb{R}^{2})$
of Navier-Stokes equations \eqref{intro-ns} in $\mathbb{R}^{2}$
with $\bUF=O(\left|\bx\right|^{-1/3})$, $P_{\bF}=O(\left|\bx\right|^{-2/3})$
and some $\bff=\left(O(r^{-7/3}),O(r^{-8/3})\right)$ such that $\bF=\left(-\frac{16a^{3}}{9},0\right)$.
Since the equations are rotational invariant, we can rotate this solution
to obtain any $\bF\neq\bzero$.
\end{proof}
The solution $(\bUF,P_{\bF})$ is represented in \figref{wake-ansatz}
for $\bF=(-F,0)$ with some $F>0$. Within a wake the velocity field
decays like $\left|\bx\right|^{-1/3}$ whereas outside the wake it
decays like $\left|\bx\right|^{-2/3}$. The width of the wake is decreasing
as the net force is increasing. Moreover, we believe that this solution
describes the general asymptote of any solutions with small enough
$\bff$ or $\bu^{*}$ having a nonzero net force $\bF\neq\bzero$:
\begin{conjecture}
\label{conj:wake-conj}For a large class of boundary conditions $\bu^{*}$
and source terms $\bff$ with a nonzero net force $\bF$, there exists
a solution to \eqref{intro-ns} with $\bu_{\infty}=\bzero$ which
satisfies
\begin{align*}
\bu & =\bUF+O(r^{-1})\,, & p & =P_{\bF}+O(r^{-2})\,,
\end{align*}
where $\left(\bUF,P_{\bF}\right)$ is the solution constructed in
\propref{wake-UF-PF}.
\end{conjecture}
Once the asymptotic behavior is determined, the idea to prove its
validity is to lift the compatibility conditions by using the asymptotic
behavior, as the Landau solution does in three dimensions. However,
due to the decay in $\left|\bx\right|^{-1/3}$ of the asymptote, instead
of \eqref{compatibility-ns-landau}, we have to consider the linear
problem
\begin{align*}
\Delta\bv-\bnabla q-\bUF\bcdot\bnabla\bv-\bv\bcdot\bnabla\bUF & =\boldsymbol{g}\,, & \bnabla\bcdot & \bu=0\,,
\end{align*}
where $\boldsymbol{g}$ is a given source term. To our knowledge,
this linear problem is not solvable with the mathematical methods
developed so far. The reason is the following: in view of the regularity,
one has to inverse the whole Laplacian $\Delta$, which is the operator
of highest degree, otherwise we loose regularity and in view of the
decay at infinity, one has to inverse
\[
\left(\bF\bcdot\bnabla\right)^{2}\bv-\bUF\bcdot\bnabla\bv-\left(\bv\bcdot\bF\right)\,\bF\bcdot\bnabla\bUF\,,
\]
which leads to regularity lost in the direction $\bF^{\perp}$. Therefore,
in order to solve this linear problem, one has to face with these
two opposing principles. This will be part of further investigations.
However, the validity of the conjecture as well as the other asymptotic
regimes for the case of a vanishing net force will be investigated
numerically in the next sections.
\begin{figure}[h]
\includefigure{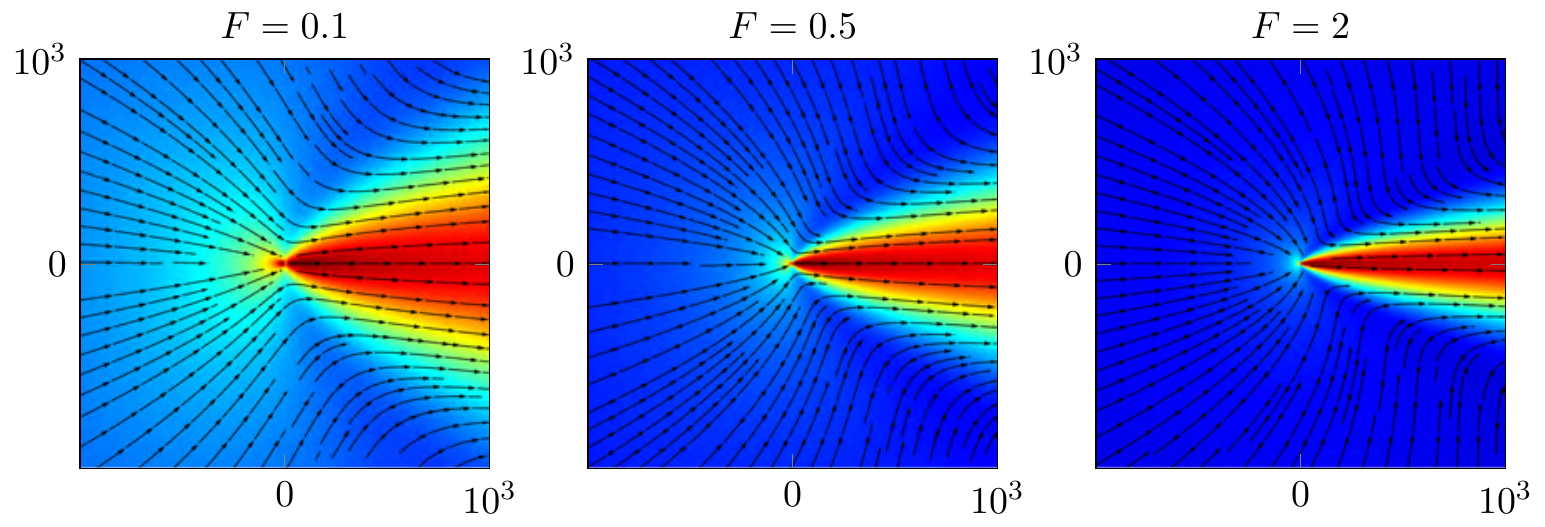}

\caption{\label{fig:wake-ansatz}Velocity field $\protect\bUF$ of \propref{wake-UF-PF}
for $\protect\bF=(-F,0)$ with different values of $F>0$. The color
represent the magnitude of $\left|\protect\bx\right|^{1/3}\protect\bUF$
in order to highlight, the fact the $\protect\bUF$ decays like $\left|\protect\bx\right|^{-1/3}$
inside a wake and like $\left|\protect\bx\right|^{-2/3}$ outside.}
\end{figure}

\section{\label{sec:wake-sim}Numerical simulations with Stokes solutions
as boundary conditions\index{Numerical simulations!with Stokes boundary conditions}\index{Numerical simulations!double straight wake}}

In an attempt to determine the general asymptotic behavior numerically,
we consider the Navier-Stokes equations \eqref{intro-ns} in the domain
$\Omega=\mathbb{R}^{2}\setminus\widebar B$ where $B=B(\bzero,1)$.
In view of \secref{failure-asymptotic} and \thmref{strong-under-compatibility},
we have seen that the problematic solutions of the Stokes equations
in order to construct a solutions of the Navier-Stokes equations are
the asymptotic term $\bS_{0}$ and $\bS_{1}$ given in \lemref{lin-stokes-asy}.
Therefore, the idea is to take for the boundary condition on $\partial B$,
the evaluation of the problematic asymptotic terms,
\[
\bu^{*}=\left.\bS_{0}+\bS_{1}\right|_{\partial B}=\bC_{0}\bcdot\mathbf{E}_{0}+\bC_{1}\bcdot\mathbf{E}_{1}\,,
\]
where $\bC_{0}\in\mathbb{R}^{2}$ and $\bC_{1}\in\mathbb{R}^{3}$
are parameters. Explicitly, by using \lemref{lin-stokes-asy}, we
have
\begin{equation}
\bu^{*}=\frac{-1}{4\pi}\left[\bC_{0}\bcdot\left(\cos\theta,\sin\theta\right)\,\be_{r}+\bC_{1}\bcdot\left(\cos(2\theta)\be_{r},\sin(2\theta)\be_{r},\be_{\theta}\right)\right]\,.\label{eq:plot-bc}
\end{equation}
The different boundary conditions are represented in \figref{plot-bc}.
Trivially the solution of the Stokes equations \eqref{compatibility-stokes}
satisfying this boundary condition grows at infinity like $\log\left|\bx\right|$
unless $\bC_{0}=\bzero$. We will see numerically that the solution
of the Navier-Stokes equations subject to the same boundary condition
will decay like $\left|\bx\right|^{-1/3}$ or faster. In order to
simulate this problem, we truncate the domain $\Omega$ to a ball
$B(\bzero,R)$ of radius $R=10^{5}$, and put open boundary conditions
on the artificial boundary $\partial B(\bzero,R)$. We make simulations
for various choices of the parameters $\bC_{0}$ and $\bC_{1}$. In
order to systematically analyze the solutions, we determine numerically
for each value of the parameters, the functions
\begin{align}
d(r) & =\max_{\theta\in[-\pi,\pi]}\left|\bu(r,\theta)\right|\,, & a(r) & =\argmax_{\theta\in[-\pi,\pi]}\left|\bu(r,\theta)\right|\,.\label{eq:functions-d-a}
\end{align}
In view of the symmetry of the boundary condition a nonzero net force
can be generated only if $\bC_{0}\neq\bzero$. 
\begin{figure}[h]
\includefigure{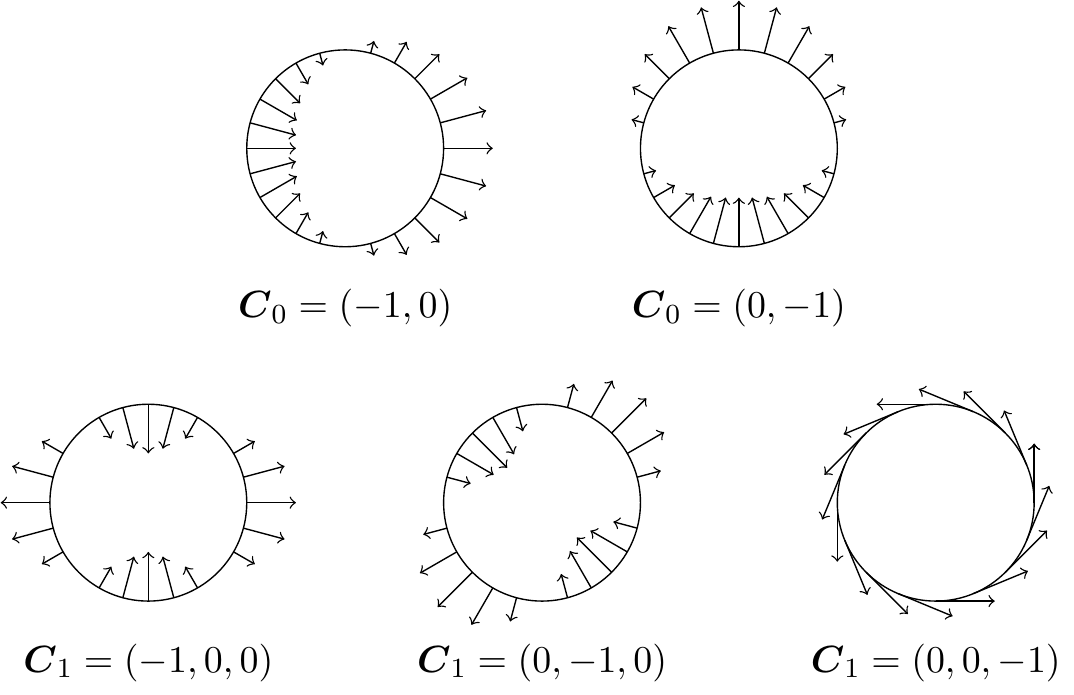}

\caption{\label{fig:plot-bc}Representation of the vector field $\protect\bu^{*}$
given by \eqref{plot-bc} with $\protect\bC_{1}=\protect\bzero$ for
the first line and $\protect\bC_{0}=\protect\bzero$ for the second
one.}
\end{figure}

\subsection{Nonzero net force}

First, we consider the case $\bC_{0}\neq\bzero$ which might generate
a nonzero net force. Without lost of generality, we can perform a
rotation such that $\bC_{0}=\left(-\mathcal{F},0\right)$ with $\mathcal{F}>0$.
In order to keep only two free parameters, we choose $\bC_{1}=\left(0,0,-\mathcal{M}\right)$
with $\mathcal{M}>0$. We perform simulations for $\mathcal{F}\in\left\{ 0,0.08,0.16,\dots,18\right\} \pi$
and $\mathcal{M}\in\left\{ 0,0.08,0.16,\dots,36\right\} \pi$. Since
the problem is highly nonlinear, we used a parametric solver in order
to follow the evolution of the solution starting from $\mathcal{F}=\mathcal{M}=0$.
Even with this parametric solver, the nonlinear solver fails to converge
for $\mathcal{F}\geq8\pi$ and $\mathcal{M}\leq28\pi$ approximately;
more precisely on the blank region of \figref{plot-Fmu}.

The velocity magnitude is represented in \figref{sim-Fmu-5,sim-Fmu-20}
respectively on the line $\mathcal{F}=2\pi$ and $\mathcal{F}=8\pi$
for varying values of $\mathcal{A}$. At $\mathcal{A}=0$, the velocity
field presents a wake behavior with a decay like $r^{-1/3}$ along
the first axis. The opening of the wake depends on $\mathcal{F}$.
As $\mathcal{A}$ is increasing, the orientation of the wake in varying
and when $\mathcal{A}$ is big enough, the wake behavior becomes blurred
and the solution has an homogeneous decay like $r^{-1}$. It is an
interesting result, that we observe a kind of phase transition between
a decay like $r^{-1/3}$ and a decay like $r^{-1}$. This is expected,
since for $\mathcal{M}>16\sqrt{3}\pi$ and $\mathcal{F}$ small enough,
\citet{Hillairet-mu2013} proved that the asymptote is given by $\mu\be_{\theta}/r$
for some $\mu>0$. 

We then make a more systematic analysis. In the region characterized
by $3\times10^{2}\leq r\leq3\times10^{4}$, the function $d$ seems
to be already in the asymptotic regime and not influenced by the artificial
boundary condition. We use this region to determine the power of decay
of the function $d$, which is represented in \figref{plot-Fmu}a
in terms of $\mathcal{F}$ and $\mathcal{M}$. We then analyze the
function $a$, by showing in \figref{plot-Fmu}b its mean value over
$3\times10^{2}\leq r\leq3\times10^{4}$. In order to determine if
this mean value is accurate or not, we compute the standard deviation
of $a$ and represent large standard deviations as more transparent
colors. At fixed value of $\mathcal{F}$ the angle is increasing with
$\mathcal{M}$ until the power of decay becomes almost $r^{-1}$.
We compute the net force and net torque acting on the body,
\begin{align*}
\bF & =\int_{\partial B}\mathbf{T}\bn\,, & M & =\int_{\partial B}\bx\bwedge\mathbf{T}\bn\,.
\end{align*}
The magnitude of the net force $\bF$ is shown in \figref{plot-Fmu}c
and its angle in \figref{plot-Fmu}d. If the net force is too small,
the angle is ill-defined, so we add more transparency to smallest
net forces. As expected the net force is zero in the region where
the power of decay is $r^{-1}$ and is increasing with $\mathcal{F}$
in the other region. In \figref{plot-Fmu}e, we represent the net
torque $M$ which increases almost linearly as a function of $\mathcal{M}$
and is independent of $\mathcal{F}$. Finally, in \figref{plot-Fmu}f,
we represent the difference between the angle of the net force and
the angle corresponding to the slowest decay. The two angles almost
coincide in the region where the angles are well-defined, \emph{i.e.}
when the net force is not too small and when the power of decay is
$r^{-1/3}$.

Moreover, one can show \citep[see][]{Guillod-Asymptoticbehaviour2013}
that the numerical solutions verify \conjref{wake-conj}, \emph{i.e.}
its asymptotic behavior is given by $\bUF$.

\begin{figure}[h]
\includefigure{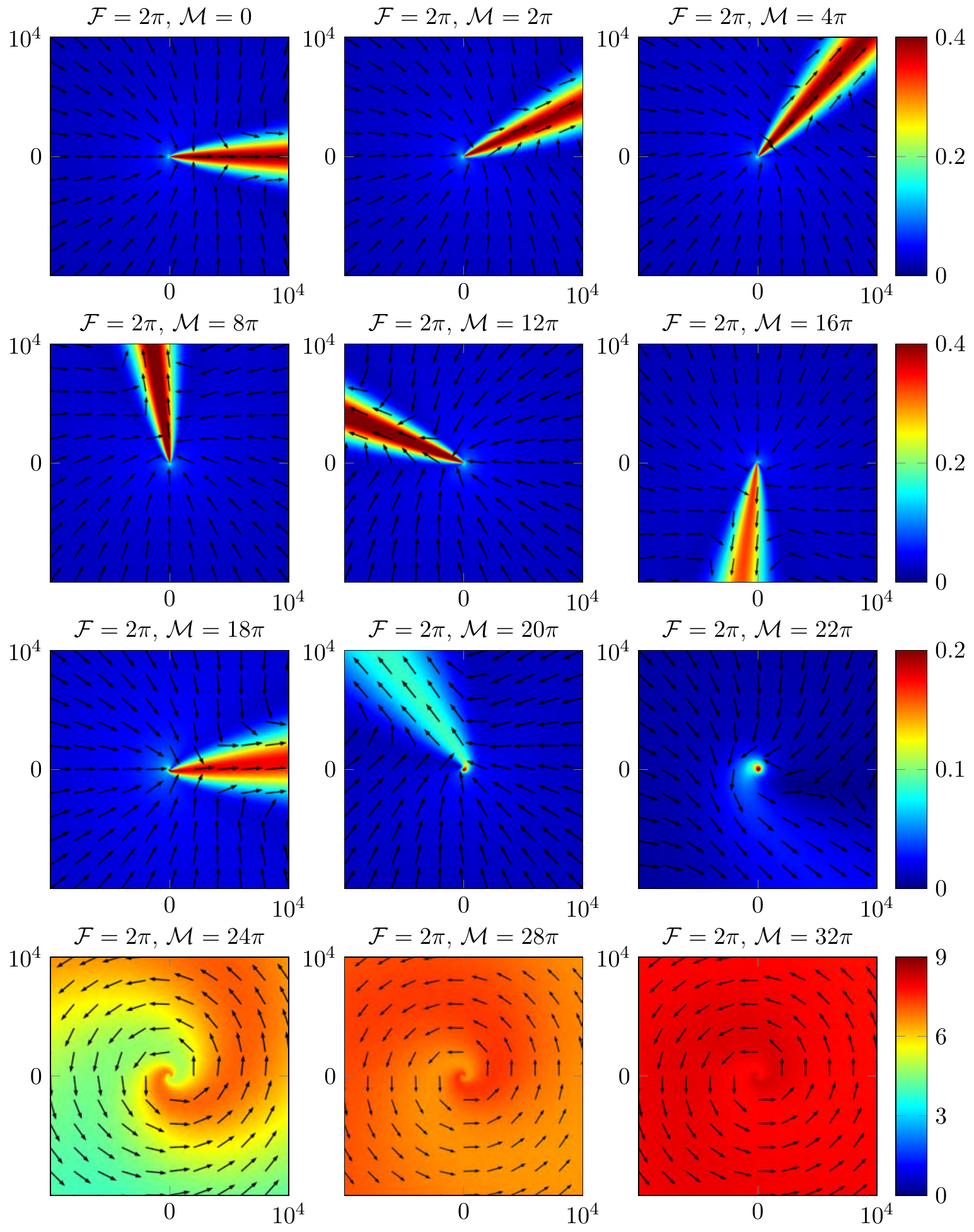}

\caption{\label{fig:sim-Fmu-5}Numerical simulations on the line $\mathcal{F}=2\pi$
of the velocity magnitude multiplied by $r^{1/3}$ for the first three
lines and by $r$ for the last one.}
\end{figure}
\begin{figure}[h]
\includefigure{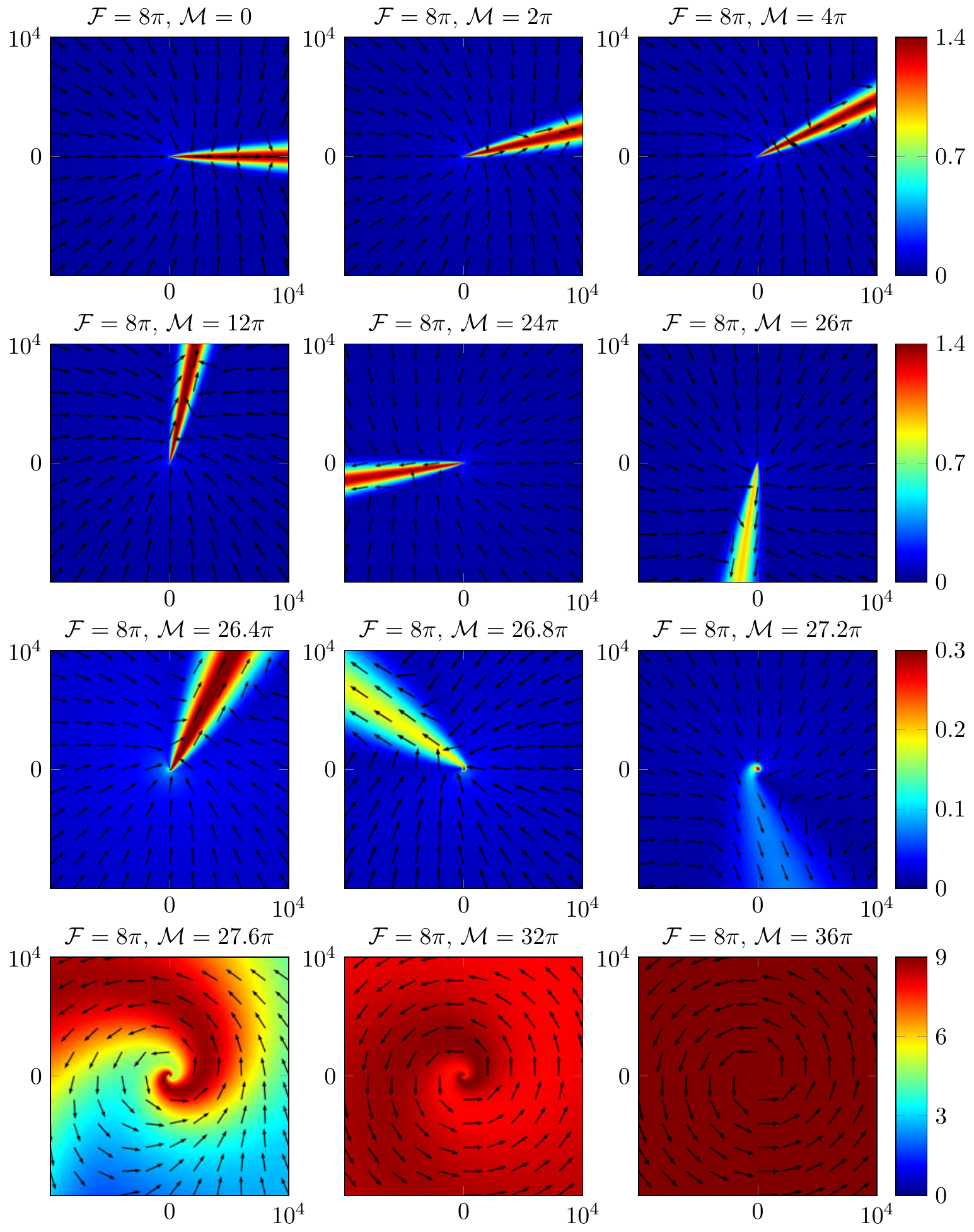}

\caption{\label{fig:sim-Fmu-20}Numerical simulations on the line $\mathcal{F}=8\pi$
of the velocity magnitude multiplied by $r^{1/3}$ for the first three
lines and by $r$ for the last one.}
\end{figure}
\begin{figure}[h]
\includefigure{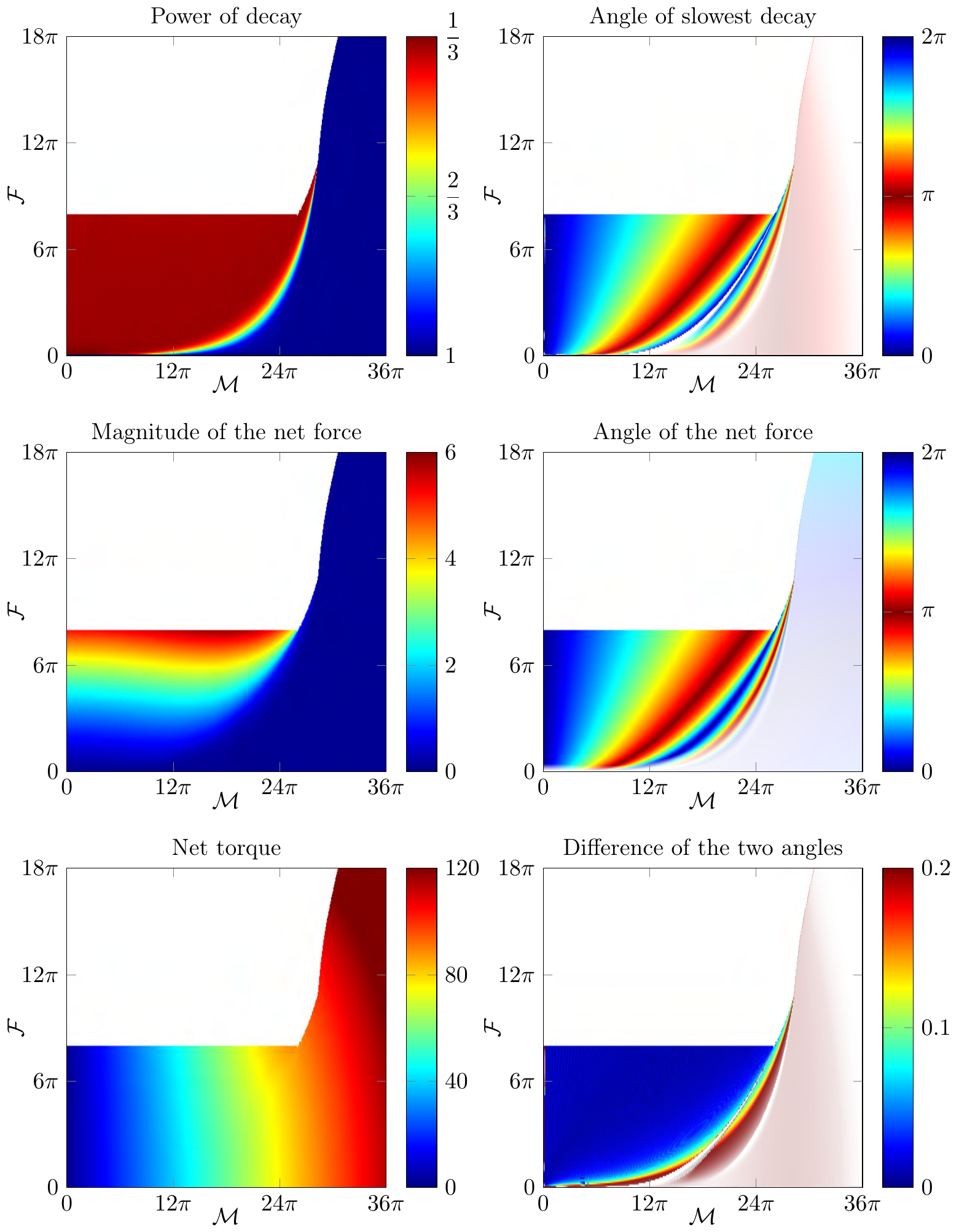}

\caption{\label{fig:plot-Fmu}Main characteristics of the solution for varying
$\mathcal{F}$ and $\mathcal{M}$: (a) the power of decay of the function
$d$; (b) the mean of the function $a$ with its standard deviation
shown with transparency; (c) the magnitude of the net force acting
on the body $B$; (d) the angle of the net force with the magnitude
of the net force in transparency; (e) the net torque acting on the
body; (f) the difference between the angle drawn on (b) and (d).}
\end{figure}
\clearpage{}

\subsection{Zero net force}

Second, we consider the case where $\bC_{0}=\bzero$, for which we
know by symmetry that the net force is zero. By a rotation and a reflection,
we can without generality assume that $\bC_{1}=\left(-\mathcal{A},0,-\mathcal{M}\right)$,
with $\mathcal{A},\mathcal{M}\geq0$. We perform numerical simulations
for $\mathcal{A}\in\left\{ 0,0.4,0.8,\dots,36\right\} \pi$ and $\mathcal{M}\in\left\{ 0,0.4,0.8,\dots,36\right\} \pi$.
Again, for values far from $\mathcal{A}=\mathcal{F}=0$, the nonlinear
solver has difficulties do converge, so we use a parametric solver
to follow the solution. At fixed value of $\mathcal{M}$, we perform
a parametric continuation from $\mathcal{F}=0$ to $\mathcal{F}=36\pi$,
as shown in \figref{plot-parametric}a, or we do the converse as shown
in \figref{plot-parametric}b and surprisingly the results are not
the same.
\begin{figure}[h]
\includefigure{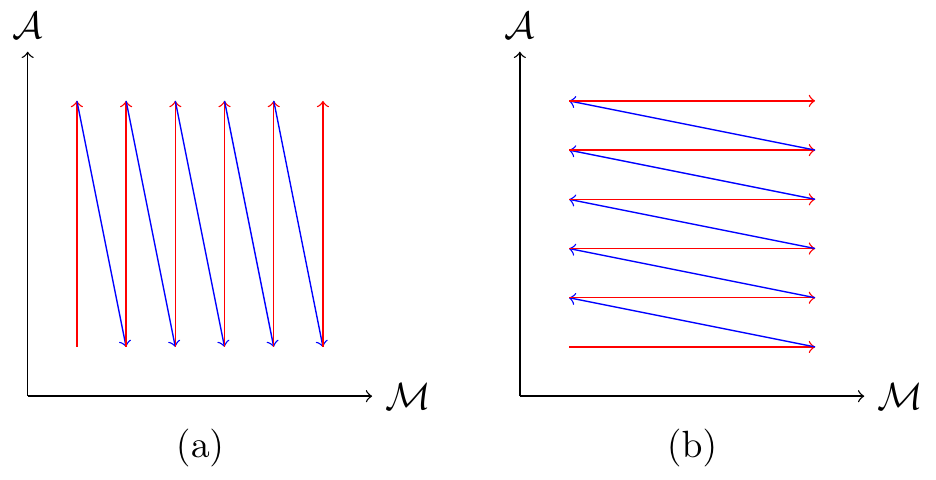}

\caption{\label{fig:plot-parametric}In order to study the dependence of the
solution on two parameters $\mathcal{A}$ and $\mathcal{F}$, we have
two choices: (a) at fixed value of $\mathcal{M}$ we perform a parametric
continuation on $\mathcal{A}$ or (b) at fixed value of $\mathcal{F}$
we use a parametric solver in $\mathcal{M}$.}
\end{figure}

The magnitude of the velocity $\bu$ on the line $\mathcal{A}=8\pi$
for varying $\mathcal{M}$ is shown in \figref{sim-Amu-20}. For such
small values of $\mathcal{A}$, we are at small Reynolds number, so
this is not clear if the computational domain is big enough for seeing
the real asymptotic behavior and not only the Stokes one. Therefore,
we cannot conclude that the velocity decays like $r^{-1}$ or like
$r^{-1/3}$. On the line $\mathcal{A}=18\pi$ (\figref{sim-Amu-45}),
the velocity decay like $r^{-1/3}$ for small values of $\mathcal{M}$
and like $r^{-1}$ for large ones, so the first two lines of the figure,
the velocity magnitude is multiplied by $r^{1/3}$ and on the last
two by $r$. At $\mathcal{M}=0$, we have a double wake characterized
by $\bUF+\bUmF$ for some $\bF=(-F,0)$ depending on $\mathcal{A}$
and this double wake is rotated by an increasing angle in term of
$\mathcal{M}$. Around $\mathcal{M}=8.8\pi$, the wake behavior disappears
and the solution is asymptotic to the harmonic solution $\mu\be_{\theta}/r$
for some $\mu\in\mathbb{R}$. On the last line of \figref{sim-Amu-45}
we represent the norm $r\left|\bu-\mu\be_{\theta}/r\right|$ for the
best $\mu\in\mathbb{R}$. The same analysis is done in \figref{sim-Amu-90}
for $\mathcal{A}=36\pi$. Near $\mathcal{A}=\mathcal{F}=16\pi$, the
solution depends on the way we approach it: either the velocity decays
like $r^{-1/3}$ either like $r^{-1}$. We note that \figref{sim-Amu-45}f
is similar to the spiral solutions found in \citet{Guillod-Generalizedscaleinvariant2015}
with $n=2$.

In the same way, we also analyze the functions \eqref{functions-d-a}.
The power of decay in both cases are respectively shown in \figref{plot-Amu}a
and \figref{plot-Amu}c. In this situation the slowest decay is given
by two angles separated by $\pi$, so we take the mean of the function
$a$ modulo $\pi$, as shown in \figref{plot-Amu}b and \figref{plot-Amu}d.
Surprisingly, the two ways we used the parametric solver do not produce
the same results in a small triangle near $\mathcal{A}=\mathcal{F}=16\pi$.
Especially the power of decay seems to be $r^{-1/3}$ when the value
of $\mathcal{A}$ is increasing and like $r^{-1}$ when the value
of $\mathcal{M}$ is increasing. This strange behavior may either
mean that the solution is not unique or that the precision of the
numerical solver in not good enough to discard one of the two solutions.

\begin{figure}[h]
\includefigure{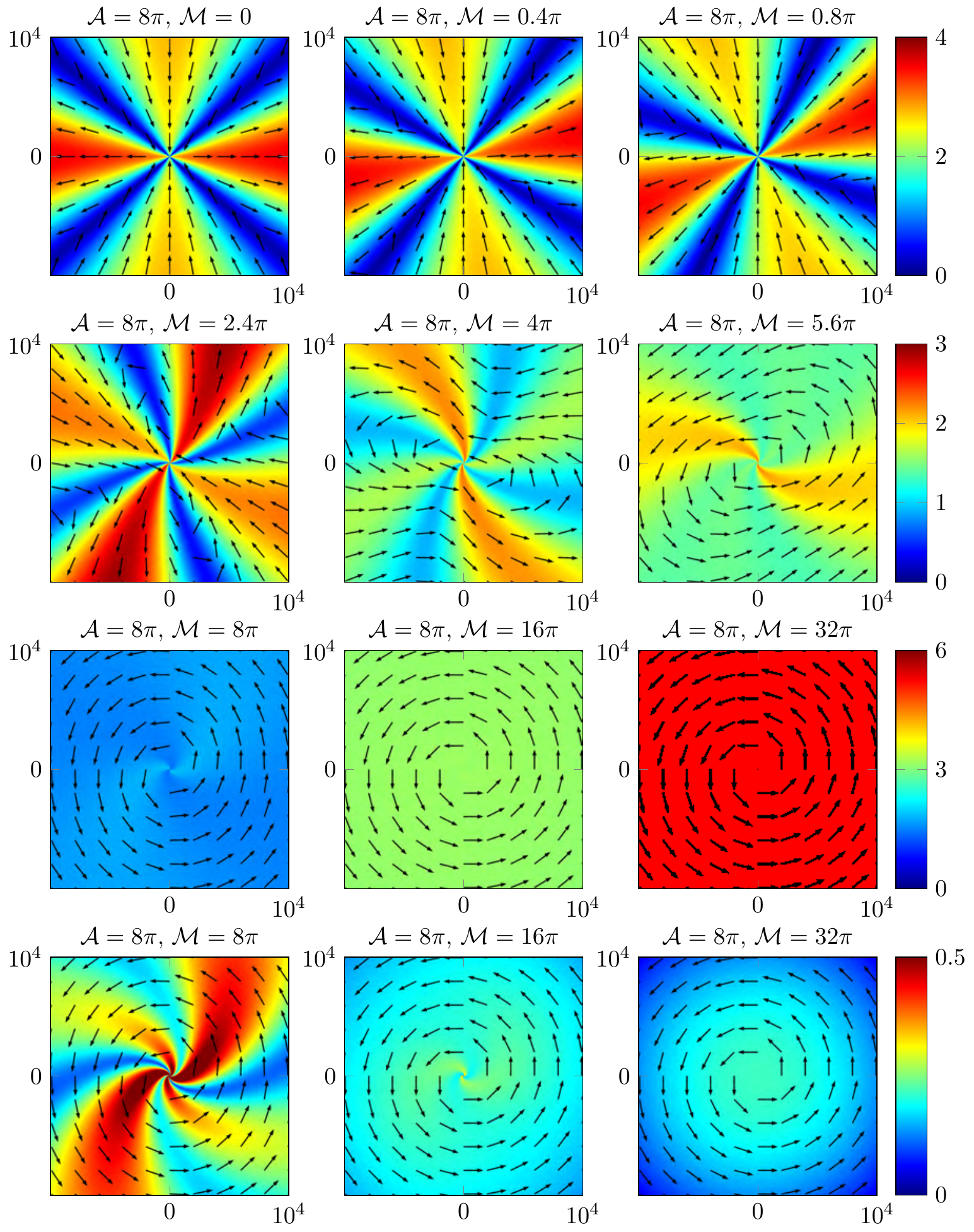}

\caption{\label{fig:sim-Amu-20}Numerical simulations on the line $\mathcal{A}=8\pi$
of the velocity magnitude multiplied by $r$ for the first three lines.
Since $\mathcal{A}$ is small, for small values of $\mathcal{M}$
the velocity behaves like the solution of the Stokes equations except
that the velocity is bigger in the outflow regions than in the inflow
regions. For $\mathcal{M}$ larger than approximately $8\pi$, the
velocity is close to the harmonic solution $\mu\protect\be_{\theta}/r$
for some $\mu\in\mathbb{R}$. In the last line we represent the magnitude
$\left|r\protect\bu-\mu\protect\be_{\theta}\right|$ of the optimal
$\mu$ that minimize the remainder.}
\end{figure}
\begin{figure}[h]
\includefigure{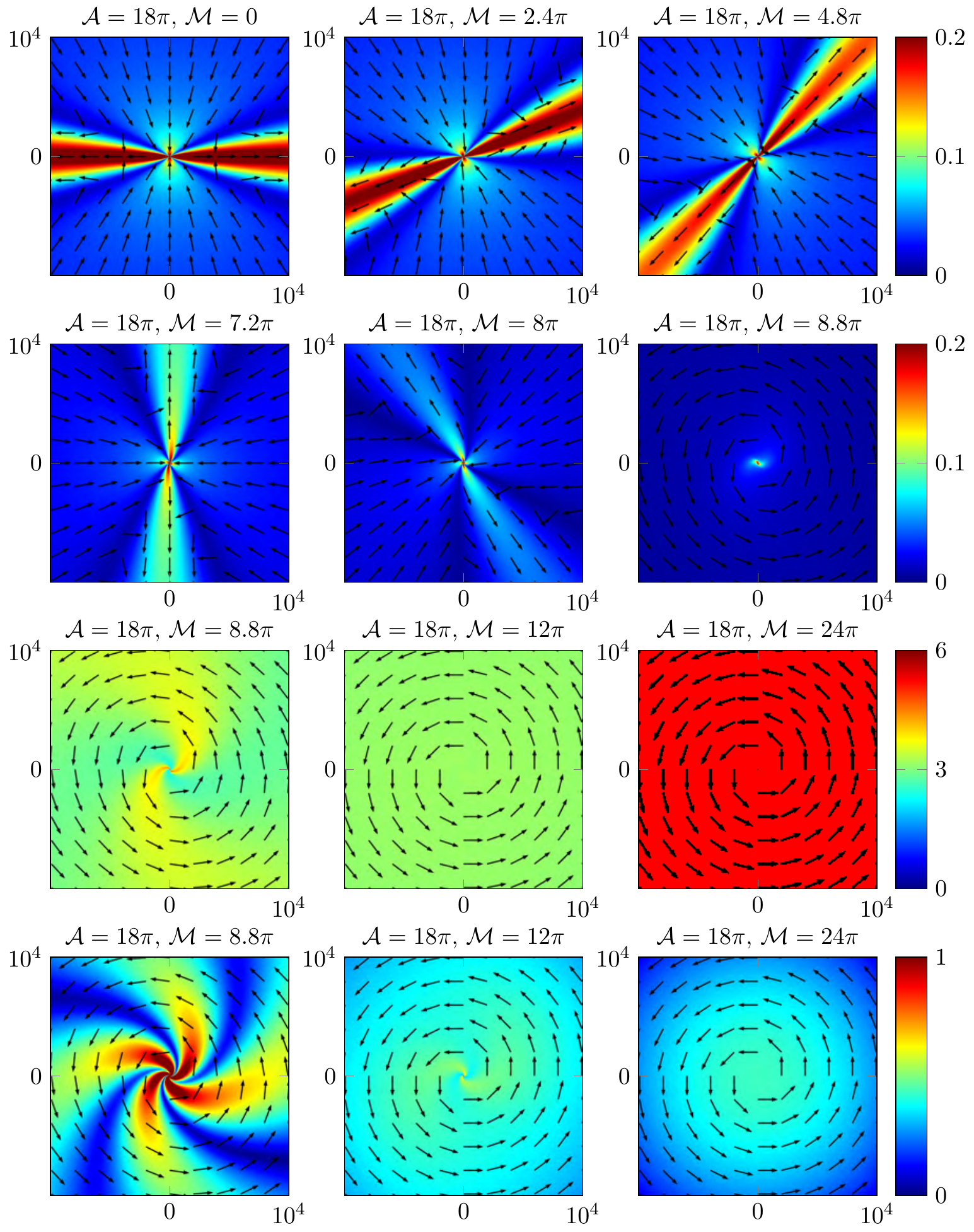}

\caption{\label{fig:sim-Amu-45}Numerical simulations for $\mathcal{A}=18\pi$.
The first two lines represent $r^{1/3}\left|\protect\bu\right|$,
the third one $r\left|\protect\bu\right|$ and the last one $\left|r\protect\bu-\mu\protect\be_{\theta}\right|$
for the best $\mu$. For small $\mathcal{M}$, the velocity is well-modeled
by the solution $\protect\bUF$ of \propref{wake-UF-PF}. As $\mathcal{M}$
increases, the double wake rotates, its magnitude decreases and disappears
around $\mathcal{M}=8.8\pi$. From this value the the velocity is
close to the exact solution $\mu\protect\be_{\theta}/r$ for some
$\mu\in\mathbb{R}$.}
\end{figure}
\begin{figure}[h]
\includefigure{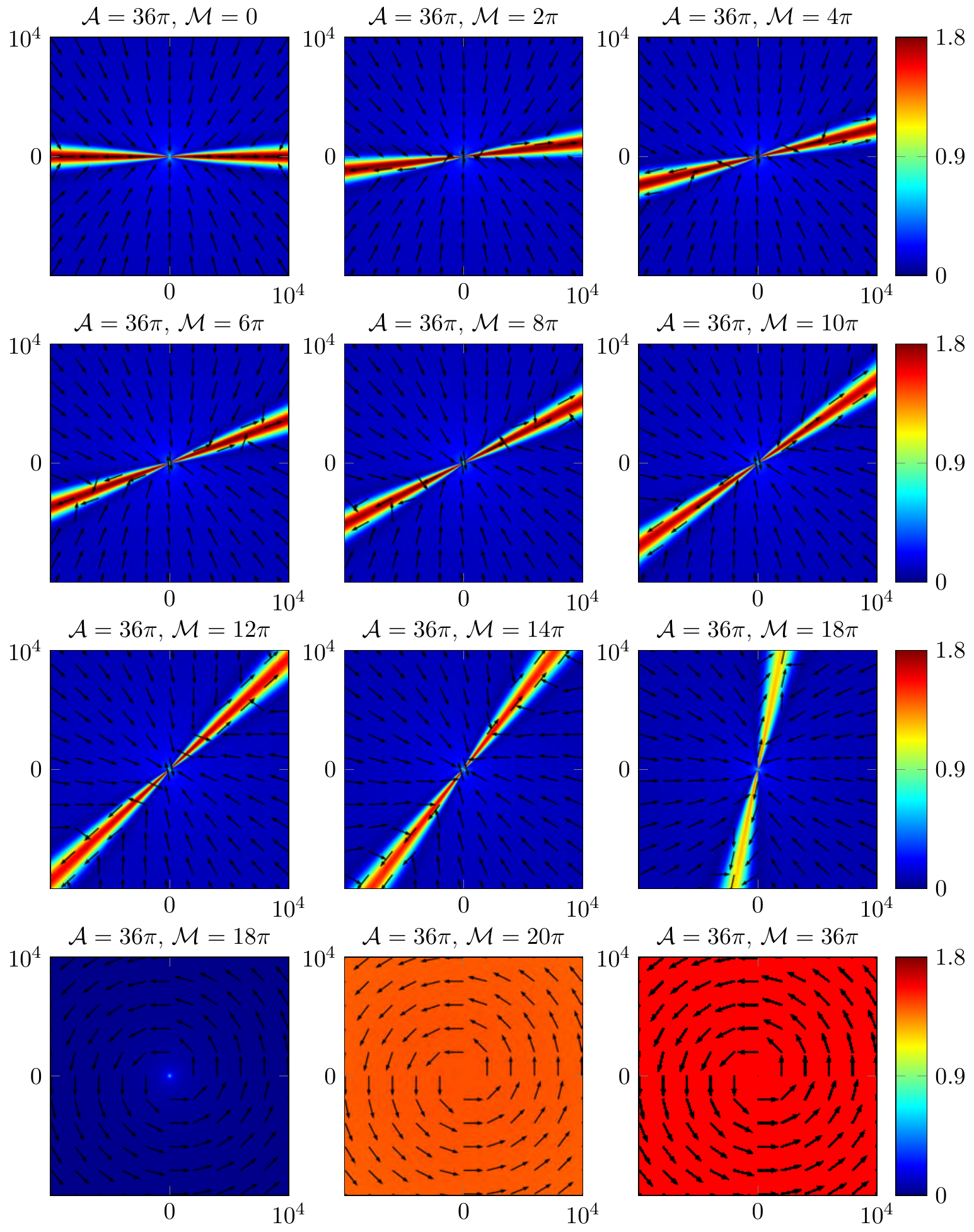}

\caption{\label{fig:sim-Amu-90}Numerical simulations for $\mathcal{A}=36\pi$.
The first three lines represent $r^{1/3}\left|\protect\bu\right|$
and the last one $r\left|\protect\bu\right|$. As $\mathcal{A}$ is
bigger than in \figref{sim-Amu-45}, the opening of the double wake
is more narrow, so it corresponds to $\protect\bUF+\protect\bUmF$
with a bigger value of $\left|\protect\bF\right|$. As $\mathcal{M}$
increases, the magnitude of the double wake is reduced and finally
the velocity decays like $r^{-1}$ for large values of $\mathcal{M}$.}
\end{figure}
\begin{figure}[h]
\includefigure{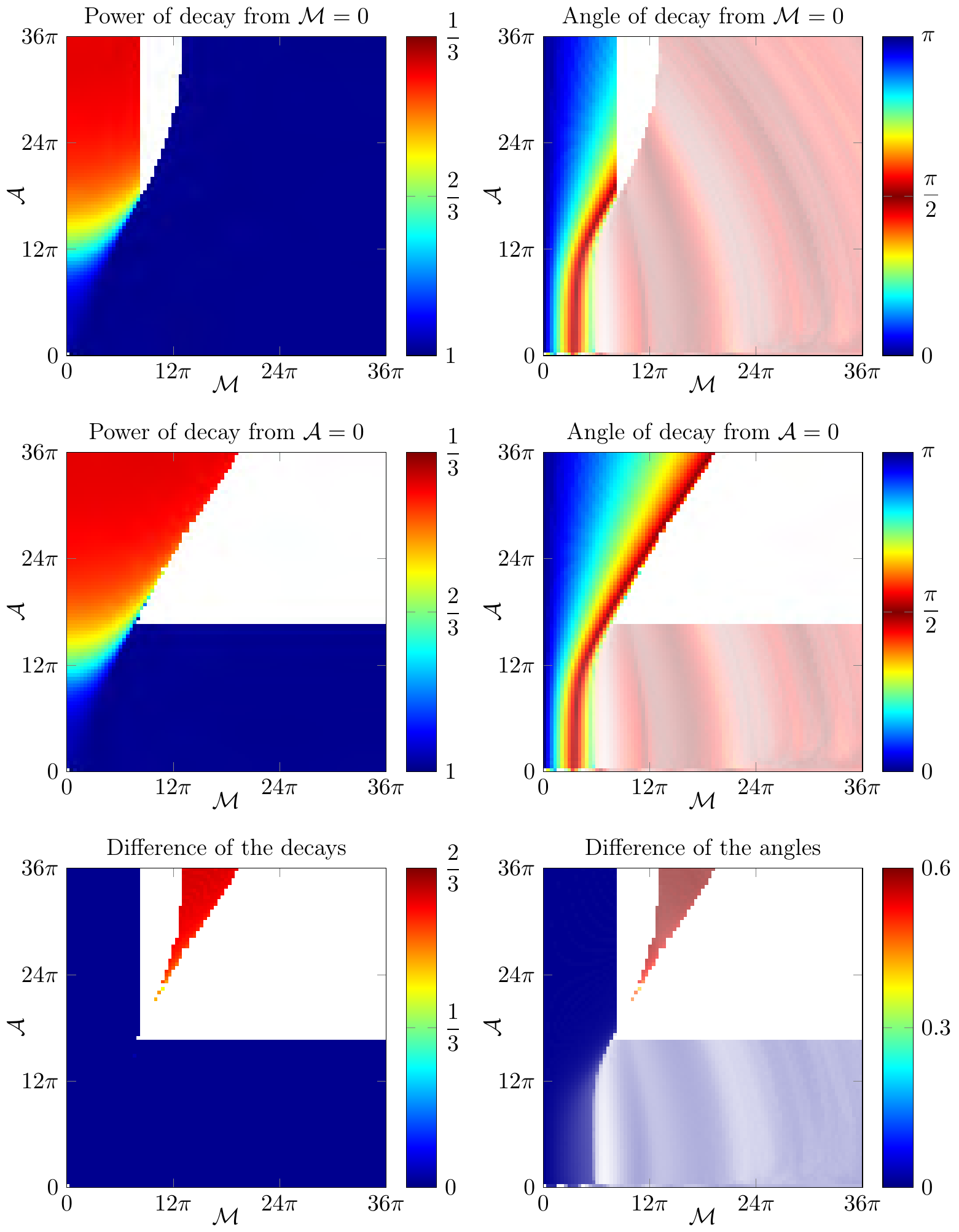}

\caption{\label{fig:plot-Amu}Main characteristics of the numerical solution
for varying $\mathcal{A}$ and $\mathcal{M}$. The power of decay
of the function $d$ when the parametric solver is used at fixed value
of $\mathcal{M}$ or $\mathcal{A}$ is drawn in (a) and (c) respectively,
its difference is (e). The angle of the slowest decay which is the
mean of the function $a$ with its standard deviation shown with transparency
is represented on (b) and (d) respectively for the parametric solver
used at fixed value of $\mathcal{M}$ or $\mathcal{A}$ and the difference
is (f).}
\end{figure}

\clearpage{}

\section{\label{sec:wake-multiple}Numerical simulations with multiple wakes\index{Numerical simulations!multiple wakes}}

\begin{wrapfigure}[19]{r}{70mm}%
\vspace{-10pt}\includefigure{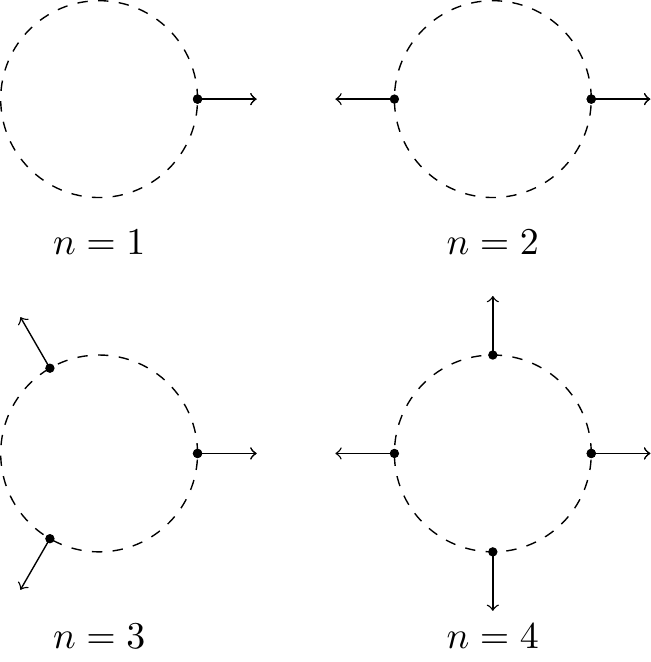}

\caption{\label{fig:delta-n-forces}Representation of the force \eqref{delta-n-f},
which is $n$ approximations of the delta-function uniformly distributed
on the circle on radius five.}
\end{wrapfigure}%
Finally, we examine the possibility of generating more than one or
two wakes. The idea is to take $\bff$ having $n$ approximations
of the delta function distributed on the circle of radius five (see
\figref{delta-n-forces}),
\begin{equation}
\bff(\bx)=-\mathcal{A}\sum_{i=0}^{n-1}\delta_{\varepsilon}(\bx-5\mathbf{R}_{2\pi i/n}\be_{1})\mathbf{R}_{2\pi i/n}\be_{1}\,,\label{eq:delta-n-f}
\end{equation}
where $\mathcal{A}\in\mathbb{R}$ is an amplitude, $\mathbf{R}_{\vartheta}\in\mathrm{SO}(2)$
is the rotation matrix of angle $\vartheta$, and $\delta_{\varepsilon}$
is the following approximation of the $\delta$-function,\index{Delta-function approximation}
\[
\delta_{\varepsilon}(\bx)=\frac{1}{\pi\varepsilon}\e^{-\left|\bx\right|^{2}/\varepsilon}\,.
\]
We perform numerical simulations in a disk $B(\bzero,R)$ of radius
$R=10^{4}$ with open boundary conditions on $\partial B(\bzero,R)$
and $\varepsilon=0.1$, which leads to the results drawn in \figref{branches-u}.
For $n=1$, we recover the straight simple wake studied in details
in \secref{asy-wake}. For $n=2$, we obtain two wakes which are in
opposite directions so that the net is effectively zero. For small
values of $\mathcal{A}$, the solution is very close to the solution
of the Stokes equations on a huge domain, so the magnitude of velocity
is quasi similar along the first and second axes, and decay like $r^{-1}$
on the computational domain. As $\mathcal{A}$ increases, this property
is more and more destroyed with the emergence of the two wakes that
decay like $r^{-1/3}$. For $n=3$, the situation is similar. For
$n=4$, for small values of $\mathcal{A}$, the velocity decays almost
like $r^{-2}$, but as $\mathcal{A}$ increases this situation becomes
unstable, and around $\mathcal{A}=96$, two wakes with an angle of
$\frac{\pi}{4}$ and $\frac{5\pi}{4}$ are created. By symmetry, the
same solution rotated by $\frac{\pi}{2}$ is also a solution, so the
choice between the two possibilities comes from the symmetry breaking
due to the meshing of the domain. As $\mathcal{A}$ increases even
more, the two wakes separate to become four distinct wakes.

Finally, we determine in \figref{branches-decay} the power of decay
of $\left|\bu\right|$ inside the wake on the region $10^{2}\leq r\leq8\times10^{3}$
in which the magnitude of the velocity seems to have a constant power
of decay not influenced by the artificial boundary conditions. For
$n=1$, the power of decay is essentially $r^{-1/3}$ as shown in
\secref{wake-sim}, except for small values of $\mathcal{A}$ for
which the computational domain is too small. For $n=2$, almost the
same situation appears: for small value of $\mathcal{A}$ the solution
is close to the Stokes solution which decays like $r^{-1}$ in a large
domain, so for small value of $\mathcal{A}$ the apparent decay of
the numerical solution is almost $r^{-1}$. For larger $\mathcal{A}$,
the two wakes decay like $r^{-1/3}$ and the velocity fields are almost
fitted by $\bUF+\bUmF$ where $\bF$ depends on $\mathcal{A}$. For
$n=3$, the solution of the Stokes equations decay like $r^{-2}$
and therefore, for small values of $\mathcal{A}$ the power of decay
inside the computational domain is near $r^{-2}$. As $\mathcal{A}$
increases, the three wakes described by some $\bUF$ emerge and decay
almost like $r^{-1/3}$. For $n=4$, there is a regime with two wakes
that break the symmetry before splitting into four wakes. The power
of decay in \figref{branches-u} seems to indicate that at small Reynolds
numbers, only one or two wakes can exist and that an higher number
of wakes is present only at large Reynolds numbers.

\begin{figure}[h]
\includefigure{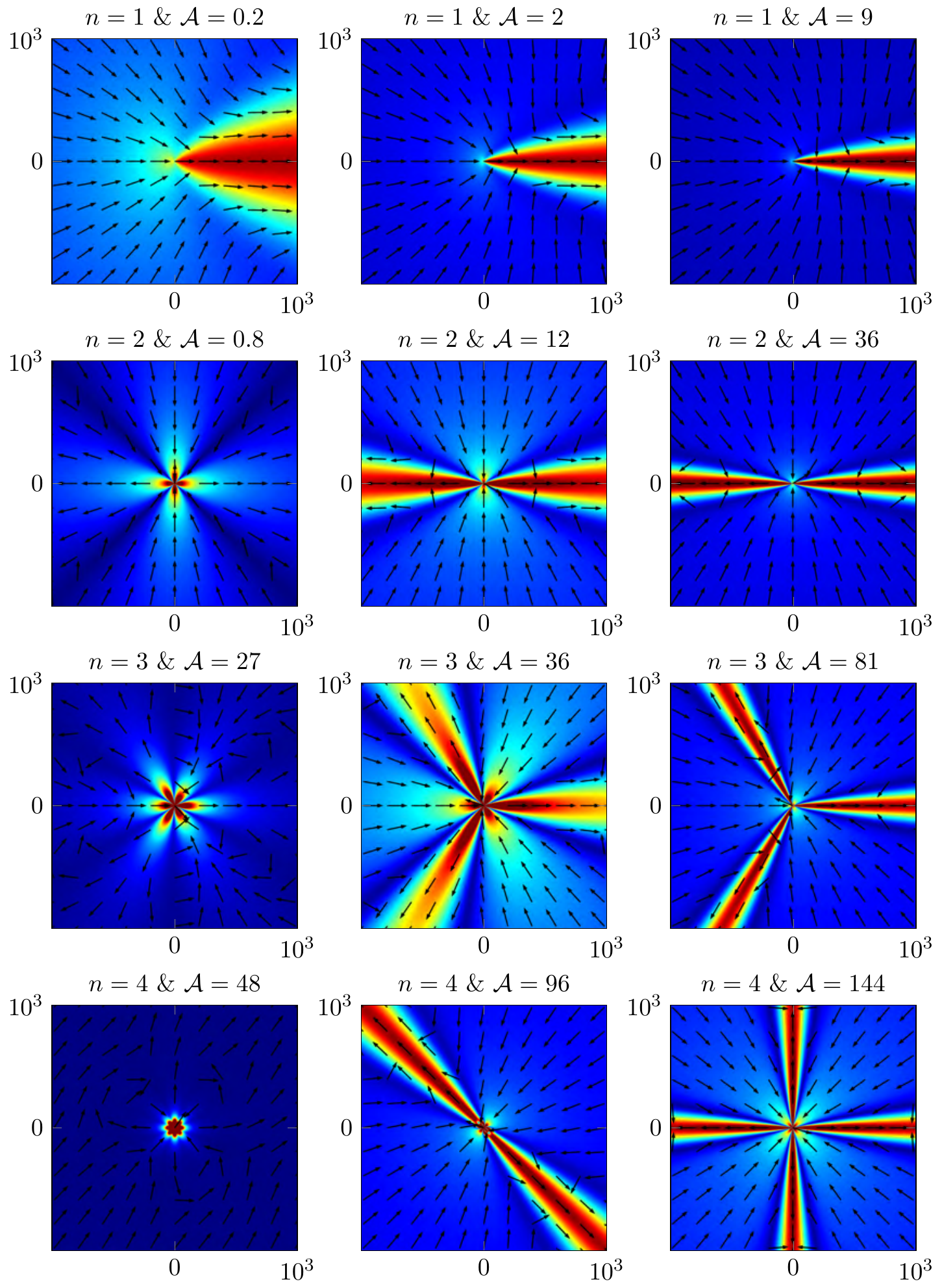}

\caption{\label{fig:branches-u}Magnitude for the velocity field $r^{1/3}\left|\protect\bu\right|$
obtained by numerical simulations with $n$ approximations of the
delta function for the source force \eqref{delta-n-f}. For small
value of the amplitude $\mathcal{A}$, the solution is close to the
solution of the Stokes equations on a large domain, but for large
data, we obtain $n$ wakes. For $n=4$ and $\mathcal{A}=96$, the
numerically found solution breaks the symmetry of the source force.}
\end{figure}

\begin{figure}[h]
\includefigure{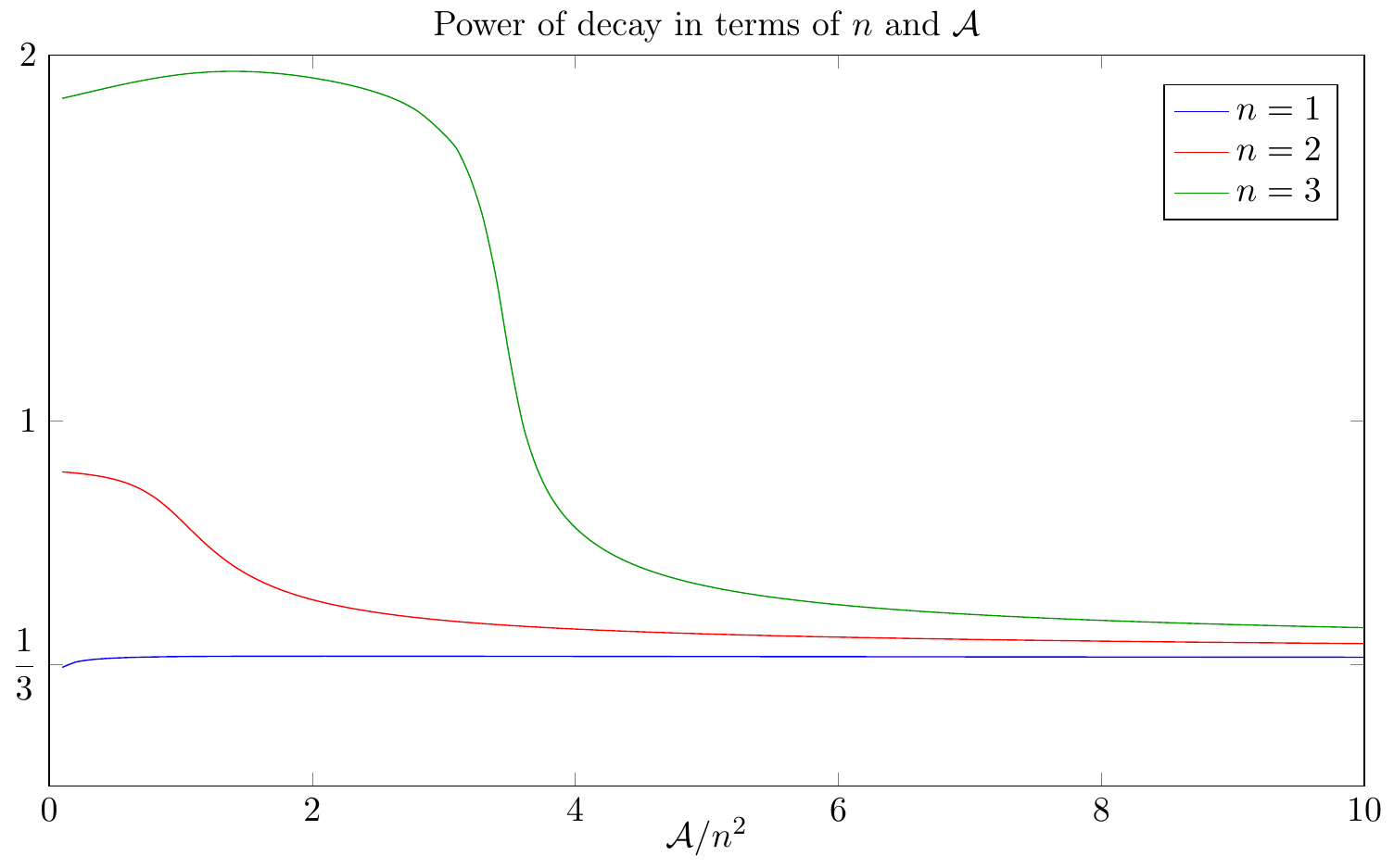}

\caption{\label{fig:branches-decay}Power of decay of the numerical solutions
fitted in the region $10^{2}\leq r\leq8\times10^{3}$. For small values
of $\mathcal{A}$, the velocity decays like the solution of the Stokes
equations in a large region which explains the behavior of the power
of decay near $\mathcal{M}=0$. For large values of $\mathcal{M}$
the velocity behaves like $r^{-1/3}$.}
\end{figure}

\clearpage{}

\section{Conclusions}

If the net force is nonzero, we have a physically motivated conjecture
for the asymptotic behavior of the solution, which is verified numerically.
In this case, the asymptote is given by $\bUF$ which is decaying
like $|\bx|^{-1/3}$ inside a wake in the direction of $\bF$ and
like $|\bx|^{-2/3}$ outside. If the net force vanishes, the velocity
can be asymptotic to the double wake $\bUF+{\boldsymbol{U}\!}_{-\bF}$
for some $\bF\in\mathbb{R}^{2}$ which also have a supercritical decay
like $|\bx|^{-1/3}$ . In another regime, the solution is asymptotic
to the exact harmonic solution $\mu\be_{\theta}/r$, where $\mu$
is a parameter. The previous section seems to indicate that at small
Reynolds number, three wakes or more are not possible. We remark that
these two regimes are clearly not the only ones. By choosing particular
boundary conditions on the disk, we can easily construct an exact
solution that is equal at large distances to spiral solutions found
in \citet{Guillod-Generalizedscaleinvariant2015} for $n=2$ and arbitrary
small $\kappa$.\index{Asymptotic expansion!Navier-Stokes solutions!for bF=bzero
@for $\bF=\bzero$}\index{Navier-Stokes equations!asymptotic expansion!for bF=bzero
@for $\bF=\bzero$} The results concerning the decay of the solutions of the Navier-Stokes
equations and their asymptotic behavior are summarized is the following
table\index{Asymptotic expansion!Navier-Stokes solutions!summary}\index{Navier-Stokes equations!asymptotic expansion!summary}:

\begin{center}
\includefigure{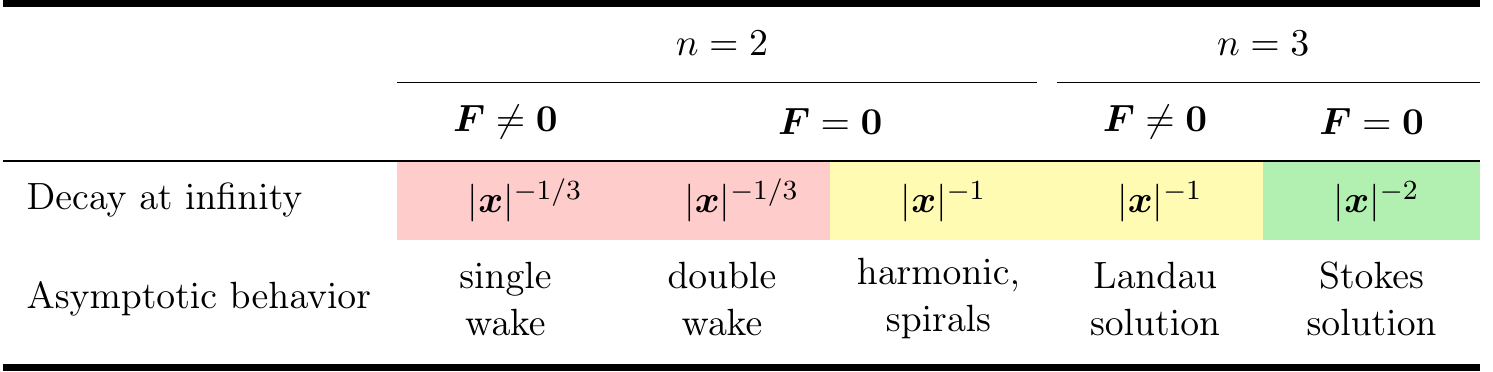}
\par\end{center}

In particular, we see that the nonlinearity of the Navier-Stokes equations
seems to allow the existence of solutions decaying to zero at infinity
even if the net force is nonzero, which removes the Stokes paradox
that is present at the linear level.

\clearpage{}

\bibliographystyle{merlin}
\phantomsection\addcontentsline{toc}{chapter}{\bibname}\bibliography{paper}

\clearpage{}\printindex{}
\end{document}